\documentclass[a4paper,10pt]{article}
\usepackage{amssymb}
\usepackage{latexsym}
\usepackage{amsmath}
\usepackage{amsthm}
\usepackage{amsfonts}
\usepackage{amssymb}
\usepackage{graphicx}

\newtheorem{Th}{Theorem}
\newtheorem{Prop}{Proposition}
\newtheorem{Co}{Corollary}
\newtheorem{Lm}{Lemma}

\newtheorem{Dfi}{Definition}
\newtheorem{Rm}{Remark}

\newcommand{\be}{\begin{equation}}
\newcommand{\ee}{\end{equation}}
\newcommand{\R}{\mathbb{R}}
\newcommand{\N}{\mathbb{N}}
\newcommand{\C}{\mathbb{C}}

\newcommand\res{\mathop{\hbox{\vrule height 7pt width .5pt depth 0pt
\vrule height .5pt width 6pt depth 0pt}}\nolimits}

\newcommand{\reset}{\setcounter{equation}{0}\setcounter{Th}{0}\setcounter{Prop}{0}\setcounter{Co}{0}
\setcounter{Lm}{0}\setcounter{Rm}{0}\setcounter{Con}{0}}

\def\ti{\tilde}
\def\lf{\left}
\def\rg{\right}

\def\al{\alpha}
\def\la{\lambda}
\def\e{\varepsilon}
\def\ep{\varepsilon}
\def\ds{\displaystyle}
\def\ov{\overline}

\def\om{\omega}
\def\p{\partial}
\def\bn{\vec{n}}
\def\bh{\mathbf{h}}
\def\bbe{\vec{e}}
\def\bbf{\vec{f}}
\def\bH{\vec{H}}
\def\bu{\vec{u}}
\def\bv{\vec{v}}

\def\bA{\vec{A}}

\def\bE{\vec{E}}
\def\bF{\vec{F}}
\def\bL{\vec{L}}
\def\bR{\vec{R}}
\def\bN{\vec{N}}

\def\bU{\vec{U}}
\def\bV{\vec{V}}
\def\bX{\vec{X}}
\def\bY{\vec{Y}}
\def\bw{\vec{w}}
\def\bW{\vec{W}}

\def\bP{\vec{\Phi}}

\def\bul{\bullet}
\def\res{\mathop{\hbox{\vrule height 7pt width .5pt 
depth 0pt\vrule height .5pt width 6pt depth 0pt}}\nolimits}
\newcommand{\D}{D}
\newcommand{\LW}{L}
\newcommand{\Riem}{Riem}
\newcommand{\Area}{Area}
\newcommand{\Sp}{\mathbb{S}}
\DeclareMathOperator{\diam}{diam}

\DeclareMathOperator{\eucl}{eucl}

\DeclareMathOperator{\supp}{supp}

\begin{document}
\title{Willmore Spheres in Compact Riemannian Manifolds}
\author{Andrea Mondino\footnote{Scuola Normale Superiore, Piazza dei Cavalieri 7, 56126 Pisa, Italy, E-mail address: andrea.mondino@sns.it. 
A. Mondino is supported by a Post Doctoral Fellowship in the ERC grant ÓGeometric Measure Theory in non Euclidean 
SpacesÓ directed by Prof. Ambrosio.} , Tristan Rivi\`ere\footnote{Department of Mathematics, ETH Zentrum,
CH-8093 Z\"urich, Switzerland.}}
\date{ }
\maketitle

{\bf Abstract :}{\it The paper is devoted to the variational analysis of the Willmore, and other $L^2$ curvature functionals, for  2-d surfaces immersed in a compact riemannian $3\leq m$-manifold $(M^m,h)$; the double goal of the paper is on one hand to give the right setting for doing the calculus of variations (including min max methods) of such functionals for immersions into manifolds and on the other hand to prove existence of possibly branched Willmore spheres under curvature or topological conditions. For this purpose, using the integrability by compensation, we develop the regularity theory for the critical points of such functionals; a crucial step consists in writing the Euler-Lagrange equation (which is a system), first in a conservative form making sense for weak $W^{1,\infty}\cap W^{2,2}$ immersions
, then as a system of conservation laws. Exploiting this new form of the equations we are able on one hand to prove full regularity of weak solutions to the Willmore equation in any codimension, on the other hand to prove a rigidity theorem concerning the relation between CMC and Willmore spheres. One of the main achievements of the paper is that for every non null 2-homotopy class $0\neq \gamma \in \pi_2(M^m)$ we produce a canonical representative given by a Lipschitz map from the 2-sphere into $M^m$ realizing a  connected family of conformal smooth (possibly branched) area constrained Willmore spheres (as explained in the introduction, this comes as a natural extension of the  minimal immersed spheres in homotopy class constructed by Sacks and Uhlembeck in \cite{SaU}
in situations when they do not exist); moreover  for every ${\cal A}>0$ we minimize the Willmore functional among  connected families of weak, possibly branched, immersions of $\Sp^2$ having total area ${\cal A}$ and we prove full regularity for the minimizer. Finally, under a mild curvature condition on $(M^m,h)$,  we minimize $\int(|\mathbb I|^2+1)$, where ${\mathbb I}$ is the second fundamental form, among weak possibly branched immersions of $\Sp^2$ and we prove the regularity of the minimizer.     
}
\medskip

\noindent{\bf Math. Class.} 30C70, 58E15, 58E30, 49Q10, 53A30, 35R01, 35J35, 35J48, 35J50.

\section{Introduction}\label{s:in}

Throughout the paper $(M^m, h)$ will be a compact connected $m$-dimensional  Riemannian  manifold. For a smooth immersion $\bP$ of a compact 2-dimensional surface $\Sigma$ into $(M^m,h)$ recall the definition of the \emph{Willmore functional}
\be\label{def:W}
W(\bP):=\int_{\Sigma} |\bH|^2 dvol_g,
\ee
where the mean curvature $\bH$ is half the trace of the second fundamental form ${\mathbb I}$ and $vol_g$ is the volume form associated to the pullback metric $g:=\bP^*h$, of the \emph{energy functional} $F$ 
\be\label{def:F}
F(\bP):=\frac {1}{2} \int_{\Sigma} |{\mathbb I}|^2 dvol_g,
\ee
and of the \emph{conformal Willmore functional} $W_{conf}$
\be\label{def:Wconf}
W_{conf}(\bP):=\int_{\Sigma} \left( |\bH|^2+\bar{K}(T\bP)\right) dvol_g,
\ee
where $\bar{K}(T\bP)$ is the sectional curvature of the ambient manifold $(M^m,h)$ computed on the tangent space of $\bP(\Sigma)$; recall moreover that $W_{conf}$ is conformally invariant (i.e. is invariant under conformal changes of the ambient metric $h$ ), see \cite{Wei}.

Observe that, by Gauss Bonnet Theorem, for immersions in the Euclidean space $\R^m$ the three functionals $W,W_{conf}$ and $F$ differ just by a topological constant, so they are equivalent from the variational point of view. This is not the case for immersions in Riemannian manifolds as it will be clear soon.

The first goal of the present paper is to develop the \emph{analysis} of the Willmore and the others $L^2$ curvature functionals in  Riemannian manifolds of any dimension.  Indeed, as for immersions in the eulidean space, there is the following functional analysis paradox: though the Willmore functional $W$ defined in \eqref{def:W} has perfect meaning for  $W^{2,2}\cap W^{1,\infty}$ weak immersions, the classical form of its Euler Lagrange equation (derived in \cite{Wei}) does not make sense for such weak objects (which are the natural ones for doing the analysis of the Willmore functional, as it will explained below; exactly as the Sobolev spaces are the natural framework to studying PDEs). Indeed it requires $L^3$ integrability of second derivatives being 
\be \label{eq:W'Weiner}
\Delta_\perp\bH+\ti{A}(\vec{H})-2 |\bH|^2\ \bH-\ti{R}(\bH)=0,
\ee
where $\ti{R}:T_{\vec{\Phi}(x)}M \to T_{\vec{\Phi}(x)}M$ is the curvature endomorphism defined by
\be 
\label{eq:defR}
\forall \vec{X} \in T_{\vec{\Phi}(x)}M \qquad \tilde{R}(\vec{X}):=-\pi_{\vec{n}} \left[\sum_{i=1} ^{2} \Riem^h(\vec{X},\vec{e}_i)\vec{e}_i \right] \quad,
\ee 
and $\tilde{A}:T_{\vec{\Phi}(x)}M \to T_{\vec{\Phi}(x)}M$  is defined as 
\be
\label{VI.44a}
\forall \vec{X} \in T_{\vec{\Phi}(x)}M \qquad \tilde{A}(\vec{X}):=\sum_{i,j=1}^2 \vec{\mathbb I}(\vec{e}_i,\vec{e}_j) <\vec{\mathbb I}(\vec{e}_i,\vec{e}_j), \vec{X}>\quad,
\ee
where $\pi_{\bn}$  is the projection onto the normal space of $\bP$ and $(\vec{e}_1,\vec{e}_2)$ is an orthonormal basis of $T\Sigma$ for the induced metric $g:=\vec{\Phi}^*h$. The same problem appears in the other $L^2$ curvature functionals since the difference in the Euler Lagrange equations is given just by lower order terms.
\medskip

A first achievement of the present work is to rewrite the Euler Lagrange equation in a conservative form which makes sense for such weak immersions. In order to be more accessible,  before we perform  the computations and we present the equations in the codimension one case in Section \ref{Sec:ConsWill3D}, then we pass to the more delicate higher codimensional case.   

For this purpose, up to a reparametrization and working on a  parametrizing disc $D^2$, we can assume that the immersion $\bP$ is conformal so that it makes sense to consider the standard complex structure of  the disc $D^2$; exploiting the complex notation is very convenient and simplifies also the initial presentation given by the second author in \cite{Riv1} for immersions in the euclidean space. The main result of Section \ref{Sec:ConsWillHCD} is the following theorem. Before stating it let us define $R^\perp_{\vec{\Phi}}(T\vec{\Phi})$ as
\be \label{Rperp}
R^\perp_{\vec{\Phi}}(T\vec{\Phi}):= \left( \pi_T\left[ \Riem^h(\vec{e}_1,\vec{e}_2)\vec{H} \right] \right)^\perp
\ee
where $(\vec{e}_1,\vec{e}_2)$ is a positive orthonormal basis of $TD^2$ for the induced metric $g:=\vec{\Phi}^*h$, $\pi_T:T_{\bP}M \to \bP_*(TD^2)$ is the tangential projection and $.^\perp$ denotes the rotation of an angle $\frac \pi 2 $ in $\bP_*(T D^2)$ in the direction from $\vec{e}_1$ towards $\vec{e}_2$, intrinsecally it can be written as $\bX^\perp= (\bP_*)\circ *_g \circ (\bP_*)^{-1}(\bX)$ for any $\bX \in \bP_*(TD^2)$ where $*_g$ is the Hodge duality operator on $(TD^2,g)$.
We shall also denote by $\pi_{\vec{n}}$ the orthogonal projection $\pi_{\vec{n}}:T_{\bP}M \to (\bP_*(TD^2))^\perp$ from the tangent space to $M^m$ onto the normal space
to $\vec{\Phi}(D^2)$.

\medskip

We shall denote by $D$ the Levi-Civita connection of $(M^m,h)$ and by an abuse of notations we also denote by $D$ the associated covariant exterior derivative. We also denote
by $D_g$ the pull back by $\vec{\Phi}$ of $D$ which is a connection - respectively a covariant exterior derivative - of the pull-back bundle $\vec{\Phi}^{-1}TM^m$.
$D^{*_g}_g$ is the adjoint of the covariant derivative $D^g$ for the induced metric $g=\vec{\Phi}^\ast h$. $\star_h$ is the Hodge operator associated to $h$ 
on multi-vectors of $M^m$ from $\wedge^pM^m$ into $\wedge^{m-p}M^m$. 
All these objects are defined in section \ref{Sec:ConsWillHCD} in \eqref{def:Dg}, \eqref{def:*g}, \eqref{def:D*}, \eqref{def:starh}, \eqref{def:wedgeM} and \eqref{def:pin}.

\begin{Th}\label{thm:ConsGen}
Let $\vec{\Phi}$ be a smooth immersion of the two dimensional disc $D^2$ into an $m$-dimensional Riemannian manifold $(M^m,h)$, then the following identity holds
\be\label{eq:ConsWillGen}
\frac{1}{2} D^{*_g}_g \left[D_g \bH -3 \pi_{\bn} (D_g \bH) + \star_h\left((*_g D_g \bn) \wedge_M \bH \right)  \right] =\Delta_\perp\vec{H}+\ti{A}(\vec{H})-2|\vec{H}|^2\ \vec{H}-R^\perp_{\vec{\Phi}}(T\vec{\Phi})
\ee
where $\Delta_\perp$ is the negative covariant laplacian on the normal bundle to the immersion, $\tilde{A}$ is the linear map given in (\ref{VI.44a}), $R^\perp$ is defined in \eqref{Rperp}.
\hfill $\Box$\end{Th}

Notice that though the right hand side does not make sense for $W^{1,\infty}\cap W^{2,2}$ weak immersions, the left hand side does. Therefore a straightforward but important consequence of Theorem~\ref{thm:ConsGen} is the following conservative form of Willmore surfaces equation making sense for $W^{1,\infty}\cap W^{2,2}$ weak immersions.
\begin{Co} \label{co:ConsWillGen}  
A smooth immersion $\vec{\Phi}$ of a 2-dimensional disc $D^2$ in $(M^m,h)$ is Willmore if and only if 
\be
\label{eq:ConsWGen}
\frac{1}{2} D^{*_g}_g \left[D_g \bH -3 \pi_{\bn} (D_g \bH) + \star_h\left((*_g D_g \bn) \wedge_M \bH \right)  \right] =\ti{R}(\bH)-R^\perp_{\vec{\Phi}}(T\vec{\Phi}),
\ee
where $\ti{R}$ and $R^\perp$ are the curvature endomorphisms defined respectively in \eqref{eq:defR} and \eqref{Rperp}.\hfill $\Box$
\end{Co}

\begin{Rm}
The Euler Lagrange equations of the other $L^2$ curvature functionals are computed in Section \ref{Sec:ConsWillHCD} and differ just by terms completely analogous to the right hand side terms of \eqref{eq:ConsWGen}. \hfill $\Box$
\end{Rm}

 Another important corollary is the conservative form of the \emph{constraint-conformal} Willmore equation. Let $\Sigma$ be a smooth closed surface and $\bP:\Sigma \hookrightarrow M$ be a smooth immersion; the pullback metric $g:=\bP^*h$ induces a complex structure $J$ on $\Sigma$, and in the associated conformal class there exists a unique constant curvature metric $c_0$ with total area $1$ (see \cite{Jo}); notice that by construction $\bP:(\Sigma,c_0) \hookrightarrow M $ is a conformal immersion.   Recall that the smooth immersion $\vec{\Phi}$ of $(\Sigma,c_0)$ is said to be \emph{constrained-conformal  Willmore} if and only if it is a critical point of the Willmore functional under the constraint that the conformal class is fixed. 
Before writing the conservative form of the Willmore functional under constraint on the conformal class let us introduce some notation. Call $Q(J)$ the space of holomorphic quadratic differentials on $(\Sigma,J)$ and let $q \in Q(J)$ written in local complex coordinates as $q=f(z) dz \otimes dz$; let $\bH_0$ be the Weingarten map is given in local coordinates, for a conformal immersion with conformal factor $\lambda=\log(|\p_x\bP|)$, by
\be\label{def:vecH0}
\bH_0:=\frac{1}{2}\, e^{-2\lambda}\pi_{\vec{n}}\lf(\p^2_{x^2}\vec{\Phi}-\p^2_{y^2}\vec{\Phi} -2i\ \p^2_{xy}\vec{\Phi}\rg)= \frac{1}{2}\lf[\vec{\mathbb I}(\vec{e}_1,\vec{e}_1)-\vec{\mathbb I}(\vec{e}_2,\vec{e}_2)-2\,i\, \vec{\mathbb I}(\vec{e}_1,\vec{e}_2)\rg]\quad,
\ee 
where $\bn_{\bP}$ is the normal space to $\bP$ and $(\bbe_1,\bbe_2)=e^{-\lambda}(\p_x \bP,\p_y \bP )$ is a positively oriented orthonormal frame of $T\bP$; recall also the definition of the Weingarten operator $\vec{h}_0$  given locally by
\be\label{def:vech0}
\vec{h}_0:=2\ \pi_{\vec{n}}(\p^2_{z^2}\vec{\Phi})\ dz\otimes dz=e^{2\la}\,\vec{H}_0\ dz\otimes dz
\ee
We introduce on the space $\wedge^{1-0}D^2\otimes\wedge^{1-0} D^2$ of $1-0\otimes1-0$ form on $D^2$ the following hermitian product\footnote{ This hermitian product integrated on $D^2$ is the {\it Weil Peterson product}.} depending on the conformal immersion $\vec{\Phi}$
\be
\label{def:WPpunct}.
(\psi_1\ dz\otimes dz,\psi_2\ dz\otimes dz)_{WP}:= e^{-4\la}\ {\psi_1(z)}\ \ov{\psi_2(z)}
\ee
where $e^\la:=|\p_{x_1}\vec{\Phi}|=|\p_{x_2}\vec{\Phi}|$. We observe that for a conformal change of coordinate $w(z)$ (i.e. $w$ is holomorphic in $z$) and for $\psi_i'$ satisfying
\[
\psi_i'\circ w\ dw\otimes dw=\psi_i\ dz\otimes dz
\]
one has, using the conformal immersion $\vec{\Phi}\circ w$ in the l.h.s.
\[
(\psi_1'\ dw\otimes dw,\psi_2'\ dw\otimes dw)_{WP}=(\psi_1\ dz\otimes dz,\psi_2\ dz\otimes dz)_{WP} 
\]
for more informations about the Weyl Peterson  product see \cite{Jo}, \cite{Riv2}, \cite{Riv4}. Now we can write the constraint-conformal Willmore equation in conservative form.

\begin{Co} \label{co:ConsConfW}  
Let $\vec{\Phi}:\Sigma \hookrightarrow M$ be a  smooth immersion into the $m\geq 3$-dimensional Riemannian manifold $(M^m,h)$ and call $c$ the conformal structure associated to $g=\bP^*h$.  Then  $\bP$ is a constrained-conformal Willmore immersion if and only if there exists an holomorphic quadratic differential $q \in Q(c)$  such that

\be
\label{eq:ConsWconf}
\frac{1}{2} D^{*_g}_g \left[D_g \bH -3 \pi_{\bn} (D_g \bH) + \star_h\left((*_g D_g \bn) \wedge_M \bH \right)  \right]= \Im(q,\vec{h}_0)_{WP}+\ti{R}(\bH)-R^\perp_{\vec{\Phi}}(T\vec{\Phi})\quad.
\ee
\hfill $\Box$
\end{Co}

Observe that, in local complex coordinates, $\Im[(q,\vec{h}_0)_{WP}]=e^{-2\lambda} \Im[f(z) \overline{\vec{H}}_0]$.

Notice that also the constraint-conformal equations of the other $L^2$ curvature functional differ just by terms completely analogous to the right hand side terms of \eqref{eq:ConsWGen}. 
\medskip

\medskip 
Exploiting the conservative form just showed, in Section \ref{Sec:SystY} we prove that the constraint-conformal Willmore equation is equivalent to a system of conservation laws (see Theorem \ref{th:ConfWill}) and in Section \ref{Sec:Regularity} we prove that weak solutions to this system of conservation laws are smooth.
For proving the regularity it is crucial to construct from the system of conservation laws some potentials $\vec{R}$ and $S$ which satisfy a critical  Wente type elliptic system (see the system \eqref{eq:SystPRS}). Using integrability by compensation we gain some regularity on $\vec{R}$ and $S$ which bootstrapped, after some work, gives the smoothness of weak solutions to the constraint-conformal Willmore equation. Therefore we are able to prove the following full regularity theorem for weak solutions to the constraint-conformal Willmore equation.    

\begin{Th}\label{Th:Regularity}{\bf [Regularity of weak constraint-conformal Willmore immersions.]}
Let $\vec{\Phi}$ be a $W^{1,\infty}$ conformal immersion of the disc $\D^2$ taking values into a sufficiently small open subset of the Riemannian manifold $(M,h)$, with second fundamental form in $L^2(D^2)$ and conformal factor $\lambda:=\log{|\p_{x^1} \bP|}\in L^\infty(D^2)$. If $\vec{\Phi}$ is a  constrained-conformal Willmore immersion then $\bP$ is  $C^\infty$.\hfill $\Box$
\end{Th}

\begin{Rm}
As the reader will see, the proof of the regularity is not just a straighforward adaptation of the Euclidean one. Indeed in the euclidean case $\vec{R}$ and $S$ were real valued and their existence was ensured by a direct application of Poicar\'e Lemma. Here the curvature terms make the situation more delicate. Indeed $\vec{R}$ and $S$, which now are complex valued, are  constructed  using the $D_z$ and $\p_z$ operators (see Lemma \ref{Lm:eqPRS}), and their construction makes use of singular integrals and Fourier analysis (see the Appendix). Notice that, in case of null curvature, the imaginary parts of $\vec{R}$ and $S$ vanish and the two, a priori different, constructions coincide. Therefore our construction is canonical and has a geometric, beside analytic, meaning.   \hfill $\Box$
\end{Rm}  

\begin{Rm}
The regularity issues regarding minimizers of $L^2$ curvature functionals in 3-dimensional riemannian manifolds have been studied also in \cite{KMS} using techniques from \cite{SiL}. Beside the fact that here we deal with  higher codimensions, the real advantage of this new approach is that  it permits to infer that \emph{any weak solution} to the equation is smooth, while in the former the regularity crucially used the minimality property. Therefore our new approach is more flexible and it is suitable for studying existence of more general critical points of \emph{saddle} type.\hfill $\Box$
\end{Rm}

\begin{Rm}
Since the difference between the Willmore equation and the Euler Lagrange equations of the other $L^2$ curvature functionals $F$ and $W_{conf}$ (also under area or conformal type constraint) is made of  subcritical terms, the Regularity Theorem  \ref{Th:Regularity} applies to them as well.\hfill $\Box$
\end{Rm}

\medskip

Another  application of the conservative form of the equation is the following. Recall that an immersion is called \emph{conformal Willmore} if it is a critical point of the conformal Willmore functional $W_{conf}$ defined in \eqref{def:Wconf}, and is called \emph{constraint-conformal conformal Willmore} if it is a critical point of $W_{conf}$ under the constraint of fixed conformal class. Notice that, since by the Uniformization Theorem there is just one smooth conformal class on $\Sp^2$, the two notions coincide for smooth immersions of $\Sp^2$.
Recall also that a smooth immersion $\bP:\Sigma \hookrightarrow M^m$ of the surface $\Sigma$ has \emph{parallel mean curvature} if the normal projection of the covariant derivative of the mean curvature $\bH$ with respect to tangent vectors to $\bP$ is null: 
\be\label{def:ParMC}
\pi_{\bn}(D \bH)=0.
\ee 
Observe that in codimension one a surface  has  parallel mean curvature if and only if it has constant mean curvature, i.e. it is a CMC surface. 

In Section \ref{Sec:PMC-CCCW}, we prove the following Proposition (the analogous proposition for immersions in the Euclidean space appears in \cite{Riv4}) which ensures abundance of constraint-conformal conformal Willmore surfaces in space forms. Since, as explained above, the conformal constraint for smooth immersions of a 2-sphere is trivial, the proposition ensures also  abundance of Conformal Willmore spheres in space forms. 

\begin{Prop}\label{Prop:PMC<CCCW}
Let $(M^m,h)$ be an $m$-dimensional Riemannian manifold of constant sectional curvature $\bar{K}$ and let $\bP:\Sigma \hookrightarrow M^m$  be a smooth immersion of the smooth  surface $\Sigma$. 

If $\bP$ has parallel mean curvature then $\bP$ is constraint-conformal conformal Willmore. \hfill $\Box$
\end{Prop}

The assumption on the sectional curvature is not trivial, indeed combining results of \cite{PX} and \cite{Mon2} we get the following rigidity theorem:

\begin{Th}{\bf [Rigidity for Willmore]}\label{Th:rigidity}
Let $(M^3,h)$ be a compact 3-dimensional Riemannian manifold with constant scalar curvature. Then $M$ has constant sectional curvature if and only if every smooth constant mean curvature sphere is conformal Willmore.\hfill $\Box$
\end{Th}

After having studied the analysis of the Euler Lagrange equation of the mentioned $L^2$ curvature functionals we move to establish existence of minimizers of such functionals.  We will study both \emph{curvature} and \emph{topological} conditions which ensure the existence of a minimizer. 

Before passing to the existence theorems observe that minimizing the Willmore and the other $L^2$ curvature functionals among smooth immersion is of course a-priori an ill posed variational problem.  In \cite{Riv2} (see also \cite{RiCours}), the second author introduced the suitable setting for dealing with minimization problems
whose highest order term is given by the Willmore energy. We now recall the notion of \emph{weak branched immersions with finite total curvature}.

\medskip

By virtue of Nash theorem we can always assume that $M^m$ is isometrically embedded in some euclidian space ${\R}^n$. We first define the Sobolev spaces
from ${\Sp^2}$ into $M^m$ as follows: for any $k\in {\N}$ and $1\le p\le\infty$
\[
W^{k,p}(\Sp^2,M^m):=\lf\{u\in W^{k,p}(\Sp^2,{\R}^n)\ \mbox{ s. t. }\ u(x)\in M^m\ \mbox{ for a.e. }x\in \Sp^2\rg\}\quad.
\]
Now we introduce the space of \emph{possibly branched lipschitz immersions}: a map $\vec{\Phi}\in W^{1,\infty}(\Sp^2,M^m)$ is a 
\emph{possibly branched lipschitz immersion}  if 
\begin{itemize}
\item[i)] there exists $C>1$ such that
\be
\label{I.1}
\forall x\in \Sp^2\quad\quad C^{-1}|d\vec{\Phi}|^2(x)\le |d\vec{\Phi}\wedge d\vec{\Phi}|(x)\le |d\vec{\Phi}|^2(x)
\ee
where the norms of the different tensors have been taken with respect to the standard metric on $\Sp^2$ and with respect to the metric $h$ on $M^m$ and
where $d\vec{\Phi}\wedge d\vec{\Phi}$ is the tensor given in local coordinates on $\Sp^2$ by $$d\vec{\Phi}\wedge d\vec{\Phi}:=2\ \p_{x_1}\vec{\Phi}\wedge\p_{x_2}\vec{\Phi}\ dx_1\wedge dx_2\in \wedge^2T^\ast {\Sp^2}\otimes \wedge^2 T_{\vec{\Phi}(x)}M^m.$$
\item[ii)] There exists at most finitely many points $\{a_1\cdots a_N\}$ such that for any compact $K\subset \Sp^2\setminus \{a_1\cdots a_N\}$
\be
\label{I.2}
ess\inf_{x\in K}|d\vec{\Phi}|(x)>0.
\ee
\end{itemize}
For any {\it possibly branched lipschitz immersion} we can define almost everywhere the {\it Gauss map} 
$$
\vec{n}_{\vec{\Phi}}:=\star_h\frac{\p_{x_1}\vec{\Phi}\wedge\p_{x_2}\vec{\Phi}}{|\p_{x_1}\vec{\Phi}\wedge\p_{x_2}\vec{\Phi}|}\ \in\ \wedge^{m-2} T_{\vec{\Phi}(x)}M^m
$$
where $(x_1,x_2)$ is a local arbitrary choice of coordinates on $\Sp^2$ and $\star_h$ is the standard Hodge operator associated to the metric  $h$ on multi-vectors in $TM$.

With these notations we define
\begin{Dfi}
\label{df-I.1}
A lipschitz map $\vec{\Phi}\in  W^{1,\infty}({\Sp^2},M^m)$ is called ''weak, possibly branched, immersion'' if $\vec{\Phi}$ satisfies
(\ref{I.1}) for some $C\ge 1$, if it satisfies (\ref{I.2}) and if the Gauss map satisfies
\be
\label{I.3}
\int_{\Sp^2}|D\vec{n}_{\vec{\Phi}}|^2\ dvol_g<+\infty
\ee
where $dvol_g$ is the volume form associated to $g:=\vec{\Phi}^\ast h$ the pull-back metric of $h$ by $\vec{\Phi}$ on $\Sp^2$, $D$ denotes the covariant derivative with respect to $h$ and the norm $|D\vec{n}_{\vec{\Phi}}|$ of the tensor $D\vec{n}_{\vec{\Phi}}$ is taken with respect to $g$ on $T^\ast \Sp^2$ and $h$ on $\wedge^{m-2}TM$.
The space of ''weak, possibly branched, immersion'' of $\Sp^2$ into $M^m$ is denoted ${\mathcal F}_{{\Sp^2}}$. \hfill $\Box$
\end{Dfi}

Using M\"uller-Sverak theory of weak isothermic charts (see \cite{MS}) and H\'elein moving frame technique  (see \cite{Hel}) one can prove the following proposition (see \cite{RiCours}).
\begin{Prop}
\label{pr-I.1}
Let $\vec{\Phi}$ be a  weak, possibly branched, immersion of ${\Sp^2}$ into $M^m$ in ${\mathcal F}_{{\Sp^2}}$ then there exists a bilipschitz homeomorphism $\Psi$ of $\Sp^2$
such that $\vec{\Phi}\circ\Psi$ is weakly conformal : it satisfies almost everywhere on $\Sp^2$
\[
\lf\{
\begin{array}{l}
\ds |\p_{x_1}(\vec{\Phi}\circ\Psi)|_h^2=|\p_{x_2}(\vec{\Phi}\circ\Psi)|_h^2\quad\\[5mm]
\ds h(\p_{x_1}(\vec{\Phi}\circ\Psi),\p_{x_2}(\vec{\Phi}\circ\Psi))=0
\end{array}
\rg.
\]
where $(x_1,x_2)$ are local arbitrary conformal coordinates in $\Sp^2$ for the standard metric. Moreover $\vec{\Phi}\circ\Psi$ is in $W^{2,2}\cap W^{1,\infty}(\Sp^2,M^m)$.
\hfill$\Box$
\end{Prop}

\medskip

\begin{Rm}
\label{rm-I.1}
In view of  Proposition~\ref{pr-I.1} a careful reader could wonder why we do not work  with \emph{conformal} $W^{2,2}$ weak, possibly branched, immersion only and why we
do not impose for the membership in ${\mathcal F}_{{\Sp^2}}$, $\vec{\Phi}$ to be conformal from the begining. The reason why this would be a wrong strategy and why
we have to keep the flexibility
for weak immersions not to be necessarily conformal is clear in the proof of the existence theorems, Section \ref{Sec:Existence} and in the Appendix where we will study  the variations of the functionals under general perturbations which do not have to respect
infinitesimally the conformal condition.  \hfill $\Box$
\end{Rm}
\medskip

Now that we have introduced the right framework we pass to discuss the existence theorems. 
\\Fix a point $\bar p\in M^m$  and a $3$-dimensional subspace ${\frak S}< T_{\bar p} M$ of the tangent space to $M$ at $\bar p$. We denote
\be \label{eq:defRcalT}
R_{\bar p}({\frak S}):= \sum_{i\neq j, i,j=1,2,3} \bar K_{\bar p} \left(\bE_i, \bE_j\right) 
\ee 
where $\{\bE_1,\bE_2,\bE_3\}$ is an orthonormal basis of ${\frak S}$ and $\bar K_{\bar p} (\bE_i, \bE_j)$ denotes the sectional curvature of $(M,h)$ computed on the plane spanned by $(\bE_i,\bE_j)$ contained in $T_{\bar p} M$. Notice that $R_{\bar p}({\frak S})$ coincides with the scalar curvature at $\bar p$ of the 3-dimensional submanifold of $M$ obtained exponentiating ${\frak S}$. Under a condition on $R_{\bar p}({\frak S})$, in the following theorem we minimize the functional $F_1$ defined on ${\cal F}$ as 
\be\label{def:F1}
F_1(\bP):= \int_{\Sp^2} \left( \frac{1}{2}|{\mathbb I}|^2+1 \right) dvol_g = F(\bP)+A(\bP).
\ee

\begin{Th}\label{TeoExWillCurv}
Let $(M^m,h)$ be a compact Riemannian manifold and assume there is a point $\bar p$ and a $3$-dimensional subspace ${\frak S}< T_{\bar p} M$ such that $R_{\bar p}({\frak S})>6$, where $R_{\bar p}({\frak S})$ is the curvature quantity defined in \eqref{eq:defRcalT}. Then there exists a branched conformal immersion $\bP$ of $\Sp^2$ into $(M^m,h)$ with finitely many branched points $b^1,\ldots,b^{N}$, smooth on $\Sp^2\setminus\{b^1,\ldots,b^{N}\}$, minimizing the functional $F_1$ in ${\mathcal F}_{\Sp^2}$, i.e. among weak branched immersions with finite total curvature. \hfill $\Box$
\end{Th}    
Observe that the unit round $m$-dimensional sphere $\Sp^m$ with canonical metric has $R_{\bar p}({\frak S})\equiv 6$ for any base point $\bar{p}$ and any subspace ${\frak S}$, so the assumption is that our ambient manifold has at least one point $\bar{p}$ and at least three directions spanning ${\frak S}$ where the manifold is ''more positively curved'' than  $\Sp^m$.
Let us make a remark about the regularity in the branch points.

\begin{Rm}
The removability of point singularities for Willmore surfaces in Euclidean space has been studied in \cite{KS}, \cite{KS2} and \cite{Riv1}; recently  Y. Bernard and the second author, in \cite{BRRem}, proved that the parametrization is smooth also in the branch points if two residues vanish. Analogous statements should hold for branched Willmore immersions in manifolds. 
\hfill $\Box$
\end{Rm}
\begin{Rm}
\label{rm-branch}
It is always possible to minimize $F_1$ by forcing the immersion to pass through a fixed family of points. For an arbitrary choice of points sufficiently close to the minimizers 
we found in theorem~\ref{TeoExWillCurv}, this should generate a Willmore sphere passing through these points but satisfying the Willmore equation only away from these points. Since in the variational argument these points
cannot be moved the corresponding residues obtained in \cite{KS}, \cite{Riv1} and \cite{BRRem} have no reason to vanish and the conformal parametrization $\vec{\Phi}$ of a minimizer
should be at most $C^{1,\al}$ in general. This should contrast presumably with the situation at  the branched points of the minimizers obtained in theorem~\ref{TeoExWillCurv}.
Since these points are left free during the minimization procedure, the first residue $\vec{\gamma}_0$ (see \cite{BRRem}) should vanish and the conformal map $\vec{\Phi}$ should be at least
$C^{2,\al}$ at these points.\hfill $\Box$
\end{Rm} 

Now let us consider the problem of minimizing the functional $F=\int |{\mathbb I}|^2$. In codimension one, E. Kuwert, J. Schygulla and the first author, in \cite{KMS},  proved the existence of a smooth immersion of $\Sp^2$ \emph{without branched points} minimizing the functional $F$ under curvature conditions on the compact ambient 3-manifold (see also \cite{MonSchy} for \emph{non compact} ambient 3-manifolds); notice that the topological argument employed for excluding the branch points crucially depends on the codimension one assumption. Therefore, in higher codimension,  it makes sense to look for minimizers of $F$ among \emph{branched} immersions, as done in the following theorem.   

\begin{Th}\label{TeoExWill}
Let $(M^m,h)$ be a compact Riemannian manifold. Assume there is a minimizing sequence for the  functional $F=\frac{1}{2}\int |{\mathbb I}|^2 $ in  ${\mathcal F}_{\Sp^2}$ ( among weak possibly branched immersions with finite total curvature), $\{\bP_k\}_{k\in \N} \subset {\mathcal F}_{\Sp^2}$,   with area bounded by positive costants from below and above: $$0<\frac{1}{C}\leq A(\bP_k)\leq C<\infty.$$
Then there exists a branched conformal immersion $\bP$ of $\Sp^2$ into $(M^m,h)$ with finitely many branched points $b^1,\ldots,b^{N}$, smooth on $\Sp^2\setminus\{b^1,\ldots,b^{N}\}$, minimizing the  functional $F$ in ${\mathcal F}_{\Sp^2}$, i.e. among weak branched immersions with finite total curvature.\hfill $\Box$
\end{Th}

\begin{Rm}
By analogous arguments to the proof of Theorem \ref{TeoExWillCurv}, the lower bound on the area is ensured if $R_{\bar p}({\frak S})>0$ for some point $\bar p$ and $3$-dimensional subspace ${\frak S}< T_{\bar p} M$.

Notice a uniform upper bound on the areas of the minimizing sequence is a crucial information for  compactness issues; moreover generally this  is not a trivial property in view of the possibility of 
totally geodesic laminations (A similar constraint appears in \cite{Mon3}). \hfill $\Box$
\end{Rm}

Up to here we studied existence of minimizeres of curvature functionals under \emph{curvature} conditions on the ambient manifold. Now we move to consider existence of area-constrained Willmore spheres under \emph{topological} conditions on the ambient manifold. 

For any $x_0\in M^m$ we denote respectively by $\pi_2(M^m,x_0)$ the \emph{homotopy groups of based maps} form ${\Sp^2}$
into $M^m$ sending the south pole to $x_0$  and by $\pi_0(C^0({\Sp^2},M^m))$ the  free homotopy classes. It is well known that the group $\pi_2(M^m,x_0)$ for different $x_0$ are isomorphic to each other and $\pi_2(M)$ denotes any of the $\pi_2(M^m,x_0)$ modulo isomorphisms. Recall that, in \cite{SaU},  J. Sacks and K. Uhlenbeck  proceeded to the minimization of the Dirichlet energy 
$$E(\vec{\Phi})=\frac{1}{2}\int_{{\Sp^2}}|d\vec{\Phi}|^2\ dvol_{{\Sp^2}}$$
 among mappings $\vec{\Phi}$ of the two sphere ${\Sp^2}$ into $M^m$ within a fixed based homotopy class in $\pi_2(M^m,x_0)$ in order to generate area minimizing, possibly branched, immersed spheres realizing this homotopy class.
\\Even if the paper had a great impact in mathematics, the program of Sacks and Uhlenbeck
was only partially successful. Indeed the possible loss of compactness arising in the minimization process can generate a union of immersed spheres
realizing the corresponding \emph{free homotopy class} but for which the underlying component in the \emph{homotopy group} $\pi_2(M^m)$ may have been forgotten (for more details see also the Introduction to \cite{MoRi1}). It is very hard in the Sacks Uhlenbeck's work to distinguish  the  classes  which are realized by minimal conformal immersions from the somehow not satisfying classes. At least Sacks and Uhlenbeck could prove that the set of {\it satisfying classes} generates, as a $\pi_1-$module, the homotopy group $\pi_2(M^m)$.

To overcome this difficulty, we minimize a curvature functional - corresponding to $A+W$ in the absence of branched points - under homotopy constraint and  we prove that, even if we still have a bubbling phenomenon, the limit object must be \emph{connected}. More precisely we show  that for every non trivial  2-homotopy group of $M^m$ there is  a canonical representative given by a Lipschitz map from $\Sp^2$ to $M$ realizing the \emph{connected} union of  conformal branched  area-constraint Willmore spheres which are smooth outside the branched points. Notice that this is a natural generalization of Sacks Uhlenbeck's procedure in a sense that, if a class  $\gamma$ in $\pi_2(M^m)$ possesses an area minimizing
immersion $\vec{\Phi}$ then $\vec{H}_{\vec{\Phi}}\equiv 0$, in particular $\vec{\Phi}$ is an area-constraint Willmore sphere minimizing $A+W$ in it's homotopy class.

Before stating the theorem let us recall that for any Lipschitz mapping $\vec{a}$ from $\Sp^2$ into $M^m$, $(\vec{a})_\ast[\Sp^2]$ denotes the current given by the push-forward by $\vec{a}$ of the current of integration over $\Sp^2$ : for any smooth two-form
$\om$ on $M^m$
\[
\lf<(\vec{a})_\ast[\Sp^2],\om\rg>:=\int_{\Sp^2}(\vec{a})^\ast\om.
\] 
Moreover we denote with $[\vec{a}]\in \pi_2(M^m)$ the 2-homotopy class corresponding to the continuous map $\vec{a}:\Sp^2 \to M^m$.
 
\begin{Th}\label{Th:ExWillHom}
Let $(M^m,h)$ be a compact Riemannian manifold and fix  $0\neq \gamma \in \pi_2(M^m)$. Then there exist finitely many branched conformal immersions $\bP^1,\ldots,\bP^N \in {\cal F}_{\Sp^2}$ and a Lipschitz map $\vec{f}\in W^{1,\infty}(\Sp^2,M^m)$ with $[\vec{f}]=\gamma$ satisfying
\begin{eqnarray}
\vec{f}(\Sp^2)&=&\bigcup_{i=1}^N \bP^i(\Sp^2)\quad \text{and } \label{eq:fcupPhi}  \\
\vec{f}_*[\Sp^2]&=&\sum_{i=1}^N \bP^i_*[\Sp^2] \label{eq:fsumPhi}\quad . 
\end{eqnarray}
Moreover for every $i$, the map $\bP_i$ is a conformal branched  area-constrained Willmore immersion which is smooth outside the finitely many branched points $b^1,\ldots,b^{N_i}$. More precisely we mean that, outside the branched points, every $\bP^i$ is a smooth solution to the Willmore equation with the Lagrange multiplier $2\bH$:
\be\label{eq:ConfWillAreaConstraint}
\frac{1}{2} D^{*_g}_g \left[D_g \bH -3 \pi_{\bn} (D_g \bH) + \star_h\left((*_g D_g \bn) \wedge_M \bH \right)  \right] =2\bH+\ti{R}(\bH)-R^\perp_{\vec{\Phi}}(T\vec{\Phi}),
\ee
where $\pi_{\bn}$ is the projection onto the normal space to $\bP$, $\star_h$ and $*_g$ are respectively the Hodge operator on $(M,h)$ and $(\Sp^2,g:=\bP^*h)$; $\ti{R}$ and $R^\perp$ are the curvature endomorphisms defined respectively in \eqref{eq:defR} and \eqref{Rperp}. The operators $D_g$, $D^{*_g}_g, \ldots$  are defined above (see also more explicit expressions in  Section \ref{Sec:ConsWillHCD}). \hfill $\Box$
\end{Th}

\begin{Rm}
With the same proof, the analogous theorem about existence of a connected family of smooth branched conformal  immersions of $\Sp^2$ which are  area-constrained critical points for the functional $F$  and are realizing a fixed homotopy class holds. \hfill $\Box$
\end{Rm}

\begin{Rm} It might be interesting to investigate whether the minimizer in a fixed homotopy class is really obtained by a Lipschitz realization of \emph{more than one} smooth branched immersions of spheres or it is realized by \emph{exactly one} smooth branched immersion of $\Sp^2$. The asymptotic behavior of the solutions at possible connection points
of 2 distinct spheres in relation with the cancellation of the first residue $\vec{\gamma}_0$ mentioned in remark~\ref{rm-branch}  (which should also hold in the situation of theorem~\ref{Th:ExWillHom})
is a starting point for studying the possibility to have such connection points while considering an absolute minimizer.\hfill $\Box$
\end{Rm}
 

Let us give here an idea of the proof of Theorem \ref{Th:ExWillHom}. Consider the following Lagrangian $\LW$ defined on ${\cal F}_{\Sp^2}$
\be\label{def:LW}
\LW(\bP):=\int_{\Sp^2} \left( \frac{1}{4} |\mathbb{I}|^2-\frac{1}{2} \bar{K}(T\bP)+1 \right) dvol_g,
\ee
where $\bar{K}(T\bP)$ is the sectional curvature of the ambient manifold $(M^m,h)$ evaluated on the tangent space to $\bP(\Sp^2)$ and observe that, by the Gauss equation, outside the branch points it holds
\be\label{GaussEquat}
\frac{1}{4} |{\mathbb {I}}|^2-\frac{1}{2} \bar{K}(T\bP)+1=|\bH|^2+1-\frac{1}{2} K_{\bP},
\ee
where $K_{\bP}$ is the Gauss curvature of $\bP$, i.e. the sectional curvature of the metric $g=\bP^*(h)$ on $\Sp^2$. Notice that, since the Gauss curvature integrated on compact subsets away the branch points gives a null lagrangian (i.e. a lagrangian with null first variation with respect to compactly supported variations), the Euler-Lagrange equation of $\LW$ coincides with the Euler-Lagrange equation of $\int(|\bH|^2+1)$ outside the branched points; therefore the critical points of $\LW$ satisfy the area-constrained Willmore equation \eqref{eq:ConfWillAreaConstraint} outside the branched points.
\\Our approach is then to minimize $L$;  the space on which the minimization procedure is performed is the set ${\cal T}$ of $N+1$-tuples $\vec{T}=(\vec{f}, \bP^1, \ldots \bP^N)$, where $N$ is an arbitrary positive integer, where $\vec{f}\in W^{1,\infty}(\Sp^2,M^m)$ and $\bP^i\in {\cal F}_{\Sp^2}$ satisfy \eqref{eq:fcupPhi} and  \eqref{eq:fsumPhi}; naturally we define
\be \label{def:LWTree}
\LW(\vec{T})=\sum_{i=1}^N \LW(\bP^i). 
\ee
Observe that, up to rescaling the ambient metric $h$ by a positive constant, we can always assume that $\bar{K}\leq 1$ on all $M$ (or equivalently choose in \eqref{def:LW}, instead of $1$, a large positive constant $C>\max_M \bar{K}$). On a minimizing sequence $\vec{T}_k$ under the constraint that the map $\vec{f_k}\in W^{1,\infty}(\Sp^2,M^m)$ is in the fixed homotopy class $0\neq \gamma \in \pi_2(M^m)$, both the areas and the $L^2$ norms of the second fundamental forms are clearly equibounded; therefore, using results from  \cite{MoRi1} we construct a  minimizer $\vec{T}_\infty=(\vec{f}_\infty, \bP^1_{\infty},\ldots,\bP^{N_\infty}_{\infty}) \in {\cal T}$ such that $\vec{f}_\infty \in W^{1,\infty}(\Sp^2,M^2)$ is still in the homotopy class $\gamma$. Using the regularity theory developed in Section \ref{Sec:Regularity} we conclude with the smoothness of the $\bP^i_{\infty}$ outside the finitely many branched points.
\\
  
Observe that, for  \emph{small} values of the  area,   smooth (contractible in $M$) area constraint-Willmore spheres have been constructed in \cite{LM2} (see also \cite{CL}, \cite{LM}, \cite{LMS}, \cite{Mon1}, \cite{Mon2}) as perturbations of small geodesic spheres using perturbative methods; notice that instead Theorem \ref{Th:ExWillHom}   deals with the global situation when the topology of the ambient manifold plays a crucial role. Moreover in the next theorem we produce area-constrained Willmore spheres for \emph{any} value of the area. More precisely consider the lagrangian $W_K$ defined on ${\mathcal F}_{\Sp^2}$ as follows
\be\label{def:WK}
W_K(\bP):=\int_{\Sp^2} \left(\frac{1}{4} |{\mathbb I}|^2-\frac{1}{2} \bar{K}(T\bP) \right) dvol_g.
\ee
Using the Gauss equation, one has
\be\label{GaussEquat}
\frac{1}{4} |{\mathbb {I}}|^2-\frac{1}{2} \bar{K}(T\bP)=|\bH|^2-\frac{1}{2} K_{\bP},
\ee 
and, as before, this implies that the critical points of $W_K$ satisfy exactly the Willmore equation outside the branch points. Notice moreover that, if one considers just \emph{non branched} immersions then $W_K$ is exactly the Willmore functional $W$ up to an additive topological constant by Gauss Bonnet Theorem, so  minimizing $W_K$ under area constraint among \emph{branched} immersions is the natural generalization of minimizing $W$ under area constraint among \emph{non branched} immersions; moreover the possibility of having a branched minimal sphere (for the existence of branched minimal spheres in Riemannian manifolds see for example \cite{SaU}) for a fixed value of the area suggests that the correct setting, for the global problem of minimizing the Willmore functional under area constaint for \emph {not necessarily small} values of the area,  is the one of branched immersions.

\begin{Th}\label{Th:AreaConstraintWill}
Let $(M^m,h)$ be a compact Riemannian manifold and fix \emph{any} ${\cal A}>0$. Then there exist finitely many branched conformal immersions $\bP^1,\ldots,\bP^N \in {\cal F}_{\Sp^2}$ and a Lipschitz map $\vec{f}\in W^{1,\infty}(\Sp^2,M^m)$ with
\begin{eqnarray}
\sum_{i=1}^N A(\bP_i)&=& {\cal A}\quad , \label{eq:ConstraintA}\\
\vec{f}(\Sp^2)&=&\bigcup_{i=1}^N \bP(\Sp^2)\quad \text{and } \label{eq:fcupPhi1}  \\
\vec{f}_*[\Sp^2]&=&\sum_{i=1}^N \bP^i_*[\Sp^2] \label{eq:fsumPhi1}\quad,  
\end{eqnarray}
such that for every $i$, the map $\bP_i$ is a conformal branched  area-constraint Willmore immersion which is smooth outside the finitely many branched points $b^1,\ldots,b^{N_i}$. More precisely we mean that, outside the brached points, every $\bP^i$ is a smooth solution to the Willmore equation  with the Lagrange multiplier ${\frak a}\bH$ (for some ${\frak a}\in \R$)
\be\label{eq:ConfWillAreaConstraint}
\frac{1}{2} D^{*_g}_g \left[D_g \bH -3 \pi_{\bn} (D_g \bH) + \star_h\left((*_g D_g \bn) \wedge_M \bH \right)  \right] ={\frak a}\bH+\ti{R}(\bH)-R^\perp_{\vec{\Phi}}(T\vec{\Phi}),
\ee
with the same notation as in Theorem \ref{Th:ExWillHom}.
Moreover the $N+1$-tuple $\vec{T}=(\vec{f},\bP^1,\ldots,\bP^N)$ minimizes the functional $W_K$ in the set of tuples ${\cal T}$ (defined above) having area ${\cal A}$, where the area and the $W_K$ functional of an element $\vec{T}=(\vec{f},\bP^1,\ldots,\bP^N) \in {\cal T}$ are defined in a natural way as $A(\vec{T})=\sum_{i=1}^N A(\bP^i)$ and $W_K(\vec{T})=\sum_{i=1}^N W_K(\bP^i)$ respectively. 
\hfill $\Box$
\end{Th}
With the same proof, the analogous theorem about existence of a connected family of smooth branched conformal immersions of $\Sp^2$ which are  area-constrained critical points for the functional $F$  and whose total area is an arbitrary ${\cal A}>0$  holds; this connected family moreover minimize the functional $F$ in $\cal{T}$ under the area constraint $A(T)={\cal A}$.

\begin{Rm}
For small area ${\cal A}< \ep_0$, by the monotonicity formula ( the monotonicity formula is a crucial tool introduced in \cite{SiL}, for the proof in this context see Lemma \ref{lem:MonFor}) the minimizer has also small diameter and thanks to the estimates contained in \cite{LM}, \cite{LM2}, the minimum of the functional $W_K$ is close to $2\pi$.  With arguments analogous to \cite{LM2} ( using \cite{LY}), one checks that, for small area ${\cal A}< \ep_0$, the minimizer produced in Theorem \ref{Th:AreaConstraintWill} is made of \emph{just one} smooth \emph{non branched} area-constrained Willmore immersion of the 2-sphere. Therefore  Theorem \ref{Th:AreaConstraintWill}  is the natural generalization of the main theorem in \cite{LM2}, where  T. Lamm and J. Metzger minimize the Willmore functional under \emph{small area} constraint among \emph{non branched} little spheres. 
\hfill $\Box$
\end{Rm}

\begin{center}

{\bf Acknowledgments}
\\This work was written while the first author was visiting the {\it Forschungsinstitut f\"ur Mathematik} at the ETH Z\"urich. He would like  to thank the Institut for the hospitality and the excellent working conditions.
\end{center}

\medskip
\begin{center}
{\bf Notations and conventions} 
\end{center}
For the Riemann curvature tensor $\Riem^h$ of $(M^m,h)$ we use the convention of \cite{Will} (notice that other authors, like \cite{DoC}, adopt theh opposite sign convention): for any $\vec{X},\vec{Y},\vec{Z} \in T_xM$ define 
$$\Riem^h(\vec{X},\vec{Y})\vec{Z}:=D_{\vec{X}}D_{\vec{Y}}\vec{Z}- D_{\vec{Y}}D_{\vec{X}} \vec{Z}-D_{[\vec{X},\vec{Y}]}\vec{Z}.$$

The \emph{Hodge operator} on $\R^m$ or more generally on the tangent space $T_xM$ of an oriented Riemannian manifold $(M^m,h)$ is the linear map from $\wedge^p T_xM$ into $\wedge^{m-p} T_xM$ which to a $p$-vector $\alpha$ assigns the $m-p$-vector $\star_h \alpha$ on $T_xM$ such that for any $p$-vector $\beta$ in $\wedge ^p T_xM$ the following identity holds:
\be\label{eq:defHodge}
\beta \wedge \star_h \alpha= <\beta,\alpha>_h \vec{E}_1\wedge \ldots \wedge \vec{E}_m
\ee
where $(\vec{E}_1,\ldots,\vec{E}_m)$ is an orthonormal positively oriented basis of $T_xM$ and $<.,.>_h$ is the scalar product on $\wedge^p T_xM$ induced by $h$. Notice that even if $M^m$ is not orientable, we can still define $\star_h$ locally. 

We will also need the concept of \emph{interior multiplication} $\llcorner$ between $p-$ and $q-$vectors, $p\geq q$, producing a $p-q$-vector such that (see \cite{Fed} 1.5.1 combined with 1.7.5): for every choice of $p-$,$q-$ and $p-q-$vectors, respectively $\alpha$, $\beta$ and $\gamma$ the following holds:
\be\label{eq:defllcorner}
<\alpha\llcorner \beta, \gamma>=<\alpha, \beta \wedge \gamma>
\ee

We call $\bul$ the following contraction operation which to a pair of $p-$ and $q-$vectors $(\alpha,\beta)$ assigns the $p+q-2-$vector $\alpha\bul \beta$ such that:
\\- if $q=1$,  $\alpha\bul \beta:= \alpha \llcorner \beta$,
\\-if $\alpha \in \wedge^p T_xM$, $\beta\in \wedge^q T_xM$ and $\gamma \in \wedge^s T_xM$ then
\be \label{def:bul}
\alpha \bul (\beta\wedge \gamma):=(\alpha \bul \beta)\wedge \gamma + (-1)^{rs} (\alpha\bul \gamma) \wedge \beta. 
\ee

\section{The conservative form of the Willmore surface equation in $3$-dimensional manifolds}\label{Sec:ConsWill3D}
\reset
Let $\Sigma^2$ be an abstract closed surface, $(M,h)$ a $3$ dimensional Riemannian manifold and $\vec{\Phi}: \Sigma^2 \hookrightarrow (M,h)$ a smooth immersion. Since the following results are local, we can work locally in a disc-neighborhood of a point and use isothermal coordinates
on this disc. This means that we can assume $\vec{\Phi}$ to be a conformal immersion from the unit disc $D^2\subset {\R}^2$
into $(M,h)$.

Let us introduce some notations. Given the conformal immersion $\vec{\Phi}:D^2 \hookrightarrow (M,h)$ we call $g:=\vec{\Phi}^*h=e^{2\lambda} (dx_1^2+dx_2^2)$ the induced metric; denote $(\vec{e}_1,\vec{e}_2)$ the orthonormal basis of $\vec{\Phi}_\ast(T\Sigma^2)$ given by
$$
\vec{e} _i:=e^{-\la}\ \frac{\p \vec{\Phi}}{\p x_i}\quad,
$$
where $e^\la=|\p_{x_1}\vec{\Phi}|=|\p_{x_2}\vec{\Phi}|$. The unit normal vector $\vec{n}$ to $\vec{\Phi}(\Sigma)$ is then given by
\[
\vec{n}=\star_h  (\vec{e_1}\wedge \vec{e_2})
\]

Denoted with $\D$ the covariant derivative of $(M,h)$ we have  the second fundamental form
$$\vec{\mathbb I}(\vec{X},\vec{Y}):= -<D_{\vec{X}} \vec{n}, \vec{Y}> \vec{n} $$
and the mean curvature
$$ \vec{H}:=\frac{1}{2} \left[\vec{\mathbb I} (\vec{e_1},\vec{e_1})+\vec{\mathbb I} (\vec{e_2},\vec{e_2})\right]. $$

Introduce moreover the {\it Weingarten Operator} expressed in conformal  coordinates $(x_1,x_2)$ as :
\[
\vec{H}_0:=\frac{1}{2}\lf[\vec{\mathbb I}(e_1,e_1)-\vec{\mathbb I}(e_2,e_2)-2\,i\, \vec{\mathbb I}(e_1,e_2)\rg]\quad.
\]

In \cite{Riv1} an alternative form to the Euler Lagrange equation of Willmore functional in euclidean setting was proposed; our goal is to do the same for immersions in a Rienannian manifold. 

\begin{Th}
\label{th-VI.8-3d}
Let $\vec{\Phi}$ be a smooth immersion of a two dimensional manifold $\Sigma^2$ into a 3-dimensional Riemannian manifold $(M^3,h)$; restricting the immersion to a small disc neighboorod of a point where we consider local conformal coordinate, we can see $\vec{\Phi}$ as a conformal immersion of $D^2$ into $(M,h)$. Then the following identity holds
\be
\label{VI.71}
\begin{array}{l}
-2 e^{2 \la} \Delta_g H\ \vec{n}-4 e^{2\la}\ \vec{H}\ (H^2-(K^g-K^h)) +2e^{2\la} R^\perp _\Phi (T\vec{\Phi}) \\[5mm]
\ds \qquad \quad= D^*\left[-2\nabla H \vec{n}+H \D \vec{n} - H\star_h (\vec{n}\wedge \D^\perp \vec{n})\right]\quad,
\end{array}
\ee
where $\vec{H}$ is the mean curvature vector of the immersion $\vec{\Phi}$, $\Delta_g$ is the negative Laplace Beltrami operator, $\star_h$ is the Hodge operator associated to metric $h$, $\D \cdot:=(\D_{\p_{x_1} \vec{\Phi}}\cdot, \D_{\p_{x_2} \vec{\Phi}}\cdot)$ and $\D^\perp\cdot:=(-\D_{\p_{x_2} \vec{\Phi}}\cdot,\D_{\p_{x_1} \vec{\Phi}}\cdot)$ and $\D^*$ is an operator acting on couples of vector fields $(\vec{V}_1,\vec{V}_2)$ along $\vec{\Phi}_*(T\Sigma)$ defined as 
$$\D^*(\vec{V}_1,\vec{V}_2):= \D_{\p_{x_1} \vec{\Phi}} \vec{V}_1+\D_{\p_{x_2} \vec{\Phi}} \vec{V}_2.$$
Finally recall the definition \eqref{Rperp} of  $R^\perp_{\vec{\Phi}}(T\vec{\Phi}):=(Riem(\vec{e_1},\vec{e_2})\vec{H})^\perp =\star_h\left( \vec{n}\wedge Riem^h(\vec{e_1},\vec{e_2})\vec{H} \right)$.  
\hfill $\Box$ 
\end{Th}
A straightforward but important consequence of Theorem~\ref{th-VI.8-3d} is the following conservative form of Willmore surfaces equations.
\begin{Co}
\label{co-VI.8}  
A conformal immersion $\vec{\Phi}$ of a 2-dimensional disc $D^2$ is Willmore if and only if 
\be
2 e^{2\la}  [R^\perp _\Phi (T\vec{\Phi}) + \vec{H} Ric_h (\vec{n}, \vec{n})]= D^*\left[-2 D\vec{H} + 3 H \D \vec{n} - \star_h (\vec{H}\wedge \D^\perp \vec{n})\right]
\ee
\hfill $\Box$
\end{Co}

Now recall that an immersion $\vec{\Phi}$ is said to be \emph{constrained-conformal  Willmore} if and only if it is a critical point of the Willmore functional under the constraint that the conformal class is fixed. In \cite{BPP} is derived the Willmore equation under conformal constraint for immersions of surfaces in a 3-dimensional Riemannian manifold, which, matched with Theorem \ref{th-VI.8-3d}, gives the following corollary.

\begin{Co}
\label{co-VI.9}  
A conformal immersion $\vec{\Phi}$ of a 2-dimensional disc $D^2$ is constrained-conformal Willmore if and only if there exists an holomorphic function $f(z)$ such that

\be
2 e^{2\la}  \left[R^\perp _\Phi (T\vec{\Phi}) + \vec{H} Ric_h (\vec{n}, \vec{n})+e^{-2\lambda}\Im (f(z) \overline{\vec{H_0}}) \right]=
 D^*\left[-2 D\vec{H} + 3 H \D \vec{n} - \star_h (\vec{H}\wedge \D^\perp \vec{n})\right]\quad
\ee
\hfill$\Box$
\end{Co}

We proceed in the following way: first we prove a general lemma for conformal immersions of the 2-disc in $(M^3,h)$, then we pass to the proof of the theorem and of its corollaries.

\begin{Lm}
\label{lm-VI.10}
Let $\vec{\Phi}$ be a conformal immersion from $D^2$ into $(M,h)$. Denote by $\vec{n}$ the unit normal vector $\vec{n}=*_h(\vec{e_1}\wedge \vec{e_2})$  of the conformal immersion $\vec{\Phi}$ and denote by $H$ the mean curvature. Then the following identity holds
\be
\label{VI.66}
-2H\ \nabla\vec{\Phi}= \D \vec{n}+ \star _h (\vec{n} \wedge \D^\perp\vec{n})\quad
\ee
where $\D \cdot:=(\D_{\p_{x_1} \vec{\Phi}}\cdot, \D_{\p_{x_2} \vec{\Phi}}\cdot)$ and $\D^\perp\cdot:=(-\D_{\p_{x_2} \vec{\Phi}}\cdot,\D_{\p_{x_1} \vec{\Phi}}\cdot)$.
\hfill$\Box$
\end{Lm}

{\bf Proof of lemma~\ref{lm-VI.10}.}
Denote $(\vec{e}_1,\vec{e}_2)$ the orthonormal basis of $\vec{\Phi}_\ast(T\Sigma^2)$ given by
$$
\vec{e}_i:=e^{-\la}\ \frac{\p \vec{\Phi}}{\p x_i}\quad,
$$
where $e^\la=|\p_{x_1}\vec{\Phi}|=|\p_{x_2}\vec{\Phi}|$. The unit normal vector $\vec{n}$ is then given by
\[
\vec{n}=\star_h (\vec{e_1}\wedge \vec{e_2})\quad.
\]
We have
\[
\lf\{
\begin{array}{l}
<\vec{e}_1,\star_h (\vec{n}\wedge \D^\perp \vec{n})>=-<\D^\perp\vec{n},\vec{e}_2>\\[5mm]
<\vec{e}_1,\star_h (\vec{n}\wedge \D^\perp \vec{n})>=<\D^\perp\vec{n},\vec{e}_1>\quad.
\end{array}
\rg.
\]
From which we deduce
\[
\lf\{
\begin{array}{l}
\ds-\star_h (\vec{n}\wedge \D_{\p_{x_2}\vec{\Phi}}\vec{n})=<\D_{\p_{x_2}\vec{\Phi}}\vec{n},\vec{e}_2> \vec{e}_1-<\D_{\p_{x_2}\vec{\Phi}}\vec{n},\vec{e}_1>\ \vec{e}_2\\[5mm]
\ds\star_h (\vec{n}\wedge \D_{\p_{x_1}\vec{\Phi}}\vec{n})=-<\D_{\p_{x_1}\vec{\Phi}}\vec{n},\vec{e}_2> \vec{e}_1+<\D_{\p_{x_1}\vec{\Phi}}\vec{n},\vec{e}_1>\ \vec{e}_2
\end{array}
\rg.
\]
Thus, by the symmetry of the second fundamental form,
\[
\lf\{
\begin{array}{l}
\ds \D_{\p_{x_1}\vec{\Phi}}\vec{n} -\star_h (\vec{n}\wedge \D_{\p_{x_2}\vec{\Phi}}\vec{n}) =[<\D_{\p_{x_2}\vec{\Phi}}\vec{n},\vec{e}_2>+<\D_{\p_{x_1}\vec{\Phi}}\vec{n},\vec{e}_1>]\ 
\vec{e}_1\\[5mm]
\ds \D_{\p_{x_2}\vec{\Phi}} \vec{n}+\star_h (\vec{n}\wedge \D_{\p_{x_1}\vec{\Phi}}\vec{n})=[<\D_{\p_{x_2}\vec{\Phi}} \vec{n},\vec{e}_2>+<\D_{\p_{x_1}\vec{\Phi}} \vec{n},\vec{e}_1>]\ \vec{e}_2
\end{array}
\rg.
\]
Since $H= -e^{-\la}\ 2^{-1}[<\D_{\p_{x_2}\vec{\Phi}} \vec{n},\vec{e}_2>+<\D_{\p_{x_1}\vec{\Phi}} \vec{n},\vec{e}_1>]$ we deduce (\ref{VI.66}) and 
Lemma~\ref{lm-VI.10} is proved.\hfill$\Box$

\medskip

\noindent {\bf Proof of theorem~\ref{th-VI.8-3d} and its corollaries}
First let us introduce the operator $\D^*$ acting on couples of vector fields $(\vec{V}_1,\vec{V}_2)$ along $\vec{\Phi}_*(T\Sigma)$ defined as 
$$\D^*(\vec{V}_1,\vec{V}_2):= \D_{\p_{x_1} \vec{\Phi}} \vec{V}_1+\D_{\p_{x_2} \vec{\Phi}} \vec{V}_2. $$
We can again assume that $\vec{\Phi}$ is conformal. First apply the operator $\D^*$ to (\ref{VI.66}) and multiply by $H$.
This gives
\be
\label{VI.67}
-2H^2 \D^* \D \vec{\Phi}-2H\nabla H \cdot \D\vec{\Phi}=H\,\D^*\lf[D\vec{n}+\star_h(\vec{n}\wedge\D^\perp\vec{n})\rg]\quad.
\ee
We replace $-2H\,\D\vec{\Phi}$ in (\ref{VI.67}) by the expression given by (\ref{VI.66}), moreover we also use
the expression of the mean curvature vector in terms of $\vec{\Phi}$ :
\be
\label{VI.67a}
\D^* D\vec{\Phi}=2e^{2\la}\ \vec{H}\quad.
\ee
So (\ref{VI.67}) becomes
\be
\label{VI.68}
\begin{array}{l}
\ds-4H^2\ \vec{H}\ e^{2\la}+\nabla H\cdot\lf[\D\vec{n}+\star_h(\vec{n}\wedge\D^\perp\vec{n}\rg]\\[5mm]
\ds\quad\quad= H\,D^*\lf[\D\vec{n}+\star_h(\vec{n}\wedge\D^\perp\vec{n})\rg]\quad.
\end{array}
\ee
By the Gauss equations, called $K^g$ the Gauss curvature of $\Sigma$ and $K^h=K^h(\vec{\Phi}_*(T\Sigma))$ the sectional curvature of $(M,h)$ evaluated on the tangent plane to $\vec{\Phi}(\Sigma)$ we have
$$(K^g-K^h)\vec{n}=- \frac{1}{2}\star_h(D\vec{n}\wedge D^\perp \vec{n}) e^{-2\lambda}.$$
Using that the Hodge duality $\star_h$ commutes with the covariant differentiation $\D$ we get 
$$\D^*[ \star_h (\vec{n} \wedge \D ^\perp \vec{n})] = \star_h( \D \vec{n} \wedge \D ^\perp \vec{n}) + \star_h [\vec{n}\wedge Riem^h(\p_{x_2}\vec{\Phi}, \p_{x_1}\vec{\Phi}) \vec{n}] $$
where we use the convention that $Riem^h(\vec{X}, \vec{Y}) \vec{Z}:= \D_{\vec{X}} \vec{Y}-\D_{\vec{Y}} \vec{X}-\D_{[\vec{X}, \vec{Y}]} \vec{Z}$; putting together the last two equations we obtain
\be \label{VI.69}
\D^*[ \star_h (\vec{n} \wedge \D ^\perp \vec{n})] =-2 e^{2\lambda}(K^g-K^h)\vec{n} + \star_h [\vec{n}\wedge Riem^h(\p_{x_2}\vec{\Phi}, \p_{x_1}\vec{\Phi}) \vec{n}]. 
\ee
Computing (\ref{VI.68})$-2H$(\ref{VI.69}) we get
\be
\label{VI.70}
\begin{array}{l}
-4 e^{2\la}\ \vec{H}\ (H^2-(K^g-K^h))+\nabla H\cdot\lf[\D\vec{n}+\star_h(\vec{n}\wedge\D^\perp\vec{n})\rg]+\\[5mm]
\ds\quad-2\star_h\left( \vec{n}\wedge Riem^h(\p_{x_2}\vec{\Phi},\p_{x_1}\vec{\Phi})\vec{H} \right)=  H\,D^*\lf[\D\vec{n}-\star_h(\vec{n}\wedge\D^\perp\vec{n})\rg]\quad.
\end{array}
\ee
Since $D^*(\nabla H \vec{n})= e^{2 \lambda} \Delta_g H \vec{n}+ \nabla H \cdot  \D\vec{n}$ we have
\be
\label{VI.70b}
\begin{array}{l}
H D^*(D\vec{n}-\star_h(\vec{n}\wedge \D^\perp \vec{n}))= 2 e^{2 \lambda} \Delta_g H \vec{n}+ \nabla H \star_h (\vec{n}\wedge D^\perp \vec{n})+\\[5mm]
\ds \qquad +D^*[-2\nabla H \vec{n}+H \D \vec{n} - H\star_h (\vec{n}\wedge \D^\perp \vec{n})]+ \nabla H \D\vec{n}\quad.
\end{array}
\ee
Plugging (\ref{VI.70b}) into (\ref{VI.70}) we obtain
\be
\label{VI.71a}
\begin{array}{l}
-4 e^{2\la}\ \vec{H}\ (H^2-(K^g-K^h)) -2\star_h\left( \vec{n}\wedge Riem^h(\p_{x_2}\vec{\Phi},\p_{x_1}\vec{\Phi})\vec{H} \right) +\\[5mm]
\ds\quad-2 e^{2 \la} \Delta_g H\ \vec{n}= D^*\left[-2\nabla H \vec{n}+H \D \vec{n} - H\star_h (\vec{n}\wedge \D^\perp \vec{n})\right]\quad.
\end{array}
\ee
Now observe that 
\[
\lf\{
\begin{array}{l}
\ds <\star_h\left( \vec{n}\wedge Riem^h(\p_{x_2}\vec{\Phi},\p_{x_1}\vec{\Phi})\vec{H} \right),\vec{e_1}>= - <Riem^h(\vec{e_2}, \vec{e_1})\vec{H}, \vec{e_2}>\\[5mm]
\ds <\star_h\left( \vec{n}\wedge Riem^h(\p_{x_2}\vec{\Phi},\p_{x_1}\vec{\Phi})\vec{H} \right),\vec{e_2}>=  <Riem^h(\vec{e_2}, \vec{e_1})\vec{H}, \vec{e_1}>
\end{array}
\rg.
\]
and the normal component is null; hence 
$$ \star_h\left( \vec{n}\wedge Riem^h(\p_{x_2}\vec{\Phi},\p_{x_1}\vec{\Phi})\vec{H} \right)=-e^{2\la}(Riem(\vec{e_1},\vec{e_2})\vec{H})^\perp=:-e^{2\la} R^\perp_{\vec{\Phi}}(T\vec{\Phi})$$
where $\cdot^\perp$ denotes the rotation in the plane $\vec{\Phi}_*(T\Sigma)$ of $\frac \pi 2$ in the sense from $\vec{e_1}$ to $\vec{e_2}$.
Therefore we finally write the relation (\ref{VI.71a}) as

\be
\label{VI.71}
\begin{array}{l}
-2 e^{2 \la} \Delta_g H\ \vec{n}-4 e^{2\la}\ \vec{H}\ (H^2-(K^g-K^h)) +2e^{2\la} R^\perp _\Phi (T\vec{\Phi}) \\[5mm]
\ds \qquad \quad= D^*\left[-2\nabla H \vec{n}+H \D \vec{n} - H\star_h (\vec{n}\wedge \D^\perp \vec{n})\right]\quad.
\end{array}
\ee
Now recall that the immersion $\vec{\Phi}$ is Willmore if and only if 
$$ (\Delta_g H) \vec{n}+2 \vec{H}(H^2-(K^g-K^h))+\vec{H} Ric_h(\vec{n},\vec{n})=0$$
so, using equation (\ref{VI.71}),   $\vec{\Phi}$ is a Willmore immersion if and only if 

\be
\label{VI.71b}
\begin{array}{l}
2 e^{2\la}  R^\perp _\Phi (T\vec{\Phi}) +2 e^{2\la}\vec{H} Ric_h (\vec{n}, \vec{n}) \\[5mm]
\ds \qquad \quad= D^*\left[-2\nabla H \vec{n}+H \D \vec{n} - H\star_h (\vec{n}\wedge \D^\perp \vec{n})\right]\quad.
\end{array}
\ee
Observing that $\D\vec{H}=\D(H\vec{n})=H\D\vec{n}+\nabla H  \vec{n}$ we can rewrite the last relation as
\be
\label{VI.71c}
\begin{array}{l}
2 e^{2\la}  [R^\perp _\Phi (T\vec{\Phi}) + \vec{H} Ric_h (\vec{n}, \vec{n})] \\[5mm]
\ds \qquad \quad= D^*\left[-2 D\vec{H} + 3 H \D \vec{n} - \star_h (\vec{H}\wedge \D^\perp \vec{n})\right]\quad
\end{array}
\ee
which is the desired identity.

In \cite{BPP} is derived the Willmore equation under conformal constraint for immersions of surfaces in a 3-dimensional Riemannian manifold; by Proposition 2 of the aforementioned paper,  $\vec{\Phi}$ is a conformal constrained Willmore immersion if and only if there exists an holomorphic quadratic differential $q\in H^0(K^2)$ such that $W'(\vec{\Phi})=\delta^*(q)$, which is equivalent  to ask that there exists an holomorphic function $f(z)$ such that $W'(\vec{\Phi})=e^{-2\lambda}\Im (f(z) \overline{\vec{H_0}})$. Hence $\vec{\Phi}$ is conformal constrained Willmore if and only if 
\be\label{eq:ProvConfWill}
- (\Delta_g H) \vec{n}-2 \vec{H}(H^2-(K^g-K^h))-\vec{H} Ric_h(\vec{n},\vec{n})= e^{-2\lambda} \Im (f(z) \overline{\vec{H_0}})
\ee
for some holomorphic function $f(z)$; we conclude using the relation \eqref{VI.71}.  
\hfill$\Box$

\section{Conservative form of the Willmore equation in manifold in arbitrary codimension}\label{Sec:ConsWillHCD}
\reset
Let us start introducing some notation. Let $\bP$ be a smooth immersion of the disc $D^2$ into a Riemannian manifold $(M^m,h)$ of dimension $m\geq 3$. We stress that at this point $\bP$ is \emph{ not } assumed to be \emph{conformal}. Let us denote with $g=g_{\bP}:=\bP^{\star} h$ the pull back metric on $D^2$ by $\bP$. Call $\star_h$ and $*_g$ the Hodge duality operators, defined in \eqref{eq:defHodge} for $p$-vectors tangent respectively to $M$ and to $D^2$. Consider a positively oriented orthormal frame $\bbf_1, \bbf_2$ of $T\D^2$ endowed with the metric $g$ and let $\bbe_1:=\bP_*(\bbf_1),\bbe_2:=\bP_*(\bbf_2)$ the corresponding orthonormal frame of $\bP_*(TD^2)$, called $D$ the covariant derivative in $(M,h)$ we define
\begin{eqnarray}
D_g: \Gamma_{D^2}(T_{\bP}M\otimes \wedge^p TD^2) &\to & \Gamma_{D^2}(T_{\bP}M\otimes \wedge^{p+1} TD^2) \nonumber \\
\bX\otimes \vec{\alpha} &\mapsto & D_{\bbe_1} \bX \otimes (\bbf_1 \wedge \vec{\alpha}) + D_{\bbe_2} \bX \otimes (\bbf_2 \wedge \vec{\alpha}) \nonumber \\
&& = g^{ij} \left[ D_{\p_{x^i}\bP} \bX \otimes \left(\frac{\p}{\p x^j} \wedge \vec{\alpha} \right) \right], \label{def:Dg}
\end{eqnarray} 
where, in the last line we used coordinates $(x^1,x^2)$ on $D^2$, $\Gamma_{D^2}$ denotes the set of the sections of the cooresponding bundle, and  $T_{\bP}M$ is the tangent bundle of $M$ along $\bP(D^2)$. Notice that the definition does not depend on the choice of coordinates choosed on $D^2$, i.e. it is intrinsic. Observe we defined $D_g$ on a generating family, so the defintion extends to the whole space.

Next extend the definition of $*_g$ to $T_{\bP}M \otimes \wedge^p TD^2$  as 
\begin{eqnarray}
*_g: T_{\bP}M \otimes \wedge^p TD^2 &\to & T_{\bP}M \otimes \wedge^{(2-p)}\; TD^2\nonumber \\
\bX\otimes \vec{\alpha} &\mapsto & \bX\otimes ( *_g \vec{\alpha} ) \label{def:*g}.
\end{eqnarray} 
Using \eqref{def:Dg} and \eqref{def:*g} above let us define
\be\label{def:D*}
D_g^{*_g}:=(-1) *_g D_g *_g.
\ee
We also need to extend the definitions of $\star_h$, $\wedge_M$, scalar product and the projection $\pi_{\bn}$ onto the normal space to $\bP$  as follows
\begin{eqnarray}
\star_h: \wedge^p T_{\bP}M \otimes \wedge^q TD^2 &\to & \wedge^{(m-p)}\; T_{\bP}M \otimes \wedge^q TD^2 \nonumber \\
 \vec{\eta} \otimes \vec{\alpha}  &\mapsto& (\star_h \vec{\eta}) \otimes \vec{\alpha} \label{def:starh}\\
\wedge_M:\wedge^p T_{\bP}M \otimes \wedge^q TD^2 \;\times \; \wedge^s T_{\bP}M &\to& \wedge^{p+s} T_{\bP}M \otimes \wedge^q TD^2 \nonumber \\
(\vec{\eta} \otimes \vec{\alpha}, \vec{\tau})&\mapsto& (\vec{\eta} \wedge_M \vec{\tau}) \otimes \vec{\alpha}. \label{def:wedgeM} \\
<.,.>: (\wedge^p T_{\bP}M \otimes \wedge^q TD^2)&\times& (\wedge^p T_{\bP}M \otimes \wedge^q TD^2)  \to \R \nonumber\\
 (\vec{\eta} \otimes \vec{\alpha}, \vec{\tau} \otimes \vec{\beta})&\mapsto& <\vec{\eta},\vec{\tau}>_h <\vec{\alpha},\vec{\beta}>_g \label{def:ScalProd} \\
\pi_{\bn}: T_{\bP}M \otimes \wedge^q TD^2 &\to&  T_{\bP}M \otimes \wedge^q TD^2 \nonumber \\
\bX\otimes \vec{\alpha}&\mapsto& (\pi_{\bn}(\bX))\otimes \vec{\alpha} \label{def:pin}
\end{eqnarray}


Define also ${\frak R}_{\bP}(T\bP)$ to be
\be\label{def:frakR}
{\frak R}_{\bP}(T\bP):=\sum_{j=1}^2\left( <\Riem^h({\mathbb I}_{1j}, \bbe_2 ) \bbe_1, \bbe_2> \bbe_j+ <\Riem^h(\bbe_1,{\mathbb I}_{2j}) \bbe_1, \bbe_2> \bbe_j  \right)
\ee

and 
\be\label{def:DR}
(D \,R)(T\bP):=\sum_{i=1}^m <(D_{\vec{E}_i} \Riem^h)(\bbe_1,\bbe_2)\bbe_1, \bbe_2> \vec{E}_i.
\ee

The goal of this  section is to prove Theorem \ref{thm:ConsGen}. Observe that, now, all the terms appearing in the statement have been defined. Leu us summarize the arguments of the proof.

{\bf Proof of Theorem~\ref{thm:ConsGen}.}
The proof is almost all the subsection below; more precisely it follows combining \eqref{eq:Cons} of Theorem \ref{th-VI.8} with \eqref{VI.98} in Remark \ref{rm-VI.11}; indeed with some straitforward computation following the defintions \eqref{def:Dg}, \eqref{def:*g}, \eqref{def:D*}, \eqref{def:starh}, \eqref{def:wedgeM} and \eqref{def:pin} one checks that the left hand side of \eqref{eq:ConsWillGen} and the left hand side of \eqref{VI.98} coincide.
\hfill$\Box$
\\

Before passing to the proof of the corollaries, some comments have to be done.
\\In order to exploit analytically equation (\ref{eq:ConsWillGen}) we will need a more explicit expression of $\pi_{\vec{n}}(D_g \bH)$. Recall the definition of  $\res$ given in \eqref{eq:defllcorner}, let as before $(\vec{e}_1,\vec{e}_2)$ be an orthonormal basis of $\bP_*(TD^2)$, call $\vec{n}$ the orthogonal $m-2$ plane given by 
$$
\star_h(\vec{e}_1\wedge\vec{e}_2)=\vec{n}
$$
and let $(\vec{n}_1\cdots\vec{n}_{m-2})$ be a positively oriented orthonormal basis of the $m-2$-plane given by $\vec{n}$
satisfying $\vec{n}=\wedge_{\al}\vec{n}_\al$. One verifies easily that
$$
\lf\{
\begin{array}{l}
\vec{n}\res\vec{e}_i=0\\[5mm]
\vec{n}\res\vec{n}_\al= (-1)^{\al-1}\wedge_{\beta\ne\al}\vec{n}_\beta\\[5mm]
\vec{n}\res(\wedge_{\beta\ne\al}\vec{n}_\beta)=(-1)^{m+\al-2}\,\vec{n}_\al
\end{array}
\rg.
$$

We then deduce the following identity :
\be
\label{VI.65ab}
\forall \vec{w}\in T_{\bP(x)}M\quad\quad \pi_{\vec{n}}(\vec{w})=(-1)^{m-1}\ \vec{n}\res(\vec{n}\res\vec{w})
\ee
From (\ref{VI.65ab}) we deduce in particular
\be
\label{VI.65ac}
\pi_{\vec{n}}(D_g\vec{H})=D_g\vec{H}-(-1)^{m-1}\ D_g(\vec{n})\res_M(\vec{n}\res \bH)-(-1)^{m-1}\ \vec{n}\res_M(D_g(\vec{n})\res_M\vec{H});
\ee
where, analogously as before, we define
\begin{eqnarray}
\res_M: (\wedge^p T_{\bP}M \otimes TD^2) \times \wedge^q T_{\bP}M &\to & \wedge^{(p-q)}\; T_{\bP}M \otimes  TD^2 \nonumber \\
 (\vec{\alpha} \otimes \vec{v},\vec{\beta})   &\mapsto& (\vec{\alpha} \otimes \vec{v})\res_M \vec{\beta}:=(\vec{\alpha}\res \vec{\beta}) \otimes \vec{v}. \label{eq:defres}
\end{eqnarray}
\\


A straightforward but important consequence of Theorem~\ref{thm:ConsGen} is the conservative form of Willmore surfaces equations given in Corollary ~\ref{co:ConsWillGen}. Let us prove it.
\\

{\bf Proof of Corollary~\ref{co:ConsWillGen}.}
Recall that the first variation of the  the Willmore functional in general  riemannian manifolds has been computed in \cite{Wei}; equating it to zero we get the classical Willmore equation in manifolds: $\vec{\Phi}$ is a Willmore immersion if and only if
\be \label{eq:W'HC}
\Delta_\perp\bH+\ti{A}(\vec{H})-2 |\bH|^2\ \bH-\ti{R}(\bH)=0
\ee
where $\ti{R}$ is the curvature endomorphism  defined in (\ref{eq:defR}). 
Collecting \eqref{eq:W'HC} and the equation \eqref{eq:ConsWillGen} we get the thesis.
\hfill$\Box$
\\

Recall that an immersion $\vec{\Phi}$ of $\Sigma$ is said to be \emph{constrained-conformal  Willmore} if and only if it is a critical point of the Willmore functional under the constraint that the conformal class is fixed. Let us prove the conservative form of the constraint-conformal Willmore surface equation given in Corollary \ref{co:ConsConfW}.
\\

{\bf Proof of Corollary~\ref{co:ConsConfW}.}
Recall (see \cite{Riv2}, and the notation given in the introduction before Corollary~\ref{co:ConsConfW} ) that an immersion $\vec{\Phi}$ is a  constrained-conformal Willmore immersion if and only if there exists an holomorphic quadratic differential $q\in Q(J)$ such that 
\be \label{eq:W'confHC}
\Delta_\perp\bH+\ti{A}(\vec{H})-2 |\bH|^2\ \bH-\ti{R}(\bH)=\Im[(q,\vec{h}_0)_{WP}]
\ee
where $\ti{R}$ is the curvature endomorphism  defined in (\ref{eq:defR}). The thesis follows  putting together (\ref{eq:W'confHC}) and the equation (\ref{eq:Cons}). 
\hfill$\Box$
\\

Now we prove that also the Euler Lagrange equation of the functional $F=\frac{1}{2}\int |{\mathbb I}|^2$ ca be written in conservative form, being the Euler Lagrange equation of $W$ plus some lower order terms. Let us start with an auxiliary lemma.

\begin{Lm}\label{lem:VarK}
Let $\bP: D^2 \hookrightarrow (M,h)$ be  a smooth immersion, then the first variation of the functional $\int_{D^2} \bar{K} (\bP_*(TD^2)) dvol_g$ with respect to a smooth compactly supported variation $\bw$ is given
\be\label{eq:VarK}
\begin{array}{l}
\ds\frac{d}{dt}\int_{D^2} \bar{K} ((\bP+t\bw)_*(TD^2)) dvol_{g_{\bP+t\bw}}\ (t=0)\\[5mm]
\ds=-\int_{D^2} <(D\,R)(T\bP)+2 {\frak R}_{\bP}(T\bP)+2 \bar{K}(\bP_*(TD^2)) \bH, \bw>  dvol_g\quad.
\end{array}
\ee
where $(D\,R)(T\bP)$ and ${\frak R}_{\bP}(T\bP)$ are the curvature quantities defined respectively in \eqref{def:DR} and \ref{def:frakR}.
\hfill $\Box$
\end{Lm}

\begin{proof} 
Let$(\bbe_1,\bbe_2)$ is an orthonormal frame of $\bP_*(TD^2)$ extended in the neighbourood of $\bP(D^2)$ by parallel translation in the normal directions, and $\pi_T$ denotes the projection on $\bP_*(TD^2)$
By definition  $\bar{K}(\bP_*(TD^2))=-<\Riem^h(\bbe_1,\bbe_2)\bbe_1,\bbe_2>$. Observe that using the orthonormality of $(\bbe_1,\bbe_2)$, the  antisymmetry of $\Riem^h(.,.)$ and the fact that $D_{\pi_{\bn}(\bw)} \bbe_i=0$ we get 
$$\Riem^h(D_{\bw} \bbe_1,\bbe_2)= \Riem^h({\mathbb I}(\pi_T(\bw),\bbe_1),\bbe_2);$$ 
recall moreover that the first variation of the volume element is $-2<\bH,\bw> vol_g$. Collecting these informations and using the symmetry of the Riemann tensor one gets 
\begin{eqnarray}
&&-\int_{D^2}\Big[<(D_{\bw}Riem^h)(\bbe_1,\bbe_2)\bbe_1,\bbe_2>+ 2 <\Riem^h({\mathbb I}(\pi_T(\bw),\bbe_1),\bbe_2) \bbe_1, \bbe_2> \nonumber \\
&&\quad \quad  \quad + 2 <\Riem^h(\bbe_1,{\mathbb I}(\pi_T(\bw),\bbe_2)) \bbe_1, \bbe_2>+ 2 \bar{K}(\bP_*(TD^2)) <\bH,\bw> \Big] dvol_g.
\end{eqnarray}
Now the thesis follows recalling the definitions \eqref{def:DR} and \ref{def:frakR}.
\end{proof}

\begin{Co} \label{co:ConsFGen}  
A smooth immersion $\vec{\Phi}$ of a 2-dimensional disc $D^2$ in $(M^m,h)$ is critical for the functional $F=\frac{1}{2}\int |{\mathbb I}|^2$ if and only if 
\be
\label{eq:ConsFGen}
\begin{array}{l}
\ds D^{*_g}_g \left[D_g \bH -3 \pi_{\bn} (D_g \bH) + \star_h\left((*_g D_g \bn) \wedge_M \bH \right)  \right] \\[5mm]
\quad=2\ti{R}(\bH)-2R^\perp_{\vec{\Phi}}(T\vec{\Phi})+(D\,R)(T\bP)+2 {\frak R}_{\bP}(T\bP)+2 \bar{K}(\bP_*(TD^2)) \bH\quad,
\end{array}
\ee
where $\ti{R}$ and $R^\perp$ are the curvature endomorphisms defined respectively in \eqref{eq:defR} and \eqref{Rperp}.\hfill $\Box$
\end{Co}

\begin{proof}
The Gauss equation yelds
\be\label{eq:HII}
\frac{1}{2} |{\mathbb I}|^2=2 |\bH|^2+ \bar{K}(\bP_*(TD^2))-K_{\bP}
\ee 
where $K_{\bP}$ is the Gauss curvature of the metric $g=\bP^*h$. Integrating over $D^2$, we get
$$F(\bP)=2 W(\bP)+ \int_{D^2} \bar{K}(\bP_*(TD^2))- \int_{D^2}K_{\bP} dvol_g.$$
Since by Gauss-Bonnet theorem the last integral reduces, up to an additive constant, to an integral on the boundary, if we take a variations $\bw$ compactly supported in $D^2$ it gives no contribution in the first variation. Therefore the thesis follows combining the first variation of $W$ given in Corollary \ref{co:ConsWillGen} and Lemma \ref{lem:VarK}. 
\end{proof}

\begin{Co} \label{co:ConsWcnfGen}  
A smooth immersion $\vec{\Phi}$ of a 2-dimensional disc $D^2$ in $(M^m,h)$ is conformal Willmore (i.e. critical for the  conformal Willmore functional $W_{conf}=\int ( |\bH|^2+\bar{K} ) dvol_g$) if and only if 
\be
\label{eq:ConsWcnfGen}
\begin{array}{l}
\ds\frac{1}{2} D^{*_g}_g \left[D_g \bH -3 \pi_{\bn} (D_g \bH) + \star_h\left((*_g D_g \bn) \wedge_M \bH \right)  \right]\\[5mm]
\quad\ds=\ti{R}(\bH)-R^\perp_{\vec{\Phi}}(T\vec{\Phi})+(D\,R)(T\bP)+2 {\frak R}_{\bP}(T\bP)+2 \bar{K}(\bP_*(TD^2)) \bH\quad.
\end{array}
\ee
Notice if $(M,h)$ has constant sectional curvature $\bar{K}$ then the right hand side is null and we get
\be
\label{eq:ConsWcnfGenKconst}
\frac{1}{2} D^{*_g}_g \left[D_g \bH -3 \pi_{\bn} (D_g \bH) + \star_h\left((*_g D_g \bn) \wedge_M \bH \right)  \right]=0\quad.
\ee
\hfill $\Box$
\end{Co}

\begin{proof}
The proof of \eqref{eq:ConsWcnfGen} follows combining Corollary \ref{co:ConsWillGen} and Lemma \ref{lem:VarK}. 

Now assume that the sectional curvature $\bar{K}$ is constant; then observe that $W_{conf}(\bP)=W(\bP)+\bar{K} A(\bP)$,
\be\label{WconfKconst}
dW_{conf}=dW-2\bar{K} \bH.
\ee
Moreover, $\bar{K}$ constant implies that (see \cite{DoC} Corollary 3.5 and recall the opposite sign convention in the Riemann tensor)
\be\label{eq:RiemKconst}
<\Riem^h(\vec{X},\vec{Y})\vec{W}, \vec{Z}>=h(\vec{X},\vec{Z}) h(\vec{Y},\vec{W})-h(\vec{X},\vec{W}) h(\vec{Y},\vec{Z}) \quad \forall \vec{X},\vec{Y},\vec{W}, \vec{Z} \in T_xM.
\ee 
Therefore, plugging \eqref{eq:RiemKconst} directly into the definitions \eqref{eq:defR} and \eqref{Rperp} we get 
\begin{eqnarray}
\tilde{R}(\bH)&=&-2\bar{K}\bH, \label{eq:R=-2barK}\\
 R^\perp_{\vec{\Phi}}(T\vec{\Phi})&=&0\quad . \label{eq:Rperp=0}
\end{eqnarray}
Equation \eqref{eq:ConsWcnfGenKconst} follows combining \eqref{WconfKconst}, \eqref{eq:R=-2barK}, \eqref{eq:Rperp=0} and Corollary \ref{co:ConsWillGen}.
\end{proof}

\begin{Co}\label{co:ConsConfWcnfGen}
Let $\vec{\Phi}:\Sigma \hookrightarrow M$ be a  smooth immersion into the $m\geq 3$-dimensional Riemannian manifold $(M^m,h)$ and call $J$ the complex structure associated to $g=\bP^*h$.  Then the immersion $\bP$ is constraint-conformal conformal Willmore (i.e. critical for the  conformal Willmore functional $W_{conf}=\int ( |\bH|^2+\bar{K} ) dvol_g$ under the constraint of fixed conformal class) if and only if there exists an holomorphic quadratic differential $q \in Q(J)$ such that
\be
\label{eq:ConsWcnfGen}
\begin{array}{l}
\ds\frac{1}{2} D^{*_g}_g \left[D_g \bH -3 \pi_{\bn} (D_g \bH) + \star_h\left((*_g D_g \bn) \wedge_M \bH \right)  \right]\\[5mm]
\ds\quad=\Im\left[(q,\vec{h}_0)_{WP}\right]+\ti{R}(\bH)-R^\perp_{\vec{\Phi}}(T\vec{\Phi}) +(D\,R)(T\bP)+2 {\frak R}_{\bP}(T\bP)+2 \bar{K}(\bP_*(TD^2)) \bH\quad.
\end{array}
\ee
Where $\bH_0, \vec{h}_0$ and $(.,.)_{WP}$ are defined in \eqref{def:vecH0}, \eqref{def:vech0} and \eqref{def:WPpunct}.

Notice that if $(M,h)$ has constant sectional curvature $\bar{K}$ then the curvature terms of the right hand side vanish and we get
\be
\label{eq:ConsWcnfGenKconst}
\frac{1}{2} D^{*_g}_g \left[D_g \bH -3 \pi_{\bn} (D_g \bH) + \star_h\left((*_g D_g \bn) \wedge_M \bH \right)  \right]=\Im\left[(q,\vec{h}_0)_{WP}\right]\quad.
\ee
\hfill $\Box$
\end{Co}

\begin{proof}
The proos is analogous to the proof of Corollary \ref{co:ConsConfW} once we have Corollary \ref{co:ConsWcnfGen}.
\end{proof}

\subsection{Derivation of the conservative form: use of conformal coordinates and complex notation}
We first introduce some complex notation
that will be useful in the sequel. In this subsection $\vec{\Phi}$ is a conformal immersion into a Riemannian manifold $(M^m,h)$ of dimension $m\geq 3$,  denote
$z=x_1+ix_2$, $\p_z=2^{-1}(\p_{x_1}-i\p_{x_2})$, $\p_{\ov{z}}=2^{-1}(\p_{x_1}+i\p_{x_2})$.

\noindent Moreover we denote\footnote{Observe that the notation has been chosen in such a way that $\overline{\vec{e}_z}=\vec{e}_{\overline{z}}$.}
\[
\lf\{
\begin{array}{l}
\ds\vec{e}_{z}:=e^{-\la}\p_z\vec{\Phi}=2^{-1}(\vec{e}_1-i\vec{e}_2)\\[5mm]
\ds\vec{e}_{\ov{z}}:=e^{-\la}\p_{\ov{z}}\vec{\Phi}=2^{-1}(\vec{e}_1+i\vec{e}_2)
\end{array}
\rg.
\]
Observe that
\be
\label{VI.199a}
\lf\{
\begin{array}{l}
\ds\lf<\vec{e}_z,\vec{e}_z\rg>=0\\[5mm]
\ds\lf<\vec{e}_z,\vec{e}_{\ov{z}}\rg>=\frac{1}{2}\\[5mm]
\ds\vec{e}_z\wedge\vec{e}_{\ov{z}}=\frac{i}{2}\,\vec{e}_1\wedge\vec{e}_2
\end{array}
\rg.
\ee
We also use the shorter notation $\D_z:=\D_{\p_z \vec{\Phi}}$ and $\D_{\bar{z}}:=\D_{\p_{\bar{z} \vec{\Phi}}}$ for the covariant derivative with respect to the vectors $\p_z \vec{\Phi}$ and $\p_{\bar{z}} \vec{\Phi}$.
Introduce moreover the {\it Weingarten Operator} expressed in our conformal coordinates $(x_1,x_2)$ :
\[
\vec{H}_0:=\frac{1}{2}\lf[\vec{\mathbb I}(\vec{e}_1,\vec{e}_1)-\vec{\mathbb I}(\vec{e}_2,\vec{e}_2)-2\,i\, \vec{\mathbb I}(\vec{e}_1,\vec{e}_2)\rg]\quad.
\]

\begin{Th}
\label{th-VI.8}
Let $\vec{\Phi}$ be a smooth immersion of a two dimensional manifold $\Sigma^2$ into an $m$-dimensional Riemannian manifold $(M^m,h)$; restricting the immersion to a small disc neighboorod of a point where we consider local conformal coordinate, we can see $\vec{\Phi}$ as a conformal immersion of $D^2$ into $(M,h)$. Then the following identity holds
\be
\label{eq:Cons}
\begin{array}{l}
4\,e^{-2\la}\,\Re\lf(\D_{\ov{z}}\lf[\pi_{\vec{n}}(\D_z\vec{H})+<\vec{H},\vec{H}_0>\ \p_{\ov{z}}\vec{\Phi}\rg]\rg)\\[5mm]
\quad=\Delta_\perp\vec{H}+\ti{A}(\vec{H})-2|\vec{H}|^2\ \vec{H}+8 \Re \lf(<\Riem^h(\vec{e}_{\ov{z}},\vec{e}_z) \vec{e}_z, \vec{H}> \vec{e}_{\ov{z}}\rg)\quad,
\end{array}
\ee
where $\vec{H}$ is the mean curvature vector of the immersion $\vec{\Phi}$, $\Delta_\perp$ is the negative covariant laplacian on the normal bundle to the immersion, $\tilde{A}$ is the linear map given in (\ref{VI.44a}),  $\D_z \cdot:=\D_{\p_z \vec{\Phi}} \cdot$, and $\D_{\ov{z}} \cdot:=\D_{\p_{\ov{z}} \vec{\Phi}} \cdot$ are the covariant derivatives in $(M,h)$ with respect to the tangent vectors $\p_z \vec{\Phi}$ and $\p_{\ov{z}} \vec{\Phi}$.\hfill $\Box$
\end{Th}

\begin{Rm}
\label{rm-VI.11}
Observe that using the identity  (\ref{z-VI.2}) proved in Lemma \ref{lm-VI.11}, the equation (\ref{eq:Cons}) can be written using \emph{real} conformal coordinates as follows
\be
\label{VI.98}
 -\frac{e^{-2\la}}{2}\,\D^*\lf[\D\vec{H}-3\pi_{\vec{n}}(\D \vec{H})+\star_h(\D^\perp\vec{n}\wedge\vec{H})\rg]=\Delta_\perp\bH+\ti{A}(\vec{H})-2 |\bH|^2\ \bH- R^\perp_{\vec{\Phi}}(T\vec{\Phi})\quad,
\ee
observe we used the equation below, which follows by definition \eqref{Rperp}, 
\be\nonumber
R^\perp_{\vec{\Phi}}(T\vec{\Phi})= -8 \Re \lf(<\Riem^h(\vec{e}_{\ov{z}},\vec{e}_z) \vec{e}_z, \vec{H}> \vec{e}_{\ov{z}}\rg)=\left( \pi_T\left[ \Riem^h(\vec{e}_1,\vec{e}_2)\vec{H} \right] \right)^\perp.
\ee
Notice that identity (\ref{VI.98}) in codimension one gives exactly the previous (\ref{VI.71}). \hfill $\Box$
\end{Rm}

A straightforward but important consequence of Theorem~\ref{th-VI.8} is the following conservative form of Willmore surfaces equation in conformal coordinates.
\begin{Co}
\label{co-VI.8HC}  
A conformal immersion $\vec{\Phi}$ of a 2-dimensional disc $D^2$ in $(M^m,h)$ is Willmore if and only if 
\be
\label{eq:ConsW}
4\,e^{-2\la}\,\Re\lf(\D_{\ov{z}}\lf[\pi_{\vec{n}}(\D_z\vec{H})+<\vec{H},\vec{H}_0>\ \p_{\ov{z}}\vec{\Phi}\rg]\rg)=\ti{R}(\bH)+8 \Re \lf(<\Riem^h(\vec{e}_{\ov{z}},\vec{e}_z) \vec{e}_z, \vec{H}> \vec{e}_{\ov{z}}\rg).
\ee
\end{Co}

\begin{proof}
Recall that $\vec{\Phi}$ is a Willmore immersion if and only if \eqref{eq:W'HC} holds. Combining (\ref{eq:W'HC}) and equation (\ref{eq:Cons}) we get the desired result. 
\end{proof}

Now recall that an immersion $\vec{\Phi}$ is said to be \emph{constrained-conformal Willmore} if and only if it is a critical point of the Willmore functional under the constraint that the conformal class is fixed. 


\begin{Co}
\label{co-VI.9HC}  
A conformal immersion $\vec{\Phi}$ of a 2-dimensional disc $D^2$ in $(M^m,h)$ is constrained-conformal  Willmore if and only if there exists an holomorphic function $f(z)$ such that
\be
\label{eq:ConsWconf}
\begin{array}{l}
4\,e^{-2\la}\,\Re\lf(\D_{\ov{z}}\lf[\pi_{\vec{n}}(\D_z\vec{H})+<\vec{H},\vec{H}_0>\ \p_{\ov{z}}\vec{\Phi}\rg]\rg)\\[5mm]
\quad=e^{-2\la}\,\Im (f(z) \overline{\vec{H_0}})+\ti{R}(\bH)+8 \Re \lf(<\Riem^h(\vec{e}_{\ov{z}},\vec{e}_z) \vec{e}_z, \vec{H}> \vec{e}_{\ov{z}}\rg)\quad.
\end{array}
\ee
\end{Co}

{\bf Proof of Corollary~\ref{co-VI.9HC}}
An immersion $\vec{\Phi}$ is a  constrained-conformal Willmore immersion if and only if there exists an holomorphic function $f$ such that 
\be \label{eq:W'confHC}
\Delta_\perp\bH+\ti{A}(\vec{H})-2 |\bH|^2\ \bH-\ti{R}(\bH)=e^{-2\la}\,\Im (f(z) \overline{\vec{H_0}})
\ee
where $\ti{R}$ is the curvature endomorphism  defined in (\ref{eq:defR}). 
\\Therefore putting together (\ref{eq:W'confHC}) and the equation (\ref{eq:Cons}) we get the thesis.
\hfill$\Box$
\\

In order to prove Theorem \ref{th-VI.8} some computational lemmas will be useful; let us start with the following.

\begin{Lm}
\label{lm-VI.11}
Let $\vec{\Phi}$ be a conformal immersion of $D^2$ into $(M^m,h)$ then
\be
\label{z-VI.1}
\pi_T(\D_z\vec{H})-i\,\star_h(\D_z\vec{n}\wedge\vec{H})=-2\,\lf<\vec{H},\vec{H}_0\rg>\ \p_{\ov{z}}\vec{\Phi}
\ee
and hence
\be
\label{z-VI.2}
\D_z\vec{H}-3\pi_{\vec{n}}(\D_z\vec{H})-i\,\star_h(\D_z\vec{n}\wedge\vec{H})=-2\,\lf<\vec{H},\vec{H}_0\rg>\ \p_{\ov{z}}\vec{\Phi}-2\,\pi_{\vec{n}}(\D_z\vec{H})
\ee
\hfill $\Box$
\end{Lm}

{\bf Proof of lemma~\ref{lm-VI.11}.}
We denote by $(\bbe_1,\bbe_2)$ the orthonormal basis of $\vec{\Phi}_\ast(TD^2)$ given by
\[
\bbe_i=e^{-\la}\ \frac{\p \bP}{\p x_i}\quad .
\]
With these notations the second fundamental form $\bh$ which is a symmetric 2-form on $TD^2$  into  $(\vec{\Phi}_\ast TD^2)^\perp$
is given by
\be
\label{VI.74}
\begin{array}{l}
\bh=\sum_{\al,i,j}h^\al_{ij}\ \bn_\al\otimes(\bbe_i)^\ast\otimes(\bbe_j)^\ast\\[5mm]
\mbox{ with }\quad h^\al_{ij}=-e^{-\la}\,\lf(\D_{\p_{x_i} \vec{\Phi}} \bn_\al,\bbe_j\rg)
\end{array}
\ee
We shall also denote
\[
\vec{h}_{ij}:=\vec{\mathbb I}(\bbe_i,\bbe_j)=\sum_{\al=1}^{m-2}h^\al_{ij}\ \vec{n}_\al
\]
In particular the mean curvature vector $\bH$ is given by
\be
\label{VI.75}
\bH=\sum_{\al=1}^{m-2} H^\al\,\bn_\al=\frac{1}{2}\sum_{\al=1}^{m-2}(h^\al_{11}+h^\al_{22})\, \bn_\al=\frac{1}{2}(\vec{h}_{11}+\vec{h}_{22})
\ee
Let $\bn$ be the $m-2$ vector of $T_{\vec{\Phi}(x)}M$ given by $\bn=\bn_1\wedge\cdots\wedge\bn_2$. We identify vectors and $m-1$-vectors in $T_{\vec{\Phi}(x)}M$ using the Hodge operator $\star_h$ for the metric $h$; for the Hodge  operator we use the standard notation (see for example \cite{Nak} Chap. 7.9.2)
$$<\alpha, \beta> \star_h 1=(\alpha \wedge \star_h \beta)$$
for any couple of $p$-vectors $\alpha$ and $\beta$, where we set $\star_h 1:= \bbe_1 \wedge \bbe_2 \wedge \bn$; then we have for instance
\be
\label{VI.76}
\star_h(\bn\wedge \bbe_1)=\bbe_2\quad\mbox{ and }\quad\star_h(\bn\wedge \bbe_2)=- \bbe_1
\ee
Since $\bbe_1,\bbe_2,\bn_1\cdots\bn_{m-2}$ is a basis of $T_{\bP(x)} M$, we can write
for every $\al=1\cdots m-2$
\[
\D \bn_\al=\sum_{\beta=1}^{m-2}<\D \bn_\al,\bn_\beta>\, \bn_\beta+\sum_{i=1}^2<\D\bn_\al,\bbe_i>\,\bbe_i
\]
and consequently
\be
\label{VI.77}
\star_h(\bn\wedge\D^\perp\bn_\al)=<\D^\perp \bn_\al,\bbe_1>\ \bbe_2
-<\D^\perp \bn_\al,\bbe_2>\ \bbe_1
\ee
Hence
\[
\star_h(\D^\perp\bn\wedge\vec{H})=-<\D^\perp \vec{H},\bbe_1>\ \bbe_2
+<\D^\perp \vec{H},\bbe_2>\ \bbe_1 =<\vec{H},\pi_{\vec{n}}(\D^\perp\bbe_1)>\ \bbe_2
-<\vec{H},\pi_{\vec{n}}(\D^\perp\bbe_2)>\ \bbe_1
\]
Using (\ref{VI.74}), we then have proved
\be
\label{z-VI.3}
\star_h(\D^\perp\bn\wedge\vec{H})=
\lf(
\begin{array}{c}
\ds-<\vec{H},\vec{h}_{12}>\ \p_{x_2}\vec{\Phi}\, +\,<\vec{H},\vec{h}_{22}>\ \p_{x_1}\vec{\Phi}\\[5mm]
\ds<\vec{H},\vec{h}_{11}>\ \p_{x_2}\vec{\Phi} \,-\,<\vec{H},\vec{h}_{12}>\ \p_{x_1}\vec{\Phi}
\end{array}
\rg)
\ee
The tangential projection of $\D\vec{H}$ is given by
\[
\begin{array}{l}
\pi_T(\D\vec{H})=<\D\vec{H},\vec{e}_1>\ \bbe_1+<\D\vec{H},\vec{e}_2>\ \bbe_2\\[5mm]
\quad=-<\vec{H},\pi_{\vec{n}}(\D \vec{e}_1)>\ \bbe_1-<\vec{H},\pi_{\vec{n}}(\D \vec{e}_2)>\ \bbe_2\quad.
\end{array}
\]
Hence
\be
\label{z-VI.4}
\pi_T(\D\vec{H})=\lf(
\begin{array}{c}
-<\vec{H},\vec{h}_{11}>\ \p_{x_1}\vec{\Phi}-<\vec{H},\vec{h}_{12}>\ \p_{x_2}\vec{\Phi}\\[5mm]
-<\vec{H},\vec{h}_{12}>\ \p_{x_1}\vec{\Phi}-<\vec{H},\vec{h}_{22}>\ \p_{x_2}\vec{\Phi}
\end{array}
\rg)
\ee
Combining (\ref{z-VI.3}) and (\ref{z-VI.4}) gives
\be
\label{z-VI.5}
-\pi_T(\D\vec{H})-\star_h(\D^\perp\bn\wedge\vec{H})=
\lf(
\begin{array}{c}
<\vec{H},\vec{h}_{11}-\vec{h}_{22}>\ \p_{x_1}\vec{\Phi}+2<\vec{H},\vec{h}_{12}>\ \p_{x_2}\vec{\Phi}\\[5mm]
2<\vec{H},\vec{h}_{12}>\ \p_{x_1}\vec{\Phi}+<\vec{H},\vec{h}_{22}-\vec{h}_{11}>\ \p_{x_2}\vec{\Phi}
\end{array}
\rg)
\ee
This last identity written with the complex coordinate $z$ is exactly (\ref{z-VI.1}) and lemma~\ref{lm-VI.11} is proved. \hfill $\Box$

\medskip

Before to move to the proof of Theorem~\ref{th-VI.8} we shall need two more lemmas. First we have
\begin{Lm}
\label{z-lm-VI.11}
Let $\vec{\Phi}$ be a conformal immersion of the disc $D^2$ into $M^m$, called $z:=x_1+ix_2$, $e^\la:=|\p_{x_1}\vec{\Phi}|=|\p_{x_2}\vec{\Phi}|$
denote
\be
\label{z-VI.200}
\vec{e}_i:=e^{-\la}\,\p_{x_i}\vec{\Phi}\quad,
\ee
and let $\vec{H}_0$ be the Weingarten Operator of the immersion expressed in the conformal coordinates $(x_1,x_2)$:
\[
\vec{H}_0:=\frac{1}{2}\lf[{\mathbb I}(\vec{e}_1,\vec{e}_1)-{\mathbb I}(\vec{e}_2,\vec{e}_2)-2\,i\, {\mathbb I}(\vec{e}_1,\vec{e}_2)\rg].
\]
Then the following identities hold
\be
\label{VI.200}
\D_{\ov{z}}\lf[e^\la\, \vec{e}_{z}\rg]=\frac{e^{2\la}}{2}\vec{H}\quad,
\ee
and
\be
\label{VI.204}
\D_{z}\lf[e^{-\la}\vec{e}_z\rg]=\frac{1}{2}\, \vec{H}_0\quad.
\ee
\hfill$\Box$
\end{Lm}
\noindent{\bf Proof of lemma~\ref{z-lm-VI.11}.}
The first identity (\ref{VI.200}) comes simply from the fact that $\D_{\ov{z}}\p_z\vec{\Phi}=\frac{1}{4}\Delta\vec{\Phi}$, from  (\ref{z-VI.200}) and
the expression of the mean curvature vector in conformal coordinates 
\[
\vec{H}=\frac{e^{-2\la}}{2}\,\Delta\vec{\Phi}\quad.
\]
It remains to prove the identity (\ref{VI.204}).
One has moreover
\be
\label{VI.201}
\D_{z}\lf[e^\la\vec{e}_z\rg]=\D_{z}\p_{z}\vec{\Phi}=\frac{1}{4}\lf[\D_{\p_{x_1}\vec{\Phi}} \p_{x_1}\vec{\Phi}-\D_{\p_{x_2}\vec{\Phi}} \p_{x_2}\vec{\Phi}-2\, i\ \D_{\p_{x_1}\vec{\Phi}} \p_{x_2}\vec{\Phi}\rg]\quad.
\ee
In one hand the projection into the normal direction gives
\be
\label{VI.202}
\pi_{\vec{n}}\lf[\D_{\p_{x_1}\vec{\Phi}} \p_{x_1}\vec{\Phi}-\D_{\p_{x_2}\vec{\Phi}} \p_{x_2}\vec{\Phi}-2\, i\ \D_{\p_{x_1}\vec{\Phi}} \p_{x_2}\vec{\Phi}\rg]=2\,{e^{2\la}}\, \vec{H}_0\quad.
\ee
In the other hand the projection into the tangent plane gives
\[
\begin{array}{l}
\ds\pi_{T}\lf[\D_{\p_{x_1}\vec{\Phi}} \p_{x_1}\vec{\Phi}-\D_{\p_{x_2}\vec{\Phi}} \p_{x_2}\vec{\Phi}-2\, i\ \D_{\p_{x_1}\vec{\Phi}} \p_{x_2}\vec{\Phi}\rg]\\[5mm]
\ds =e^{-\la}\ \lf<\p_{x_1}\vec{\Phi},\lf[\D_{\p_{x_1}\vec{\Phi}} \p_{x_1}\vec{\Phi}-\D_{\p_{x_2}\vec{\Phi}} \p_{x_2}\vec{\Phi}-2\, i\ \D_{\p_{x_1}\vec{\Phi}} \p_{x_2}\vec{\Phi}\rg]\rg>\ \vec{e}_1\\[5mm]
\ds +\,e^{-\la}\ \lf<\p_{x_2}\vec{\Phi},\lf[\D_{\p_{x_1}\vec{\Phi}} \p_{x_1}\vec{\Phi}-\D_{\p_{x_2}\vec{\Phi}} \p_{x_2}\vec{\Phi}-2\, i\ \D_{\p_{x_1}\vec{\Phi}} \p_{x_2}\vec{\Phi}\rg]\rg>\ \vec{e}_2\quad.
\end{array}
\]
This implies after some computation
\be
\label{VI.203}
\begin{array}{l}
\ds\pi_{T}\lf[\D_{\p_{x_1}\vec{\Phi}} \p_{x_1}\vec{\Phi}-\D_{\p_{x_2}\vec{\Phi}} \p_{x_2}\vec{\Phi}-2\, i\ \D_{\p_{x_1}\vec{\Phi}} \p_{x_2}\vec{\Phi}\rg]\\[5mm]
\ds =2\,e^\la\ \lf[\p_{x_1}\la-i\p_{x_2}\la\rg]\ \vec{e}_1-2\,e^\la\ \lf[\p_{x_2}\la+i\p_{x_1}\la\rg]\ \vec{e}_2\\[5mm]
\ds=8\ \p_{z}e^\la\ \vec{e}_z\quad.
\end{array}
\ee
The combination of (\ref{VI.201}), (\ref{VI.202}) and (\ref{VI.203}) gives
\[
\D_{z}\lf[e^\la\vec{e}_z\rg]=\frac{e^{2\la}}{2}\, \vec{H}_0+2\,\p_{z}e^\la\ \vec{e}_z\quad,
\]
which implies (\ref{VI.204}).\hfill$\Box$

\medskip

The last lemma we shall need in order to prove Theorem~\ref{th-VI.8} is the  Codazzi-Mainardi identity that we recall
and prove below.

\begin{Lm}
\label{z-lm-VI.12}{\bf[Codazzi-Mainardi Identity.]}
Let $\vec{\Phi}$ be a conformal immersion of the disc $D^2$ into $(M^m,h)$, called $z:=x_1+ix_2$, $e^\la:=|\p_{x_1}\vec{\Phi}|=|\p_{x_2}\vec{\Phi}|$
denote
\be
\label{z-VI.200}
\vec{e}_i:=e^{-\la}\,\p_{x_i}\vec{\Phi}\quad,
\ee
and denote $\vec{H}_0$  the Weingarten Operator of the immersion expressed in the conformal coordinates $(x_1,x_2)$:
\[
\vec{H}_0:=\frac{1}{2}\lf[{\mathbb I}(\vec{e}_1,\vec{e}_1)-{\mathbb I}(\vec{e}_2,\vec{e}_2)-2\,i\, {\mathbb I}(\vec{e}_1,\vec{e}_2)\rg].
\]
Then the following identity holds
\be
\label{z-VI.203}
e^{-2\la}\,\p_{\ov{z}}\lf(e^{2\la}\,<\vec{H},\vec{H}_0>\rg)=<\vec{H},\D_{z}\vec{H}>+<\vec{H}_0,\D_{\ov{z}}\vec{H}>+2<\Riem^h(\vec{e}_{\ov{z}},\vec{e}_z)\p_z\vec{\Phi}, \vec{H}>\quad.
\ee
\hfill $\Box$
\end{Lm}
{\bf Proof of lemma~\ref{z-lm-VI.12}.}
Using (\ref{VI.204}) we obtain
\[
<\D_{\ov{z}}\vec{H}_0, \vec{H}>=2\, \lf<\D_{\ov{z}}\lf[\D_z\lf(e^{-2\la}\,\p_z\vec{\Phi}\rg)\rg],\vec{H}\rg>=2\, \lf<\D_{{z}}\lf[\D_{\ov{z}}\lf(e^{-2\la}\,\p_z\vec{\Phi}\rg)\rg],\vec{H}\rg>+2<\Riem^h(\vec{e}_{\ov{z}},\vec{e}_z)\p_z\vec{\Phi}, \vec{H}>\quad.
\]
Thus
\begin{eqnarray}
<\D_{\ov{z}}\vec{H}_0, \vec{H}>&=&-4\,\lf<\D_{{z}}\lf[\p_{\ov{z}}\la\ e^{-2\la}\ \p_z\vec{\Phi}\rg],\vec{H}\rg>+\lf<\D_z\lf[\frac{e^{-2\la}}{2}\ \Delta\vec{\Phi}\rg],\vec{H}\rg>+2<\Riem^h(\vec{e}_{\ov{z}},\vec{e}_z)\p_z\vec{\Phi}, \vec{H}>\nonumber \\
&=&-2\p_{\ov{z}}\la\ \lf<\vec{H}_0,\vec{H}\rg>+\lf<\D_z\vec{H},\vec{H}\rg>+2<\Riem^h(\vec{e}_{\ov{z}},\vec{e}_z)\p_z\vec{\Phi}, \vec{H}>\quad.
\end{eqnarray}
This last identity implies the Codazzi-Mainardi identity (\ref{z-VI.203}) and Lemma~\ref{z-lm-VI.12} is proved.\hfill $\Box$

\medskip

\noindent{\bf Proof of theorem~\ref{th-VI.8}.}
Due to Lemma~\ref{lm-VI.11}, as explained in remark~\ref{rm-VI.11}, it suffices to prove in conformal parametrization the identity
(\ref{z-VI.98}).
First of all we observe that
\be
\label{z-VI.205a}
4\,e^{-2\la}\,\Re\lf(\pi_{\vec{n}}\lf(\D_{\ov{z}}\lf[\pi_{\vec{n}}(\D_z\vec{H})\rg]\rg)\rg)=e^{-2\la}\ \pi_{\vec{n}}\lf(\D^*\lf[\pi_{\vec{n}}(\D\vec{H})\rg]\rg)=\Delta_\perp\vec{H}
\ee
The tangential projection gives
\be
\label{z-VI.205}
4\,e^{-2\la}\,\pi_{T}\lf(\D_{\ov{z}}\lf[\pi_{\vec{n}}(\D_z\vec{H})\rg]\rg) =8\, e^{-2\la}\ \lf<\D_{\ov{z}}(\pi_{\vec{n}}(\D_z\vec{H})),\vec{e}_z\rg>\ \vec{e}_{\ov{z}}+\,8\, e^{-2\la}\ \lf<\D_{\ov{z}}(\pi_{\vec{n}}(\D_z\vec{H})),\vec{e}_{\ov{z}}\rg>\ \vec{e}_{{z}}
\ee
Using the fact that $\vec{e}_z$ and $\vec{e}_{\ov{z}}$ are orthogonal to the normal plane we have in one
hand using (\ref{VI.200})
\be
\label{z-VI.206}
\ds\lf<\D_{\ov{z}}(\pi_{\vec{n}}(\D_z\vec{H})),\vec{e}_z\rg>=-e^{-\la}\,\lf<\pi_{\vec{n}}(\D_z\vec{H}),\D_{\ov{z}}\lf[ e^\la\,\vec{e}_z\rg]\rg>=-\frac{e^{\la}}{2}\,\lf<\D_z\vec{H},\vec{H}\rg>
\ee
and on the other hand using (\ref{VI.204})
\be
\label{z-VI.207}
\lf<\D_{\ov{z}}(\pi_{\vec{n}}(\D_z\vec{H})),\vec{e}_{\ov{z}}\rg>=-e^{\la}\,\lf<\pi_{\vec{n}}(\D_z\vec{H}),\D_{\ov{z}}\lf[ e^{-\la}\,\vec{e}_{\ov{z}}\rg]\rg>=-\frac{e^{\la}}{2}\,\lf<\D_z\vec{H},\ov{\vec{H}_0}\rg>.
\ee
Combining (\ref{z-VI.205}), (\ref{z-VI.206}) and (\ref{z-VI.207}) we obtain
\be
\label{z-VI.208}
4\,e^{-2\la}\,\pi_{T}\lf(\D_{\ov{z}}\lf[\pi_{\vec{n}}(\D_z\vec{H})\rg]\rg)=-4\ e^{-2\la}\lf[\lf<\D_z\vec{H},\vec{H}\rg>\,\p_{\ov{z}}\vec{\Phi}+\lf<\D_z\vec{H},\ov{\vec{H}_0}\rg>\,\p_z\vec{\Phi}\rg].
\ee
Putting (\ref{z-VI.205a}) and (\ref{z-VI.208}) together we obtain
\be
\label{z-VI.209a}
4\,e^{-2\la}\,\Re\lf(\D_{\ov{z}}\lf[\pi_{\vec{n}}(\D_z\vec{H})\rg]\rg)
=\Delta_\perp\vec{H}-4\,e^{-2\la}\Re\lf[\lf[\lf<\D_{{z}}\vec{H},\vec{H}\rg>+\lf<\D_{\ov{z}}\vec{H},{\vec{H}_0}\rg>\rg]\,\p_{\ov{z}}\vec{\Phi}\rg]
\ee
Using Codazzi-Mainardi identity (\ref{z-VI.203}) and using also again identity (\ref{VI.204}), (\ref{z-VI.209a}) becomes
\be
\label{z-VI.209}
\begin{array}{l}
\ds 4\,e^{-2\la}\,\Re\lf(\D_{\ov{z}}\lf[\pi_{\vec{n}}(\D_z\vec{H})+<\vec{H},\vec{H}_0>\ \p_{\ov{z}}\vec{\Phi}\rg]\rg)\\[5mm]
\ds=\Delta_\perp\vec{H}+2\,\Re\lf(\lf<\vec{H},\vec{H}_0\rg>\ \ov{\vec{H}_0}+4 <\Riem^h(\vec{e}_{\ov{z}},\vec{e}_z) \vec{e}_z, \vec{H}> \vec{e}_{\ov{z}} \rg)\quad.
\end{array}
\ee
The definition (\ref{VI.44a}) of $\ti{A}$ gives
\[
\ti{A}(\vec{H})=\sum_{i,j=1}^2<\vec{H},\vec{h}_{ij}>\ \vec{h}_{ij}\quad,
\]
hence a short elementary computation gives
\[
\ti{A}(\vec{H})-2|\vec{H}|^2\ \vec{H}
=2^{-1}\,\lf<\vec{H},{\vec{h}_{11}-\vec{h}_{22}}\rg>\ (\vec{h}_{11}-\vec{h}_{22})+2<\vec{H},\vec{h}_{12}>\ \vec{h}_{12} \quad .
\]
Using $\vec{H}_0$ this expression becomes
\be
\label{z-VI.210}
\ti{A}(\vec{H})-2|\vec{H}|^2\ \vec{H}=2\Re\lf(\lf<\vec{H},\vec{H}_0\rg>\ \ov{\vec{H}_0}\rg)
\ee
Combining (\ref{z-VI.209}) and (\ref{z-VI.210}) gives 
\be
\label{z-VI.210b}
\begin{array}{l}
4\,e^{-2\la}\,\Re\lf(\D_{\ov{z}}\lf[\pi_{\vec{n}}(\D_z\vec{H})+<\vec{H},\vec{H}_0>\ \p_{\ov{z}}\vec{\Phi}\rg]\rg)\\[5mm]
=\Delta_\perp\vec{H}+\ti{A}(\vec{H})-2|\vec{H}|^2\ \vec{H}+8 \Re \lf(<\Riem^h(\vec{e}_{\ov{z}},\vec{e}_z) \vec{e}_z, \vec{H}> \vec{e}_{\ov{z}}\rg)\quad.
\end{array}
\ee
which is the desired equality and Theorem~\ref{th-VI.8} is proved.\hfill$\Box$

\section{Parallel mean curvature Vs constraint-conformal conformal Willmore surfaces}\label{Sec:PMC-CCCW}
\reset
As an application of the Conservative form of the Willmore equation, in  this section we  prove the link between parallel mean curvature surfaces and constraint-conformal conformal Willmore surfaces mentioned in the introduction; notice that Proposition ~\ref{Prop:PMC<CCCW} gives a lot of examples of constraint-conformal conformal Willmore surfaces.
\medskip

\noindent{\bf Proof of Proposition~\ref{Prop:PMC<CCCW}.}
Observe that the proof in the Euclidean case was given by the second author in \cite{Riv4}, here we adapt the computations to the Riemannian setting. Up to a change of coordinates, we can assume that $\bP$ is a smooth conformal immersion. Since  $(M^m,h)$ has constant sectional curvature $\bar{K}$, then writing the equation \eqref{eq:ConsWcnfGen} using the conformal parametrization gives that $\bP$ is constraint-conformal conformal Willmore if and only if there exists a holomorphic function $f(z)$ such that (see also \eqref{eq:ConsWconf})
\be\label{eq:CCCWeqConf}
4\,e^{-2\la}\,\Re\lf(\D_{\ov{z}}\lf[\pi_{\vec{n}}(\D_z\vec{H})+<\vec{H},\vec{H}_0>\ \p_{\ov{z}}\vec{\Phi}\rg]\rg)=e^{-2\la}\,\Im (f(z) \overline{\vec{H_0}}).
\ee
Now assume that $\bH$ is parallel, that is $\pi_{\bn}(D_z \bH)=\pi_{\bn}(D_{\bar z} \bH)=0$. From Codazzi-Mainardi identity \eqref{z-VI.203}, observing that the curvature term vanish as showed in the proof of Corollary \ref{co:ConsWcnfGen} (it is nothing but $R^{\perp}_{\bP}(T\bP)$) we obtain
$$e^{-2\la}\,\p_{\ov{z}}\lf(e^{2\la}\,<\vec{H},\vec{H}_0>\rg)=0,$$
therefore $f(z):=e^{2\la}\,<\vec{H},\vec{H}_0>$ is holomorphic. Since by assumption $\pi_{\vec{n}}(\D_z\vec{H})=0$, we can write the left hand side of \eqref{eq:CCCWeqConf} as
$$
4\,e^{-2\la}\,\Re\lf[\D_{\ov{z}}\lf(e^{2\lambda}<\vec{H},\vec{H}_0>\ \p_{\ov{z}}\vec{\Phi}\,{e^{-2\lambda}}\rg)\rg]=4\,e^{-2\la}\,\Re\lf[ f(z) \D_{\ov{z}}(e^{-\lambda} \bbe_{\ov{z}} )\rg];
$$
Now using \eqref{VI.204} we write the right hand side of the last equation as  $2 e^{-2\lambda} \Re\lf( f(z) \ov{\bH}_0\rg)=e^{-2 \lambda} \Im\lf( 2i f(z) \ov{\bH}_0 \rg)$. We have just shown that $\bP$ satisfies the constraint-conformal conformal Willmore equation \eqref{eq:CCCWeqConf} with holomorphic function $2i e^{2\la}\,<\vec{H},\vec{H}_0>$.
\hfill$\Box$

\medskip

\noindent{\bf Proof of Theorem ~\ref{Th:rigidity}.}
One implication follows directly from Proposition~\ref{Prop:PMC<CCCW} observing that the constraint on the conformal class is trivial on smooth immersions of spheres by the Uniformization Theorem.

Let us prove the opposite implication by contradiction: we assume that the compact Riemannian 3-manifold $(M^3,h)$ has constant scalar curvature $Scal_0$ but it is not a space form and we exhibit an embedded sphere which has constant mean curvature but is not  conformal Willmore. 
\\ First of all let us denote  $S_{\mu\nu}:=Ric_{\mu\nu}-\frac{1}{3} \,Scal \, h_{\mu \nu} $ the trace-free Ricci tensor of $(M,h)$ and observe that under our assumptions 
\be\label{S>0} 
{\frak m}:=\max_{x\in M} \|S_x\|^2>0.
\ee 
Indeed if $\|S\|^2\equiv 0$ then the manifold is Einstein, but the 3-dimensional Einstein manifolds are just the space forms (for example see \cite{Pet} pages 38-41).
\\Now consider the function $r:M \to \R$ defined as 
\be\label{def:r}
r(x):= -\frac{11}{378} \|S_x\|^2+\frac{55}{1134} Scal^2(x)-\frac{1}{21} \triangle Scal (x)=-\frac{11}{378} \|S_x\|^2+\frac{55}{1134} Scal_0^2
\ee 
where in the last equality we used that $Scal\equiv Scal_0$. The function $r$ we just defined is exactly the function $r$ defined at page 276 in \cite{PX}; this can be seen using the irreducible decomposition of the Riemann curvature tensor which implies (notice that we are assuming $M$ to be 3-d, so the Weyl tensor vanishes)
$$\|Riem^h\|^2(x)= \frac{1}{3} Scal^2(x)+ 4 \|S_x\|^2; $$
plugging this expression in the formula in \cite{PX} and taking $m=3$, after some straighforward computations we end up with \eqref{def:r}. Let us recall  Theorem 1.1 of \cite{PX}:
\\There exists $\rho_0 > 0$ and a smooth function $\phi:M\times(0,\rho_0)\to \R$ such that
\\(i) For all $\rho \in (0, \rho_0)$, if $\bar x$ is a critical point of the function $\phi(.,\rho)$ then, there
exists an embedded hyper-surface $S^{\sharp}_{\bar{x},\rho}$whose mean curvature is constant equal to $\frac{1}{\rho}$ 
 and that is a normal graph over the geodesic sphere $S_{\bar{x},\rho}$ for some function which is bounded
by a constant times $\rho^3$  in  $C^{2,\alpha}$ topology.
\\(ii) For all $k>0$, there exists $c_k > 0$ which does not depend on $\rho \in (0,\rho_0)$ such
that
\be\label{eq:estphi}
\|\phi(.,\rho)-Scal+\rho^2 r\|_{C^k(M)}\leq C_k \rho^3. 
\ee
Now, for $\rho_0$ small enough, we claim that at all  points of global minimum of $\phi$ we have $\|S\|^2>0$. If it is not the case let $x^{\phi}_\rho$ be a point of global minimum of $\phi(.,\rho)$ and observe that\eqref{eq:estphi} and \eqref{def:r} yelds
\be\label{eq:xphi}
\phi(x^{\phi}_\rho,\rho)\geq Scal_0+  \frac{55}{1134} Scal_0^2 \, \rho^2 - C_0 \rho^3;
\ee
on the other hand, at a maximum point $x^{S}$ for $\|S\|^2$,  we have analogously
\be\label{eq:xS}
\phi(x^{S},\rho)\leq Scal_0+  \frac{55}{1134} Scal_0^2 \, \rho^2-\frac{11}{378} {\frak m} \rho^2  + C_0 \rho^3;
\ee
now \eqref{eq:xphi} and  \eqref{eq:xS} together with the  crucial fact that $\frak m >0$ (ensured by the fact that $(M,h)$ is not space form) imply that for $\rho$ small enough $\phi(x^{S},\rho)<\phi(x^{\phi}_\rho,\rho)$ contradicting the minimality of $x^{\phi}_\rho$.
Collecting Theorem 1.1 of \cite{PX} and what we have just proved we conclude the following: for $\rho_0$ small enough, for every $\rho\leq \rho_0$ consider a minimum $x_\rho$ a point for $\phi(.,\rho)$, then
\\a) $\|S_{x_\rho}\|^2>0$
\\b) there exists an embedded CMC sphere $S^{\sharp}_{x_\rho,\rho}$ whose mean curvature given by a normal graph over the geodesic sphere $S_{\bar{x},\rho}$ for some function which is bounded
by a constant times $\rho^3$  in  $C^{2,\alpha}$ topology. Observe that, since the graph function satisfies the mean curvature equation, bootstrapping the $C^{2,\alpha}$ bound using Schauder estimates, one gets  the graph function is bounded in $C^{4,\alpha}$ norm by a constant times $\rho^3$.

But now, since a) holds,   Theorem 1.4 in \cite{Mon2} implies that for small $\rho$ the CMC perturbed geodesic spheres constructed in b) cannot be  conformal Willmore immersions. The proof is now complete.  
\hfill$\Box$

\section{Conformal constrained Willmore surfaces in manifold in arbitrary codimension via a system of conservation laws}\label{Sec:SystY}
\reset
Let us start with a general lemma for surfaces.

\begin{Lm}\label{lm:SysX}
Let $\vec{\Phi}$ be a conformal immersion of the disc $\D^2$ into a Riemannian manifold $(M,h)$ and let $\vec{X}$ be the following $L^1+H^{-1}$ vector field
\be
\label{eq:defX}
\vec{X}:=-2i\lf<\bH, \bH_0 \rg> \p_{\ov{z}}\bP -2i \, \pi_{\bn}(\D_z \bH);
\ee
then the following system of equations holds
\be
\label{eq:SystX}
\lf\{
\begin{array}{l}
\ds \Im\left[\lf<\vec{e}_{\ov{z}},\vec{X}\rg>\right]=0 \quad{(SysX-1)}\\[5mm]
\ds \Im\left[\vec{e}_{\ov{z}}\wedge(\vec{X}+2i D_z\vec{H}) \rg]=0 \quad{(SysX-2)} 
\end{array}
\rg.
\ee
where, given two complex vectors fields $\vec{X}, \vec{Y} \in \Gamma(TM\otimes \C):\vec{X}=\vec{X}_1+i\vec{X}_2, \vec{Y}=\vec{Y}_1+i\vec{Y}_2$ with $\vec{X}_1,\vec{X}_2,\vec{Y}_1,\vec{Y}_2 \in \Gamma(TM)$ we use the notation $<\vec{X},\vec{Y}>$ to denote the quantity
$$<\vec{X},\vec{Y}>:=h(\vec{X}_1,\vec{Y}_1)-h(\vec{X}_2,\vec{Y}_2) + \,i\, h(\vec{X}_1, \vec{Y}_2)+ \,i\, h(\vec{X}_2,\vec{Y}_1)$$
where, of course, $h(.\,,\,.)$ denotes the standard scalar product of tangent vectors in the Riemannian manifold $(M,h)$. 
\hfill $\Box$
\end{Lm} 

\noindent{\bf Proof of Lemma~\ref{lm:SysX}.}
First of all by Lemma \ref{lm-VI.11} we can write $\vec{X}$ as
\be
\label{eq:X}
\vec{X}:=-2i\lf<\bH, \bH_0 \rg> \p_{\ov{z}}\bP -2i \, \pi_{\bn}(\D_z \bH)=i\D_z\vec{H}-3i\pi_{\vec{n}}(\D_z\vec{H})+\,\star_h(\D_z\vec{n}\wedge\vec{H}).
\ee
Let us start by the first equation (SysX-1). Since $\vec{e}_{\ov{z}}$ is tangent and $\pi_{\bn}(\D_z\vec{H})$ is normal to $(\bP)_*(T\D^2)$, the scalar product symplifies as
\be \label{eq:ScalX1}
\Im\left[\lf<\vec{e}_{\ov{z}},\vec{X}\rg>\right]=\Im\left[\lf<\vec{e}_{\ov{z}},i\pi_{T}(\D_z\vec{H})\rg>\rg] + \Im\left[\lf<\vec{e}_{\ov{z}},\,\star_h(\D_z\vec{n}\wedge\vec{H})  \rg>\right]. 
\ee
Identity (\ref{z-VI.1}) together with  (\ref{VI.199a}) gives
\be\label{eq:ScalX2}
\Im\left[\lf<\vec{e}_{\ov{z}},i\pi_{T}(\D_z\vec{H})\rg>\rg]=-\Im\left[\lf<\vec{e}_{\ov{z}},\,\star_h(\D_z\vec{n}\wedge\vec{H})  \rg>\right].
\ee
Putting together (\ref{eq:ScalX1}) and (\ref{eq:ScalX2}) we obtain (SysX-1).
\\Now let us prove (SysX-2). Since $\vec{e}_{\bar{z}} \wedge \vec{e}_{\bar{z}}=0$ we have
\be\label{eq:WedgeX1}
\Im\left[\vec{e}_{\ov{z}}\wedge \vec{X}\rg]=\Im\left[\vec{e}_{\ov{z}}\wedge \lf(-2i \, \pi_{\bn}(\D_z \bH)\rg)\rg]=\Im\left[\vec{e}_{\ov{z}}\wedge \lf(-2i\, \D_z \bH \rg) \rg]-\Im\left[\vec{e}_{\ov{z}}\wedge \lf(-2i \, \pi_{T}(\D_z \bH)\rg)\rg].
\ee
In order to have (SysX-2) it's enough to prove that $\Im\left[\vec{e}_{\ov{z}}\wedge \lf(-2i \, \pi_{T}(\D_z \bH)\rg)\rg]=0$. Using (\ref{VI.199a}) we write
\be
\label{eq:WedgeX2}
\pi_{T}(\D_z \bH)=2<\D_z \bH,\vec{e}_z> \vec{e}_{\bar{z}}+2<\D_z \bH,\vec{e}_{\bar{z}}> \vec{e}_{z},
\ee
hence, again $\vec{e}_{\bar{z}} \wedge \vec{e}_{\bar{z}}=0$ implies that
\be\label{eq:WedgeX3}
\Im\left[\vec{e}_{\ov{z}}\wedge \lf(-2i \, \pi_{T}(\D_z \bH)\rg)\rg]=\Im\left[\vec{e}_{\ov{z}}\wedge \lf(-4i \, <\D_z \bH,\vec{e}_{\bar{z}}> \vec{e}_{z}\rg)\rg]=4 \Re \left[ \lf( \, <\D_z \bH,\vec{e}_{\bar{z}}> \rg) \vec{e}_{z} \wedge \vec{e}_{\ov{z}}\rg].
\ee
Now use that $\vec{H}$ is orthogonal to $\vec{e}_1 , \vec{e}_2$ and that $ \pi_{\bn}(\D_{\vec{e}_1} \vec{e}_2)={\mathbb I}_{12}= {\mathbb I}_{21}= \pi_{\bn}(\D_{\vec{e}_2} \vec{e}_1) $ to conclude that 
\begin{eqnarray}
4 \Re \left[ \lf( \, <\D_z \bH,\vec{e}_{\bar{z}}> \rg) \vec{e}_{z} \wedge \vec{e}_{\ov{z}}\rg]&=&\frac 1 2 \lf[ \lf( \lf<\D_{\vec{e}_2} \vec{H},\vec{e}_1 \rg> -\lf<\D_{\vec{e}_1} \vec{H},\vec{e}_2 \rg>\rg) \vec{e}_1 \wedge \vec{e}_2\rg]\nonumber\\
&=&-\frac 1 2 \lf[ \lf( \lf< \vec{H},\D_{\vec{e}_2}\vec{e}_1 \rg> -\lf< \vec{H},\D_{\vec{e}_1}\vec{e}_2 \rg>\rg) \vec{e}_1  \wedge \vec{e}_2\rg] \nonumber \\
&=&0.\nonumber
\end{eqnarray}
\hfill$\Box$

\begin{Th}\label{th:ConfWill}

Let $\vec{\Phi}$ be a conformal immersion of the disc $\D^2$ into a Riemannian manifold $(M,h)$, then $\vec{\Phi}$ is a conformal constrained Willmore immersion if and only if there exists an $H^{-1}+L^1$ vector field $\vec{Y}$  such that 
\be
\label{eq:SystConfWill}
\lf\{
\begin{array}{l}
\ds \Im\left[\lf<\vec{e}_{\ov{z}},\vec{Y}\rg>\right]=0 \quad{(Sys-1)}\\[5mm]
\ds \Im\left[\vec{e}_{\ov{z}}\wedge(\vec{Y}+2i D_z\vec{H}) \rg]=0 \quad{(Sys-2)}\\[5mm]
\ds \Im \left[\D_{\bar{z}} \vec{Y} \right]=- e^{2 \lambda} \lf(\frac 1 2 \tilde{R}(\vec{H})+4 \Re\left[<\Riem^h(\vec{e}_{\ov{z}},\vec{e}_z) \vec{e}_z, \vec{H}> \vec{e}_{\ov{z}}\right] \rg)\quad{(Sys-3)} \quad.
\end{array}
\rg.
\ee
\hfill $\Box$
\end{Th}



\noindent{\bf Proof of Theorem~\ref{th:ConfWill}.}
Let us first prove the ''only if'' part: we assume that $\vec{\Phi}$ is constrained conformal Willmore and prove that there exists an $H^{-1}+L^1$ vector field $\vec{Y}$ satisfying the system  of equations (\ref{eq:SystConfWill}). 
\\Recall that $\vec{\Phi}$ is a conformal constrained Willmore immersion if and only if there exists an holomorphic function $f(z)$ such that the equation (\ref{eq:W'confHC}) is satisfied, namely
\be \nonumber
\Delta_\perp\bH+\ti{A}(\vec{H})-2 |\bH|^2\ \bH-\ti{R}(\bH)= e^{-2\lambda} \Im (f(z) \overline{\vec{H_0}})= <q,\vec{h}_0>_{WP}.
\ee
where $\vec{h}_0=\vec{H}_0\, dz\otimes dz$.
We claim that the vector 
\be \label{eq:defY}
\vec{Y}:=e^{-\lambda} f \vec{e}_{\ov{z}}-2i\lf<\bH, \bH_0 \rg> \p_{\ov{z}}\bP -2i \, \pi_{\bn}(\D_z \bH)
\ee
satisfies the system. Recalled the definition of the vector $\vec{X}$ given in \eqref{eq:defX}, observe that $\vec{Y}=e^{-\lambda} f \vec{e}_{\ov{z}}+\vec{X}$. Since $\vec{X}$ satisfies the system \eqref{eq:SystX} and since $<\vec{e}_{\bar{z}},\vec{e}_{\bar{z}}>=0=\vec{e}_{\bar{z}}\wedge\vec{e}_{\bar{z}}$ we conclude that $\vec{Y}$ satisfies the first two equations of system \eqref{eq:SystConfWill}. Now let us prove the third equation (Sys-3), we have
\begin{eqnarray}
\Im \left[\D_{\bar{z}} \vec{Y} \right] &=& -2 \Im \lf[i \D_{\bar{z}}\lf(\lf<\bH, \bH_0 \rg> \p_{\ov{z}}\bP + \, \pi_{\bn}(\D_z \bH) \rg)\rg]+ \Im \left[D_{\bar{z}} (e^{-\lambda} f \vec{e}_{\ov{z}} )\right] \nonumber \\
&=& -2 \Re \lf[\D_{\bar{z}}\lf(\lf<\bH, \bH_0 \rg> \p_{\ov{z}}\bP + \, \pi_{\bn}(\D_z \bH) \rg)\rg] + \Im \left[f \D_{\bar{z}}(e^{-\lambda}  \vec{e}_{\ov{z}} )\right]+ \Im \left[(\p_{\bar{z}} f)\, e^{-\lambda}  \vec{e}_{\ov{z}} \right] \nonumber.
\end{eqnarray}  
Recall the identity (\ref{VI.204}), sum and substract $e^{2 \lambda} \lf(\frac 1 2 \tilde{R}(\vec{H})+4 \Re\left[<\Riem^h(\vec{e}_{\ov{z}},\vec{e}_z) \vec{e}_z, \vec{H}> \vec{e}_{\ov{z}}\right] \rg)$  and get

\begin{eqnarray}
\Im \left[\D_{\bar{z}} \vec{Y} \right] &=& -e^{2 \lambda} \lf(\frac 1 2 \tilde{R}(\vec{H})+4 \Re\left[<\Riem^h(\vec{e}_{\ov{z}},\vec{e}_z) \vec{e}_z, \vec{H}> \vec{e}_{\ov{z}}\right] \rg)+ \Im \left[(\p_{\bar{z}} f)\, e^{-\lambda}  \vec{e}_{\ov{z}} \right] \label{eq:DzY} \\
&& -2 \Re \lf[\D_{\bar{z}}\lf(\lf<\bH, \bH_0 \rg> \p_{\ov{z}}\bP + \, \pi_{\bn}(\D_z \bH) \rg)\rg] + \frac{1}{2} \Im \left[f \ov{\bH_0}\right] \nonumber \\
&& +e^{2 \lambda} \lf(\frac 1 2 \tilde{R}(\vec{H})+4 \Re\left[<\Riem^h(\vec{e}_{\ov{z}},\vec{e}_z) \vec{e}_z, \vec{H}> \vec{e}_{\ov{z}}\right] \rg). \nonumber
\end{eqnarray}  
Now recall that $\vec{\Phi}$ is conformal constrained Willmore if and only if the identity (\ref{eq:ConsWconf}) holds, moreover $f$ is holomorphic so $\p_{\ov{z}}f=0$, therefore we can conclude that
$$\Im \left[\D_{\bar{z}} \vec{Y} \right] = -e^{2 \lambda} \lf(\frac 1 2 \tilde{R}(\vec{H})+4 \Re\left[<\Riem^h(\vec{e}_{\ov{z}},\vec{e}_z) \vec{e}_z, \vec{H}> \vec{e}_{\ov{z}}\right] \rg)$$
as desired.
\medskip

For the other implication assume there exists a vector $\vec{Y}$ satisfying the system \eqref{eq:SystConfWill} and write $\vec{Y}$ as
$$\vec{Y}=A \vec{e}_z+B\vec{e}_{\ov{z}}+\vec{V}$$
where where $A$ and $B$ are complex number and $\vec{V}:=\pi_{\vec{n}}(\vec{Y})$
is a complex valued normal vector to the immersed surface. The first equation of (\ref{eq:SystConfWill}), using (\ref{VI.199a}), is equivalent to
\be
\label{VI.207}
\Im\, A=0
\ee
Observe that if we write 
$$\D_z\vec{H}=C\ \vec{e}_z+D\ \vec{e}_{\ov{z}}+\vec{W}$$ 
where $\vec{W}=\pi_{\vec{n}}(\D_z\vec{H})$, one has, using (\ref{VI.200}) and the fact
that $\vec{H}$ is orthogonal to $\vec{e}_{\ov{z}}$
\be
\label{VI.208aa}
C=2\lf<\vec{e}_{\ov{z}},\D_z\vec{H}\rg>=-2\lf<\D_z(e^\la\ \vec{e}_{\ov{z}}),\vec{H}\rg>\ e^{-\la}=-e^\la\ |\vec{H}|^2\quad.
\ee
Hence we deduce in particular
\be
\label{VI.208}
\Im\, C=0\quad.
\ee
We have moreover using (\ref{VI.204})
\be
\label{VI.208ab}
D=2\lf<\vec{e}_{{z}},\D_z\vec{H}\rg>=-2\lf<\D_z(e^{-\la}\ \vec{e}_{{z}}),\vec{H}\rg>\ e^{\la}=-e^\la\ <\vec{H}_0,\vec{H}>\quad.
\ee
Thus combining (\ref{VI.208aa}) and (\ref{VI.208ab}) we obtain
\be
\label{VI.208ac}
\D_z\vec{H}=-|\vec{H}|^2\ \p_z\vec{\Phi}-<\vec{H}_0,\vec{H}>\ \p_{\ov{z}}\vec{\Phi}+\pi_{\vec{n}}(\D_z\vec{H}).
\ee
Using (\ref{VI.199a}), the second line in the conservation law (\ref{eq:SystConfWill}) is equivalent to
\be
\label{VI.209}
\lf\{
\begin{array}{l}
\ds\Im(i\,A-2C)=0\\[5mm]
\ds\Im\lf(\vec{e}_{\ov{z}}\wedge\lf[\vec{V}+2i\vec{W}\rg]\rg)=0 \quad.
\end{array}
\rg.
\ee
We observe that $\vec{e}_1\wedge\lf[\vec{V}+2i\vec{W}\rg]$ and $\vec{e}_2\wedge\lf[\vec{V}+2i\vec{W}\rg]$ are linearly independent since $\lf[\vec{V}+2i\vec{W}\rg]$ is orthogonal to the tangent plane, moreover we combine (\ref{VI.208}) and (\ref{VI.209}) and we obtain that (\ref{VI.209}) is equivalent to
\be
\label{VI.210}
\lf\{
\begin{array}{l}
\ds\Im(i\,A)=0\\[5mm]
\ds\vec{e}_1\wedge\Im\lf(\vec{V}+2i\vec{W}\rg)=0\\[5mm]
\ds\vec{e}_2\wedge\Im\lf(i\ \lf[\vec{V}+2i\vec{W}\rg]\rg)=0 \quad.
\end{array}
\rg.
\ee
Combining (\ref{VI.207}) and (\ref{VI.210}) we obtain that  the first two conservation laws of  \eqref{eq:SystConfWill} are equivalent to
\be
\label{VI.211}
\lf\{
\begin{array}{l}
A=0\\[5mm]
\vec{V}=-2i\,\vec{W}=-2i\,\pi_{\vec{n}}(\D_z\vec{H})
\end{array}
\rg.
\ee
Or in other words, for a conformal immersion $\vec{\Phi}$ of the disc into ${\R}^m$, there exists a vector field $\vec{Y}$ satisfying the first two equations of the system  \eqref{eq:SystConfWill}  if and only if there exists
a complex valued function $B$ and a vector field $\vec{Y}$  such that 
\be
\label{VI.212}
\vec{Y}=B\,\vec{e}_{\ov{z}}-2i\,\pi_{\vec{n}}(\D_z\vec{H})\quad.
\ee
We shall now exploit the third equation of \eqref{eq:SystConfWill}  by taking $\D_{\ov{z}}$ of (\ref{VI.212}). 
Let 
\be
\label{VI.212zz}
f:=\,e^{\la} B+2i\, e^{2 \lambda}\,\lf<\vec{H},\vec{H}_0\rg>\quad.
\ee
With this notation (\ref{VI.212}) becomes
\be
\label{VI.212az}
\vec{Y}=e^{-\la}\, f\,\vec{e}_{\ov{z}}-2i\lf<\vec{H},\vec{H}_0\rg>\ \p_{\ov{z}}\vec{\Phi}-2i\,\pi_{\vec{n}}(\D_z\vec{H})
\ee

which is exactly equation (\ref{eq:defY}) (recall  we defined a vector $\vec{Y}$ in that way starting from a conformal immersion $\vec{\Phi}$ satisfying the conformal constraint Willmore equation). Then repeating the computations above (i.e. the ones for the ''only if'' implication) we get that the  equation (\ref{eq:DzY}) is still valid, but since  $\vec{Y}$ satisfies (Sys-3) we get

\begin{eqnarray}
0 &=& -2 \Re \lf[\D_{\bar{z}}\lf(\lf<\bH, \bH_0 \rg> \p_{\ov{z}}\bP + \, \pi_{\bn}(\D_z \bH) \rg)\rg] + \frac{1}{2} \Im \left[f \ov{\bH_0}\right]  \label{eq:PrWConf}\\
&& +e^{2 \lambda} \lf(\frac 1 2 \tilde{R}(\vec{H})+4 \Re\left[<\Riem^h(\vec{e}_{\ov{z}},\vec{e}_z) \vec{e}_z, \vec{H}> \vec{e}_{\ov{z}}\right] \rg)+\Im \left[(\p_{\bar{z}} f)\, e^{-\lambda}  \vec{e}_{\ov{z}} \right]\quad.\nonumber
\end{eqnarray}
Consider the normal and the tangential projections of (\ref{eq:PrWConf}). The tangential projection of identity (\ref{eq:Cons}) gives
\be \label{eq:PT}
\pi_T \lf(-2 \Re \lf[\D_{\bar{z}}\lf(\lf<\bH, \bH_0 \rg> \p_{\ov{z}}\bP + \, \pi_{\bn}(\D_z \bH) \rg)\rg]\rg) = - 4\, e^{2 \lambda} \Re\left[<\Riem^h(\vec{e}_{\ov{z}},\vec{e}_z) \vec{e}_z, \vec{H}> \vec{e}_{\ov{z}}\right]; 
\ee 
on the other hand, the tangential projection of (\ref{eq:PrWConf}) gives

\be \label{eq:PT1}
\begin{array}{l}
\ds 0=\pi_T \lf(-2 \Re \lf[\D_{\bar{z}}\lf(\lf<\bH, \bH_0 \rg> \p_{\ov{z}}\bP + \, \pi_{\bn}(\D_z \bH) \rg)\rg]\rg)\\[5mm]
\quad\ds + 4\, e^{2 \lambda} \Re\left[<\Riem^h(\vec{e}_{\ov{z}},\vec{e}_z) \vec{e}_z, \vec{H}> \vec{e}_{\ov{z}}\right]+\Im \left[(\p_{\bar{z}} f)\, e^{-\lambda}  \vec{e}_{\ov{z}} \right], 
\end{array}
\ee
so, combining (\ref{eq:PT}) and (\ref{eq:PT1}) we obtain
\be \label{eq:PT2}
0=\Im \left[(\p_{\bar{z}} f)\, e^{-\lambda}  \vec{e}_{\ov{z}} \right]. 
\ee
The normal projection of (\ref{eq:PrWConf}) gives
\be\label{eq:PN}
0=\pi_{\vec{n}} \lf(-2 \Re \lf[\D_{\bar{z}}\lf(\lf<\bH, \bH_0 \rg> \p_{\ov{z}}\bP + \, \pi_{\bn}(\D_z \bH) \rg)\rg] 
\quad\ds+ \frac{1}{2} \Im \left[f \ov{\bH_0}\right] + \frac{e^{2 \lambda}}{2} \tilde{R}(\vec{H}) \rg)\quad.
\ee
Therefore, putting togheter (\ref{eq:PT}), (\ref{eq:PT2}) and (\ref{eq:PN}) we conclude that (\ref{eq:PrWConf}) implies the following system
\be
\label{VI.215}
\lf\{
\begin{array}{l}
\ds4\,e^{-2\la}\,\Re\lf(\D_{\ov{z}}\lf[\pi_{\vec{n}}(\D_z\vec{H})+<\vec{H},\vec{H}_0>\ \p_{\ov{z}}\vec{\Phi}\rg]\rg)\\[5mm]
\quad\quad\ds=e^{-2\la}\, \Im (f(z) \overline{\vec{H_0}})+\ti{R}(\bH)+8 \Re \lf(<\Riem^h(\vec{e}_{\ov{z}},\vec{e}_z) \vec{e}_z, \vec{H}> \vec{e}_{\ov{z}}\rg)\\[5mm]
\ds\Im\lf(\p_{\ov{z}}f\ \,\vec{e}_{\ov{z}}\rg)=0\quad.
\end{array}
\rg.
\ee
The second line is equivalent to
\[
\p_{\ov{z}}f\ \,\vec{e}_{\ov{z}}-\p_z\ov{f}\,\vec{e}_z=0.
\] 
Taking the scalar product with $\vec{e}_z$ and using (\ref{VI.199a}), observe that (\ref{VI.215})
is equivalent to
\be
\label{VI.216}
\lf\{
\begin{array}{l}
\ds4\,e^{-2\la}\,\Re\lf(\D_{\ov{z}}\lf[\pi_{\vec{n}}(\D_z\vec{H})+<\vec{H},\vec{H}_0>\ \p_{\ov{z}}\vec{\Phi}\rg]\rg)\\[5mm]
\quad\quad\ds=e^{-2\la}\, \Im (f(z) \overline{\vec{H_0}})+\ti{R}(\bH)+8 \Re \lf(<\Riem^h(\vec{e}_{\ov{z}},\vec{e}_z) \vec{e}_z, \vec{H}> \vec{e}_{\ov{z}}\rg)\\[5mm]
\ds \p_{\ov{z}}f=0\quad.
\end{array}
\rg.
\ee
The second line gives  that $f=f(z)$ is holomorphic and the equation in the first line is exactly the equation of the conformal constrained Willmore surfaces (\ref{eq:ConsWconf}); therefore we proved that the existence of a vector field $\vec{Y}$ satisfying the system of conservation laws \eqref{eq:SystConfWill} implies that the immersion $\vec{\Phi}$ is conformal constrained Willmore.\hfill$\Box$

\section{Regularity for  Willmore immersions}\label{Sec:Regularity}
\reset
We start by using the divergence structure of the constrained-conformal Willmore equation in order to construct potentials which will play a crucial role in the regularity theory.

\begin{Lm}\label{Lm:RS}
Let $\vec{\Phi}$ be a $W^{1,\infty}$ conformal immersion of the disc $\D^2$ taking values into a sufficiently small open subset of the Riemannian manifold $(M,h)$, with second fundamental form in $L^2(D^2)$ and conformal factor $\lambda\in L^\infty(D^2)$. Assume $\vec{\Phi}$ is a  constrained-conformal Willmore immersion; then there exist the following potential vector fields:
\begin{itemize}
\item[{\rm (i)}] there exists a complex vector field (i.e. a vector field with values in the complexified tangent bundle of $M$) $\bL \in L^{2,\infty}(D^2)$ with $\nabla \Im \bL \in L^{2,\infty}(D^2)$ satisfying
\be \label{eq:DzL=Y}
\lf\{
\begin{array}{l}
\D _z \bL= \bY \quad 	\text{on } D^2\\[5mm]
\Im \bL=0 \quad \text{on } \p D^2\quad,
\end{array}
\rg.
\ee
where $\bY$ is the vector given in \eqref{eq:defY} in the proof of Theorem \ref{th:ConfWill};
\item[{\rm (ii)}] there exists a complex valued  function $S \in W^{1,(2,\infty)}(D^2)$ with $\nabla^2 \Im S \in L^q(D^2)$ for every $1<q<2$
satisfying
\be \label{eq:DzS}
\lf\{
\begin{array}{l}
\p _z S= <\p_z\bP, \overline{\bL}>  \quad 	\text{on } D^2\\[5mm]
\Im S=0 \quad \text{on } \p D^2\quad;
\end{array}
\rg.
\ee
\item[{\rm (iii)}] there exists a complex valued 2-vector field $\bR \in W^{1,(2,\infty)}(D^2)$ with $\nabla^2 \Im \bR \in L^q(D^2)$ for every $1<q<2$ satisfying
\be \label{eq:DzR}
\lf\{
\begin{array}{l}
\D _z \bR= \p_z\bP \wedge \overline{\bL}-2 i \p_z\bP \wedge \bH  \quad 	\text{on } D^2\\[5mm]
\Im \bR=0 \quad \text{on } \p D^2\quad.
\end{array}
\rg.
\ee
\end{itemize}
\hfill $\Box$
\end{Lm}

\begin{proof}
 (i): Let $\bY$ be the vector field given by \eqref{eq:defY} in the proof of Theorem \ref{th:ConfWill} and observe that, by our assumption of the immersion $\bP$, we have $\bY\in H^{-1}+L^1(D^2)$. Moreover, since $\bY$ satisfyes equation (Sys-3) of \eqref{eq:SystConfWill}, then
\be \label{eq:ImDzYL2}
\|\Im(D_{\bar z} \bY)\|_{L^2(D^2)}\leq C \|H\|_{L^2(D^2)}\quad.
\ee 
Since $\bP$ is taking values into a small open subset $V\subset M$, by choosing Riemann normal coordinates on $V$ centred in $\bP(0)$,  we can assume that the functions
\be\label{eq:defgamma}
\gamma^j_k := \Gamma^j_{kl} \p_z \Phi^l, \quad \gamma^j_k\in C^0\cap W^{1,2}(D^2)
\ee
are smaller than the $\epsilon$ given in the statement of Lemmas \ref{lemmaA2} and \ref{lemmaA2bis}; extend $\gamma^j_k$ to the whole $\C$  and multiply them by a smooth cutoff function in order to obtain
\be\label{eq:boundgamma}
\gamma^j_k \in C^0\cap W^{1,2}(\C),\quad \supp \gamma^j_k \subset B_2(0),\quad \|\gamma^j_k\|_{L^\infty(\C)}\leq \epsilon\quad. 
\ee
Using $\gamma^j_k$ we can extend the operator $D_z$ to complex vector fields $ \bU\in L^1_{loc}(\C)$ in the following way
$$D_z U^j:= \p_z U^j+ \sum_{k=1}^m \gamma^j_k U^k \quad \text {in distributional sense}. $$
Analogously extend $Y^j\in H^{-1}+L^1(D^2)$ to functions $\tilde Y^j \in \mathring{H}^{-1}+L^1(\C)$, where $\mathring{H}^{-1} (\C)$ is the dual of homogeneous Sobolev space $\mathring{H}^{1}(\C)$ (this is just a technical point for applying Lemma \ref{lemmaA2}) such that
 $$\|\tilde Y^j \|_{\mathring{H}^{-1}+L^1(\C)}\leq C \|Y^j \|_{H^{-1}+L^1(D^2)}< \infty\quad,$$
 $$\|\Im (D_z \tilde{Y}^j) \|_{L^1(\C)} \leq C \|\Im(D_z Y^j)\|_{L^1(D^2)}\leq C \|\Im(D_z Y^j)\|_{L^2(D^2)}<\infty\quad.$$
For convenience, in the following we identify $Y^j$ and its extension. Now we apply Lemma \ref{lemmaA2} and  define $\bL \in L^{2,\infty} (D^2) $ to be the unique solution to the problem
\be \nonumber
\lf\{
\begin{array}{l}
\D _z \bL= \bY \quad 	\text{on } D^2\\[5mm]
\Im \bL=0 \quad \text{on } \p D^2\quad.
\end{array}
\rg.
\ee
Observe that moreover the same lemma gives that $\nabla (\Im \bL) \in {L^{2,\infty}(D^2)}$ which implies that $\Im \bL \in L^p(D^2)$ for every $1<p<\infty$.
\\

\emph{Proof of } (ii). Let us start with a computation; from (Sys-1) of \eqref{eq:SystConfWill}, since by (i) we have $\D _z \bL= \bY$ on $D^2$, then 
$$0=\Im\left(<\p_{\bar{z}}\bP, D_z \bL> \right)=\Im\left(\p_z<\p_{\bar{z}}\bP,  \bL>-<D_z \p_{\bar{z}}\bP,  \bL > \right).$$
Using identity \eqref{VI.200}, by complex conjugation we obtain
\be \label{eq:ImSLqpre}
\Im(\p_{\bar{z}}<\p_{z}\bP,  \overline{\bL}>)=-\frac{e^{2\lambda}}{2} <\bH,\Im \bL>\quad \in L^q(D^2) \text{ for every } 1<q<2,
\ee
where the $L^q$ bound follows by H\"older inequality observing that by (i) we have $\nabla \Im \bL \in L^{2,\infty}(D^2)$ then $\Im \bL \in L^p(D^2)$ for every $1<p<\infty$; on the other hand, by assumption, $\bH \in L^2(D^2)$.
\\By (i), we have  $<\p_{z}\bP,  \overline{\bL}>\in L^{2,\infty}(D^2)$ and as before we extend it to the whole $\C$  keeping controled the norms: $<\p_{z}\bP,  \overline{\bL}> \in L^1\cap L^{2,\infty}(\C)$ and $\Im(\p_{\bar{z}}<\p_{z}\bP,  \overline{\bL}>) \in L^q(\C)$ for every $1<q<2$.
\\Now we apply Lemma \ref{lemmaA2bis}, with $m=1$ and $\gamma^j_k=0$, and define $S \in W^{1,(2,\infty)}(D^2)$ to be the unique solution to
\be \nonumber
\lf\{
\begin{array}{l}
\p _z S= <\p_z\bP, \overline{\bL}>  \quad 	\text{on } D^2\\[5mm]
\Im S=0 \quad \text{on } \p D^2\quad;
\end{array}
\rg.
\ee
moreover $\nabla^2 \Im S \in L^{q}(D^2)$ for every $1<q<2$ which implies, by Sobolev Embedding Theorem, $\nabla \Im S \in L^p(D^2)$ for all $1<p<\infty$.
\\

\emph{Proof of } (iii). Since $\bY=D_z \bL$, equation (Sys-2) in \eqref{eq:SystConfWill} gives
$$0=\Im\left[\p_{\bar{z}} \bP \wedge D_z\bL+2i \p_{\bar{z}} \bP \wedge D_z\bH\right]=-\Im\left[D_{\bar{z}} \left(\p_{z} \bP \wedge \overline{\bL} -2i \p_z \bP \wedge \bH \right) - (D_{\bar{z}} \p_z \bP) \wedge \overline{\bL}+2i (D_{\bar{z}} \p_z \bP) \wedge \bH \right], $$
using  \eqref{VI.200} we obtain
\be \label{eq:ImRLq}
\Im\left[D_{\bar{z}} \left(\p_{z} \bP \wedge \overline{\bL} -2i \p_z \bP \wedge \bH \right)\right]=-\frac{e^{2 \lambda}}{2} \bH \wedge \Im\bL \quad \in L^{q}(D^2) \text{ for every } 1<q<2,
\ee
where the $L^q(D^2)$ estimate comes from H\"older inequality since $\bH \in L^2(D^2)$ and $\Im \bL \in L^p(D^2)$ for every $1<p<\infty$. As in (ii) extend the complex valued vector field $\p_{z} \bP \wedge \overline{\bL} -2i \p_z \bP \wedge \bH \in L^{2,\infty}(D^2)$ to  a complex valued vector field on $\C$ keeping the norms controled:  $\p_{z} \bP \wedge \overline{\bL} -2i \p_z \bP \wedge \bH \in L^1\cap L^{2,\infty}(\C)$ and $\Im\left[D_{\bar{z}} \left(\p_{z} \bP \wedge \overline{\bL} -2i \p_z \bP \wedge \bH \right)\right] \in L^{q}(D^2)$ for every $1<q<2$.
\\As in (ii), we apply Lemma \ref{lemmaA2bis} in order to define $\bR \in W^{1,(2,\infty)}(D^2)$ as the unique solution to 
\be \nonumber
\lf\{
\begin{array}{l}
\D _z \bR= \p_z\bP \wedge \overline{\bL}-2 i \p_z\bP \wedge \bH  \quad 	\text{on } D^2\\[5mm]
\Im \bR=0 \quad \text{on } \p D^2\quad;
\end{array}
\rg.
\ee
 moreover $\nabla^2 \Im \bR \in L^{q}(D^2)$ for every $1<q<2$ which implies, by Sobolev Embedding Theorem, $\nabla \Im \bR \in L^p(D^2)$ for all $1<p<\infty$.  
\end{proof}

Next we play with the introduced $\bR$ and $S$ in order to produce, in the following lemma, an elliptic system of Wente type involving $\bP$, $\bR$ and $S$.

\begin{Lm}\label{Lm:eqPRS}
Let $\vec{\Phi}$ be a $W^{1,\infty}$ conformal immersion of the disc $\D^2$ taking values into a sufficiently small open subset of the Riemannian manifold $(M,h)$, with second fundamental form in $L^2(D^2)$ and conformal factor $\lambda\in L^\infty(D^2)$. Assume $\vec{\Phi}$ is a  constrained-conformal Willmore immersion and let $\bR \in W^{1,(2,\infty)}(D^2)$ and $S \in W^{1,(2,\infty)}(D^2)$ be given by Lemma \ref{Lm:RS}; then $\bR$ and $S$ satisfy the following coupled system on $D^2$:
\be \label{eq:SystRS}
\lf\{
\begin{array}{l}
\D _z \bR= (-1)^{m+1} \star_h \left[\bn \bul i\D _z \bR \right] + (i \p_z S) \star_h \bn  \\[5mm]
\p_z S = <-i \D_z \bR, \star_h \bn> \quad.
\end{array}
\rg.
\ee
\hfill $\Box$
\end{Lm}

\begin{proof}
By definition, $\bR$ satisfies the  equation \eqref{eq:DzR} on $D^2$, i.e:
\be\label{DzR}
\D _z \bR= \p_z\bP \wedge \overline{\bL}-2 i \p_z\bP \wedge \bH\quad.
\ee
Taking the $\bul$ contraction defined in \eqref{def:bul} between $\bn$ and $\D_z \bR$ we obtain

\begin{eqnarray} 
\bn \bul \D _z \bR &=&-(\bn \llcorner \overline{\bL}) \wedge \p_z \bP +2i (\bn \llcorner \bH) \wedge \p_z \Phi  \nonumber  \\
                   &=&-[\bn \llcorner \pi_{\bn} (\overline{\bL})]\p_z \bP+2i (\bn \llcorner \bH) \wedge \p_z \Phi \quad, \label{eq:nbulDzR}
\end{eqnarray}
where $\llcorner$ is the usual contraction defined in \eqref{eq:defllcorner}.
\\For a normal vector $\bN$, a short computation using just the definitions of $\star_h$ and $\llcorner$ gives
\begin{eqnarray}
\star_h[(\bn\llcorner \bN) \wedge \bbe_1]&=&(-1)^{m} \;\bN \wedge \bbe_2 \nonumber \\
\star_h[(\bn\llcorner \bN) \wedge \bbe_2]&=&(-1)^{m+1} \;\bN \wedge \bbe_1, \nonumber 
\end{eqnarray}
where, as usual, $\bbe_1$ and $\bbe_2$ is the ortonormal basis of $T\bP(D^2)$ given by the vectors $\p_{1}\bP$, $\p_2 \bP$ normalized. Since $\p_z \bP=\frac{1}{2} \left[ \p_1 \bP-i \p_2 \bP \right]$, we get
\be\label{eq:nNP}
\star_h[(\bn\llcorner \bN) \wedge \p_z \bP]=(-1)^{m} \;\bN \wedge (i\p_z \bP).
\ee
Combining \eqref{eq:nbulDzR} and \eqref{eq:nNP} we have
\be\label{eq:starnDzR}
\star_h[\bn \bul \D _z \bR]=(-1)^{m+1} \pi_{\bn} (\overline{\bL}) \wedge (i\p_z \bP)+2 i (-1)^{m} \bH \wedge (i\p_z \bP)\quad,
\ee
multiplying both sides with $i (-1)^m$ gives
\be \label{eq:quasi}
(-1)^{m+1} \star_h[\bn \bul (i\D _z \bR)]= \p_z \bP \wedge \pi_{\bn} (\overline{\bL}) -2i  \p_z \bP \wedge \bH\quad.
\ee
Combining \eqref{DzR} and \eqref{eq:quasi} we obtain
\be\label{eq:quasi'}
(-1)^{m+1} \star_h[\bn \bul (i\D _z \bR)]= \D_z \bR - \p_z \bP \wedge \pi_{T} (\overline{\bL})\quad. 
\ee 
Observing that, by \eqref{VI.199a}
$$\pi_T(\overline{\bL})=2<\overline{\bL},\bbe_{\bar{z}}> \bbe_{z}+2<\overline{\bL},\bbe_{{z}}> \bbe_{\bar{z}}\quad, $$
then
$$\p_z \bP \wedge \pi_{T} (\overline{\bL})= \p_z \bP \wedge (2<\overline{\bL},\bbe_{{z}}>) \bbe_{\bar{z}}= (2<\overline{\bL},\p_z \bP>) \bbe_{{z}} \wedge  \bbe_{\bar{z}}\quad,$$
using again \eqref{VI.199a}, and the definition of $S$ \eqref{eq:DzS} gives
\be\label{eq:PLS}
\p_z \bP \wedge \pi_{T} (\overline{\bL})= (i \p_z S) \star_h \bn\quad.
\ee
The combination of \eqref{eq:PLS} and \eqref{eq:quasi'} gives the  first equation of \eqref{eq:SystRS}. The second equation is obtained by taking the scalar product between the first equation and $\star_h \bn$ once one have observed 
that
$$<\star_h \bn, \star_h (\bn \bul \D _z \bR)>=0\quad. $$
This fact comes from \eqref{eq:starnDzR} which implies that $\star_h (\bn \bul \D _z \bR)$ is a linear combination of wedges of tangent and normal vectors to $T\bP(D^2)$. This concludes the proof.
\end{proof}

\begin{Prop}
Let $\vec{\Phi}$ be a $W^{1,\infty}$ conformal immersion of the disc $\D^2$ taking values into a sufficiently small open subset of the Riemannian manifold $(M,h)$, with second fundamental form in $L^2(D^2)$ and conformal factor $\lambda\in L^\infty(D^2)$. Assume $\vec{\Phi}$ is a  constrained-conformal Willmore immersion and let $\bR \in W^{1,(2,\infty)}(D^2)$ and $S \in W^{1,(2,\infty)}(D^2)$ be given by Lemma \ref{Lm:RS}; then the couple $(\Re\bR, \Re S )$ satisfies the following system on $D^2$
\be \label{eq:SystPRS}
\lf\{
\begin{array}{l}
\triangle (\Re \bR) = (-1)^m \star_h [ \D \bn \bul \D^{\perp} (\Re \bR) ] - \star_h [ \D \bn \;\, \nabla^{\perp} (\Re S)] + \tilde{F}   \\[5mm]
\triangle (\Re S) = <\D (\star_h \bn), \D^{\perp} (\Re \bR)   > + \tilde{G} \quad;
\end{array}
\rg.
\ee
where $\tilde{F}$ and $\tilde{G}$ are some functions ($\tilde{F}$ is 2-vector valued) in $L^q(D^2)$ for every $1<q<2$.
Moreover we denoted $\triangle (\Re \bR):= \D_{\p_{x_1} \bP} \D_{\p_{x_1} \bP} (\Re \bR)+\D_{\p_{x_2} \bP} \D_{\p_{x_2} \bP} (\Re \bR)$, observe this differ from the intrinsic Laplace-Beltrami operator by a factor $e^{2\lambda}$. For a more explicit shape of the equations  see \eqref{eq:LapReR} and \eqref{eq:LapReS} in the end of the proof.
\hfill $\Box$
\end{Prop}

\begin{proof}
Let us start by proving the first equation.  Applying the $\D_{\bar{z}}$ operator to the first equation of \eqref{eq:SystRS} we have 
$$ \D_{\bar{z}}\D _z \bR= (-1)^{m+1} i \star_h \D_{\bar{z}}\left[\bn \bul \D _z \bR \right] + i \D_{\bar{z}}\left[ \p_z S \star_h \bn\right],$$
whose real part is
\be\label{eq:DbzDz}
\Re(\D_{\bar{z}}\D _z \bR )= (-1)^{m}  \star_h \Im\left(\D_{\bar{z}}\left[\bn \bul \D _z \bR \right]\right)-\Im \left( \D_{\bar{z}}\left[ \p_z S \star_h \bn\right] \right).
\ee
Observe that 
\be\label{eq:triangle}
\D_{\bar{z}}\D _z \bR := \frac{1}{4} \left[(\D_{\p_{x_1}\bP} + i \D_{\p_{x_2}\bP})(\D_{\p_{x_1}\bP}  - i  \D_{\p_{x_2}\bP})\right] \bR = \frac{1}{4} \triangle \bR - \frac{i}{4} \left[\D_{\p_{x_1}\bP}, \D_{\p_{x_2}\bP} \right] \bR,
\ee
where $\left[\D_{\p_{x_1}\bP}, \D_{\p_{x_2}\bP} \right] :=(\D_{\p_{x_1}\bP} \D_{\p_{x_2}\bP} - \D_{\p_{x_2}\bP} \D_{\p_{x_1}\bP}) $ is the usual bracket notation. An easy computation in local coordinates shows that all the derivatives appearing in $[\D_{\p_{x_1}\bP}, \D_{\p_{x_2}\bP}] (\bR)$ cancel out together with all the mixed terms, giving
\be\label{eq:Bracket}
\begin{array}{l}
\ds\left[\D_{\p_{x_1}\bP}, \D_{\p_{x_2}\bP} \right] \left(\sum_{i,j=1}^m R^{ij} \bE_i \wedge \bE_j \right)\\[5mm]
\ds\quad\quad= \sum_{i,j=1}^m R^{ij} \left[\left( \Riem(\p_{x^1}\bP, \p_{x^2}\bP) \bE_i \right) \wedge \bE_j +  \bE_i \wedge  \left( \Riem(\p_{x^1}\bP, \p_{x^2}\bP)  \bE_j \right) \right]\quad,
\end{array}
\ee 
where as before $\{\bE_i\}_{i=1,\ldots,m}$ is an orthormal frame of $T_{\bP(x)}M$.
Putting together \eqref{eq:DbzDz} and \eqref{eq:triangle} we obtain
\be\label{eq:LapR}
\triangle (\Re \bR)=4 (-1)^{m}  \star_h \Im\left(\D_{\bar{z}}\left[\bn \bul \D _z \bR \right]\right)- 4\Im \left( \D_{\bar{z}}\left[ \p_z S \star_h \bn\right] \right) -\left[\D_{\p_{x_1}\bP}, \D_{\p_{x_2}\bP} \right] (\Im \bR)
\ee
Using that the $\bul$ contraction commutes with the covariant derivative (this fact follows by the definitions and by the identity $\D h=0$, i.e. the connection is metric) we compute
\be\label{eq:4.2}
\begin{array}{l}
\Im\left[\D_{\bar{z}}(\bn \bul \D_z \bR)\right]= \Im\left [\bn \bul (\D_{\bar{z}} \D_z \bR)+ \D_{\bar{z}} \bn \bul \D_z \bR\right]\\[5mm]
\ds\quad=\frac{1}{4} \bn \bul [\triangle  (\Im \bR)- [\D_{\p_{x_1}\bP}, \D_{\p_{x_2}\bP}] (\Re \bR) ]+\Im \left[ \D_{\bar{z}} \bn \bul \D_z \bR \right]\quad.
\end{array}
\ee
A short computation gives
\begin{eqnarray}
\Im\left[ \D_{\bar{z}} \bn \bul \D_z \bR \right]&=&\frac{1}{4} \left[\D_{\p_{x^1}\bP} \bn \bul \D_{\p_{x^1}\bP} (\Im \bR)+ \D_{\p_{x^2}\bP} \bn \bul \D_{\p_{x^2}\bP} (\Im \bR)   \right] \nonumber \\
       && + \frac{1}{4} \left[\D_{\p_{x^2}\bP} \bn \bul \D_{\p_{x^1}\bP} (\Re \bR)- \D_{\p_{x^1}\bP} \bn \bul \D_{\p_{x^2}\bP} (\Re \bR)   \right]\quad.\label{eq:5.1}
\end{eqnarray}
Analogously, using that $\star_h$ commutes with the covariant derivative, another short computation gives
\begin{eqnarray}
\Im\left[ \D_{\bar{z}}( \p_z S \star_h \bn )\right]&=&\frac{1}{4} \triangle(\Im S) \star_h \bn+ \frac{1}{4} \left[   \p_{x^1}(\Im S) \; \D_{\p_{x^1}\bP}(\star_h  \bn)+\p_{x^2}(\Im S) \; \D_{\p_{x^2}\bP}(\star_h  \bn)\right]  \nonumber \\
       && + \frac{1}{4} \left[   \p_{x^1}(\Re S) \; \D_{\p_{x^2}\bP}(\star_h  \bn)-\p_{x^2}(\Re S) \; \D_{\p_{x^1}\bP}(\star_h  \bn)\right]. \label{eq:6.1}
\end{eqnarray}
Combining \eqref{eq:LapR},\eqref{eq:Bracket}, \eqref{eq:4.2}, \eqref{eq:5.1} and \eqref{eq:6.1} we conclude that
\begin{eqnarray}
\triangle (\Re \bR)&=& (-1)^{m}  \star_h \left[\D_{\p_{x^2}\bP} \bn \bul \D_{\p_{x^1}\bP} (\Re \bR)- \D_{\p_{x^1}\bP} \bn \bul \D_{\p_{x^2}\bP} (\Re \bR)   \right] \nonumber \\
&& + \left[   \p_{x^2}(\Re S) \; \D_{\p_{x^1}\bP}(\star_h  \bn)- \p_{x^1}(\Re S) \; \D_{\p_{x^2}\bP}(\star_h  \bn)\right] + \tilde{F} \label{eq:LapReR}
\end{eqnarray}
where $\tilde{F} \in L^q(D^2)$ for every $1<q<2$, and we used that $\D \bn \in L^{2}(D^2), \bR \in W^{1,(2,\infty)}(D^2), S\in  W^{1,(2,\infty)}(D^2), \Im \bR \in W^{2,q}(D^2), \Im S \in W^{2,q}(D^2)$ for every $1<q<2$. This is exactly the first equation of \eqref{eq:SystPRS}.
\\

The second equation of \eqref{eq:SystPRS} is obtained in analogous way: applying the $\p_{\bar {z}}$ operator  to the second equation of \eqref{eq:SystRS} we obtain
$$\triangle S= 4 \p_{\bar{z}} \p_{z} S=-4 i \p_{\bar{z}} <D_z\bR, \star_h \bn>.$$
A short computation gives

\begin{eqnarray}
\triangle (\Re S)&=& 4 \Im\left[\p_{\bar{z}} <D_z \bR, \star_h \bn> \right] \nonumber \\
                 &=& \p_{x_1} <D_{\p_{x^1}\bP} \Im \bR, \star_h \bn> + \p_{x^2} <D_{\p_{x^2}\bP} \Im \bR, \star_h \bn> \nonumber\\
                 &&- <[D_{\p_{x^1}\bP}, D_{\p_{x^2}\bP}] \Re \bR, \star_h \bn> \nonumber \\
                 &&+ <D_{\p_{x^1}\bP} \Re \bR, D_{\p_{x^2}\bP} (\star_h \bn)> -  <D_{\p_{x^2}\bP} \Re \bR, D_{\p_{x^1}\bP} (\star_h \bn)>\quad .
\end{eqnarray}
Recalling that $\Im \bR \in W^{2,q}(D^2)$ and $\bR \in W^{1,(2,\infty)}(D^2)$, and using \eqref{eq:Bracket},  we conclude that
\be \label{eq:LapReS}
\triangle (\Re S)= <D_{\p_{x^1}\bP} (\Re \bR), D_{\p_{x^2}\bP} (\star_h \bn)>-  <D_{\p_{x^2}\bP} (\Re \bR), D_{\p_{x^1}\bP} (\star_h \bn)>+ \tilde{G}
\ee
where $\tilde{G} \in L^q(D^2)$ for every $1<q<2$. This is exactly the second equation of \eqref{eq:SystPRS}.
\end{proof}

Now we are in position to prove the $C^\infty$ regularity of constraint-conformal Willmore immersions.

\begin{Th}\label{Th:Regularity1}
Let $\vec{\Phi}$ be a $W^{1,\infty}$ conformal immersion of the disc $\D^2$ taking values into a sufficiently small open subset of the Riemannian manifold $(M,h)$, with second fundamental form in $L^2(D^2)$ and conformal factor $\lambda\in L^\infty(D^2)$. If $\vec{\Phi}$ is a  constrained-conformal Willmore immersion then $\bP$ is  $C^\infty$.\hfill $\Box$
\end{Th}

\begin{proof}
Let us call $\bA:=(\Re \bR, \Re S)=(\Re R^{ij}, \Re S)$ the vector of the components (in local coordinates in the small neighboorod $V\subset M$) of the real parts of the  potentials $\bR$ and $S$. Using coordinates also in the domain $D^2$, one easily checks that the system \eqref{eq:SystPRS} has the form

\be\label{eq:SystAB}
\triangle A^i= \sum_{k} \left[\p_{x^1}B^i_k\; \p_{x^2} A^k\;-\;\p_{x^2}B^i_k \;\p_{x^1} A^k \right]+ F^i,
\ee 
where $F^i \in L^{q}(D^2)$ for every $1<q<2$, $\nabla A^i \in L^{2,\infty}(D^2)$, $\nabla B^i_k \in L^2(D^2)$.
\\

\emph{Step 1}: $\nabla A^i \in L^2_{loc} (D^2)$.  Let us write $A^i$ as 
\be\label{eq:Ai}
A^i=\varphi^i+V^i+W^i \quad \text{on } D^2,
\ee 
where $\varphi^i, V^i, W^i$ solve the following problems
\be \label{eq:varphi}
\lf\{
\begin{array}{l}
\triangle \varphi^i  = \sum_{k} \left[\p_{x^1}B^i_k\; \p_{x^2} A^k\;-\;\p_{x^2}B^i_k \;\p_{x^1} A^k \right] \quad \text{on } D^2   \\[5mm]
\varphi^i = 0  \quad \text{on } \p D^2 \quad;
\end{array}
\rg.
\ee

\be \label{eq:V}
\lf\{
\begin{array}{l}
\triangle V^i  = F^i \quad \text{on } D^2   \\[5mm]
V^i = 0  \quad \text{on } \p D^2 \quad;
\end{array}
\rg.
\ee

\be \label{eq:W}
\lf\{
\begin{array}{l}
\triangle W^i  = 0 \quad \text{on } D^2   \\[5mm]
W^i = A^i  \quad \text{on } \p D^2 \quad.
\end{array}
\rg.
\ee

Since the right hand side of \eqref{eq:varphi} is sum  of $L^{2,\infty}-L^2-$Jacobians, by a refinement of the Wente inequality obtained by Bethuel \cite{Bet} as a consequence of a result by  Coifman Lions Meyer and Semmes, we have $\nabla \varphi^i \in L^2(D^2)$.
\\On the other hand, since $F^i \in L^{q}(D^2)$, for every $1<q<2$, it follows that $V^i \in W^{2,q}(D^2)$, which implies by Sobolev embedding that $\nabla V^i \in L^2(D^2)$.
\\Finally, $W^i$ is an harmonic $W^{1,{2,\infty}}(D^2)$ function, therefore the gradient $\nabla W^i \in L^2_{loc} (D^2)$.

We conclude that $\nabla A^i=\nabla \varphi^i+\nabla V^i+\nabla W^i \in L^2_{loc}(D^2)$.
\\

\emph{Step 2}: $\nabla A^i \in L^p_{loc} (D^2)$ for some $p>2$.
\\We first claim that there exists $\alpha >0$ such that
\be \label{eq:decayNA}
\sup_{x_0\in B_{\frac{1}{2}}(0), \rho<\frac{1}{4}} \quad \frac{1}{\rho^{\alpha}} \int_{B_\rho(x_0)} |\nabla A|^2 < \infty\quad.
\ee 
Since $\nabla B \in L^2(D^2)$, by absolute continuity of the integral, for every $\epsilon>0$ there exists a $\rho_0>0$ such that
\be\label{eq:Beps}
\sup_{x_0\in B_{\frac{1}{2}}(0)} \quad \int_{B_{\rho_0}(x_0)} |\nabla B|^2 < \epsilon^2\quad.
\ee
Consider $\rho<\rho_0$ ($\epsilon>0$ will be chosen later depending on universal constants) and $x_0 \in B_{\frac{1}{2}(0)}$.
Analogously to Step 1 let us write 
\be\label{eq:Ai'}
A^i=\varphi^i+V^i+W^i \quad \text{on } B_{\rho}(x_0)\quad,
\ee 
where $\varphi^i, V^i, W^i$ solve the following problems
\be \label{eq:varphi'}
\lf\{
\begin{array}{l}
\triangle \varphi^i  = \sum_{k} \left[\p_{x^1}B^i_k\; \p_{x^2} A^k\;-\;\p_{x^2}B^i_k \;\p_{x^1} A^k \right] \quad \text{on }  B_{\rho}(x_0)   \\[5mm]
\varphi^i = 0  \quad \text{on } \p  B_{\rho}(x_0) \quad;
\end{array}
\rg.
\ee

\be \label{eq:V'}
\lf\{
\begin{array}{l}
\triangle V^i  = F^i \quad \text{on }  B_{\rho}(x_0)   \\[5mm]
V^i = 0  \quad \text{on } \p  B_{\rho}(x_0) \quad;
\end{array}
\rg.
\ee

\be \label{eq:W'}
\lf\{
\begin{array}{l}
\triangle W^i  = 0 \quad \text{on } B_{\rho}(x_0)   \\[5mm]
W^i = A^i  \quad \text{on } \p  B_{\rho}(x_0) \quad.
\end{array}
\rg.
\ee 
Notice that $\varphi^i, V^i, W^i$ are different from the ones in Step 1 since they solve different problems, in any case for convenience of notation we call them in the same way. 

Let us start analyzing $\varphi^i$ solution to \eqref{eq:varphi'}. Observe that the right hand side of the equation is a sum of jacobians which,  by Step 1, now are in $L^{2}_{loc}(D^2)$. By Wente estimate \cite{We} (see also \cite{RiCours} Theorem III.1 ) we have
\be \label{eq:decayp}
\|\nabla \varphi^i \|_{L^2(B_{\rho}(x_0))}\leq C \|\nabla B\|_{L^2(B_{\rho}(x_0))} \|\nabla A\|_{L^2(B_{\rho}(x_0))}\leq C \epsilon \|\nabla A\|_{L^2(B_{\rho}(x_0))}\quad,
\ee
where, in the last inequality, we used \eqref{eq:Beps}.

Now we pass to consider \eqref{eq:V'}. Call
\be\label{Vtilde}
\tilde{V}^i(x):=V^i(\rho x+ x_0) \qquad \tilde{F}^i(x):=\rho^2 F^i(\rho x + x_0)
\ee
and observe that, since $V^i$ satisfyes \eqref{eq:V'}, then $\tilde{V}^i$ solves
\be \label{eq:Vtilde}
\lf\{
\begin{array}{l}
\triangle \tilde{V}^i  = \tilde{F}^i \quad \text{on }  D^2   \\[5mm]
\tilde{V}^i = 0  \quad \text{on } \p  D^2 \quad,
\end{array}
\rg.
\ee
which implies, by $W^{2,q}$ estimates on $\tilde{V}^i$ and Sobolev embedding, that
\be\label{eq:EstVtilde}
\left(\int_{D^2} |\nabla \tilde{V}^i|^2\right) ^{\frac{1}{2}} \leq C \left( \int_{D^2} |\tilde{F} ^i|^q \right)^{\frac{1}{q}}.
\ee
Now, using that the left hand side is invariant under rescaling while the right hand side has a scaling factor given by the area and the definition of $\tilde F^i$, we obtain
\be\label{eq:decayV}
\left(\int_{B_{\rho}(x_0)} |\nabla V^i|^2\right) ^{\frac{1}{2}} \leq C \rho^{2-\frac{2}{q}} \left( \int_{B_{\rho(x_0)}} |F^i|^q \right)^{\frac{1}{q}} \leq C \rho^\alpha \quad \text {for some } \alpha>0\quad,
\ee
where in the last inequality we used that $F^i \in L^q(D^2)$ and $1<q<2$.

At last we study the decay of the $L^2$ norm of the gradient of the harmonic function $W^i$ solving \eqref{eq:W'}. Notice that, since $W^i$ is harmonic, then $\triangle |\nabla W^i|^2= 2 |\nabla^2 W^i|^2\geq 0$. An elementary calculation shows that for any non negative subharmonic function $f$ in $\R^n$ one has  $d/dr (r^{-n} \int_{B_r} f)\geq 0$ (see also \cite{RiCours} Lemma III.1). It follows that
\be\label{eq:decayW}
\int_{B_{\delta \rho}(x_0)}|\nabla W^i|^2 \leq \delta^2 \int_{B_\rho(x_0)}|\nabla W^i|^2\leq C \delta^2 \int_{B_\rho(x_0)}|\nabla A|^2\quad,
\ee
where, in the last inequality, we used that $W^i$ solves \eqref{eq:W'}.

Collecting \eqref{eq:decayp}, \eqref{eq:decayV} and \eqref{eq:decayW} gives
$$\int_{B_{\delta \rho}(x_0)}|\nabla A|^2 \leq C \delta^2 \int_{B_\rho(x_0)}|\nabla A|^2+  C \epsilon^2  \int_{B_\rho(x_0)}|\nabla A|^2 + C \rho^\alpha $$
where the stricly positive constants $\alpha$ and $C$ are independent of $\epsilon$, $\delta$, $x_0$ and $\rho$. Now, in the beginning of Step 2, choose $\epsilon$ and $\rho_0$ such that $C \epsilon^2<\frac{1}{4}$, moreover choose $\delta$ in \eqref{eq:decayW} such that $C \delta^2< \frac{1}{4}$; it follows that for every $x_0 \in B_{\frac{1}{2}}(0)$ and every $\rho<\rho_0$ we have
$$\int_{B_{\delta \rho}(x_0)} |\nabla A|^2 < \frac{1}{2} \int_{B_{ \rho}(x_0)} |\nabla A|^2 +C \rho^\alpha \quad \text{for some } \alpha>0. $$
It is a standard fact which follows by iterating the inequality (see for instance Lemma 5.3 in \cite{KMS}) that there exist $C,\alpha>0$ such that for every $x_0 \in B_{\frac{1}{2}}(0)$ and every $\rho<\rho_0$
\be\label{eq:decayA}
\int_{B_{\rho}(x_0)} |\nabla A|^2 \leq C \rho^\alpha\quad,
\ee
which implies our initial claim \eqref{eq:decayNA}.

Now we easily get that there exists $\beta>0$ such that
\be\label{eq:decayLapA}
\sup_{x_0\in B_{\frac{1}{2}}(0), \rho<\frac{1}{4}} \quad \frac{1}{\rho^{\beta}} \int_{B_\rho(x_0)} |\triangle A| < \infty\quad.
\ee
Indeed, by \eqref{eq:decayNA} and \eqref{eq:SystAB}, for every $x_0 \in B_{\frac{1}{2}}(0)$ and $\rho<\frac 1 4$ we obtain
$$
\begin{array}{l}
\ds\int_{B_\rho(x_0)}|\triangle A|\leq \int_{B_\rho(x_0)} |\nabla B| |\nabla A| + \int_{B_{\rho}(x_0)} |F| \\[5mm]
\ds\quad\quad\leq \|\nabla B\|_{L^2(D^2)}\left[ \int_{B_\rho(x_0)}|\nabla A|^2 \right]^{\frac{1}{2}}+ |B_{\rho}(x_0)|^{\frac{1}{q'}} \|F\|_{L^q(D^2)} \leq C \rho^\beta\quad.
\end{array}
$$
By a classical result of Adams \cite{Ad}, \eqref{eq:decayLapA} implies  that $\nabla A \in L^p_{loc} (B_{\frac{1}{2}}(0))$ for some $p>2$. With analogous arguments one gets that  $\nabla A \in L^p_{loc} (D^2)$ for some $p>2$.
\\

\emph{Step 3}: $\bH \in L^p_{loc} (D^2)$ for some $p>2$.
\\From Step 2 we obtain that $\nabla (\Re \bR)$ and $\nabla (\Re S)$ are in $L^{p}_{loc}(D^2)$ for some $p>2$; recalling that, by Lemma \ref{Lm:RS}, $\nabla^2 (\Im \bR)$ and $\nabla^2 (\Im S)$ are in $L^{q}(D^2)$ for every  $1<q<2$ then, by Sobolev embedding,  $\nabla  \bR$ and $\nabla S$ are in $L^{p}_{loc}(D^2)$ for some $p>2$.
\\Using equation \eqref{eq:DzR} and observing that $<\p_z \bP, \p_{\bar{z}} \bP>=\frac{1}{2} e^{2\lambda}$, a simple computation gives
\be \label{eq:DzRpzP}
D_z\bR \llcorner \p_{\bar z}\bP= \frac{e^{2\lambda}}{2} \overline{\bL}- <\overline{\bL},\p_{\bar z} \bP> \p_z\bP-i e^{2\lambda} \bH\quad.
\ee
Using the definition of $\p_z$ and $\p_{\bar{z}}$ we write
\begin{eqnarray}
\Im \left[ <\overline{\bL},\p_{\bar z} \bP> \p_z \bP \right]&=&\frac {1} {4} \Big[ -<\p_{x^1} \bP, \Re \bL> \p_{x^2} \bP- <\p_{x^1} \bP, \Im \bL> \p_{x^1} \bP \nonumber \\
&& \quad + <\p_{x^2} \bP, \Re\bL> \p_{x^1} \bP - <\p_{x^2} \bP, \Im\bL> \p_{x^2} \bP\Big] \quad.  \label{eq:pzPL}
\end{eqnarray}
On the other hand, \eqref{eq:DzS} gives
\begin{eqnarray}
<\p_{x^1} \bP, \Re \bL>&=& 2 \Re(\p_z S)+<\p_{x^2} \bP, \Im \bL> \label{pxPL1} \\
<\p_{x^2} \bP, \Re \bL>&=&-2 \Im(\p_z S)-<\p_{x^1} \bP, \Im \bL> \quad .\label{pxPL2}
\end{eqnarray}
Insering \eqref{pxPL1} and \eqref{pxPL2} in \eqref{eq:pzPL} we obtain after some elementary computations

\be\label{eq:ImdzPL}
\Im\left( <\p_{\bar{z}} \bP, \overline{\bL}> \p_z \bP \right)= \Re\left[\p_z S (i\p_{\bar{z}} \bP) \right]-2 \Re \left[<\p_z\bP, \Im \bL> \p_{\bar{z}}\bP \right].
\ee
Therefore, combining \eqref{eq:DzRpzP} and \eqref{eq:ImdzPL} we get that
\be \label{eq:HLploc}
e^{2 \lambda} \bH=-\Im\left[D_z\bR \llcorner \p_{\bar z}\bP\right]- \frac{e^{2\lambda}}{2} \Im \bL-\Re\left[\p_z S (i\p_{\bar{z}} \bP) \right]+ \Re \left[<\p_z\bP, \Im \bL> \p_{\bar{z}}\bP \right]; 
\ee
since by Step 2  $D_z\bR$ and $D_z S$ are in $L^p_{loc}(D^2)$ for some $p>2$ and by Lemma \ref{Lm:RS} $\nabla (\Im L) \in L^{(2,\infty)}(D^2)$, we  conclude that $\bH \in L^p_{loc}(D^2)$ for some $p>2$.
\\

\emph{Step 4}: smoothness of $\bP$ by a bootstrap argument.
\\Since $\bP$ is a conformal parametrization, then $\triangle \bP= e^{2\lambda} \bH$ and by Step 3 we infer that $\bP \in W^{2,p}_{loc}(D^2)$ for some $p>2$. Now the Willmore equation in divergence form (see \eqref{eq:ConsW} for the free problem and \eqref{eq:ConsWconf} for the conformal-constraint problem) becomes subcritical in $\bH$: written in local coordinates it has the form 
$$\triangle \bH= \tilde H \quad \text{with } \tilde H \in W^{-1,\frac{p}{2}}_{loc}(D^2)  $$
then $\bH \in W^{1,\frac{p}{2}}_{loc}(D^2)$ and by Sobolev embedding $\bH \in L^{\frac{2p}{4-p}}$, notice that ${\frac{2p}{4-p}}>p$ since $p>2$; reinserting this information in the same equation iteratively we get $\bH \in W^{1,p}_{loc}(D^2)$ for every $p < \infty$, therefore $\bP \in W^{3,p}_{loc}(D^2)$ for every $p < \infty$. Inserting this information into the same equation gives that $\bH \in W^{2,p}_{loc} (D^2)$
for every $p < \infty$, therefore $\bP \in  W^{4,p}_{
loc} (D^2)$ for every $p < \infty \ldots$  continuing this bootstrap argument gives
that $\bP \in W^{k,p}_{loc} (D^2)$ for every $k > 0$, $1 < p < \infty$ which implies that $\bP \in C^ {\infty}_{loc}(D^2)$. 
\end{proof}

\section{A priori geometric estimates under curvature conditions}\label{Sec:APrioriGeomEst}
\reset
\subsection{Diameter bound from below on a minimizing sequence}
We start by computing the Willmore functional and the Energy functional on small geodesic 2-spheres in a Riemannian manifold $(M^m,h)$ of arbitrary codimension; the corresponding expansions in codimension 1 were obtained by the first author in \cite{Mon1}, \cite{Mon2}, \cite{KMS}. First we introduce some notation. 

Let $(M^m,h)$ be an $m$-dimensional Riemannian manifold. Fix a point $\bar p$ and a $3$-dimensional subspace ${\frak S}< T_{\bar p} M$ of the tangent space to $M$ at $\bar p$. Denote with $S_{\bar p,\rho}^{\frak S} \subset M$ the geodesic sphere obtained exponentiating the sphere in ${\frak S}$ of center $0$ and radius $\rho$.  An equivalent way to define  is the following: consider normal coordinates $(x^1,\ldots,x^m)$ in $M$ centred at $\bar p$ such that $(\frac{\partial} {\partial x^1}|_0, \frac{\partial} {\partial x^2}|_0, \frac{\partial} {\partial x^3}|_0)$ are an orthonormal basis of ${\frak S}$, then $S_{\bar p,\rho}^{\frak S}:=\{(x^1)^2+(x^2)^2+(x^3)^2=\rho^2\}\cap \{x^4=\ldots=x^{m}=0\}$.
Let us denote
\be \label{eq:defRT}
R_{\bar p}({\frak S}):= \sum_{i\neq j, i,j=1,2,3} \bar K_{\bar p} \left(\frac{\partial} {\partial x^i}|_0, \frac{\partial} {\partial x^j}|_0\right) 
\ee 
where $\bar K_{\bar p} (\frac{\partial} {\partial x^i}|_0, \frac{\partial} {\partial x^j}|_0)$ denotes the sectional curvature of $(M,h)$ computed on the plane spanned by $(\frac{\partial} {\partial x^i}|_0, \frac{\partial} {\partial x^j}|_0))$ contained in $T_{\bar p} M$.  
\begin{Lm}\label{Lm:ExpWE}
We have the following expansions for the Willmore functional,the Energy functional and the area for small spheres $S_{\bar p,\rho}^{\frak S}$ defined above:
\be
W(S_{\bar p,\rho}^{\frak S}):=\int_{S_{\bar p,\rho}^{\frak S}} |H|^2 d\mu_g= 4 \pi-\frac{2 \pi}{3} R_{\bar p}({\frak S}) \rho^2 +o(\rho^2)
\ee

\be
F(S_{\bar p,\rho}^{\frak S}):=\frac{1}{2}\int_{S_{\bar p,\rho}^{\frak S}} |\mathbb I|^2 d\mu_g= 4 \pi-\frac{2 \pi}{3} R_{\bar p}({\frak S}) \rho^2 +o(\rho^2).
\ee

\be
A(S_{\bar p,\rho}^{\frak S})= 4 \pi \rho^2+ o(\rho^2)\quad.
\ee
In particular, if at some point $\bar p \in M$ there exists a  3-dimensional subspace ${\frak S}<T_{\bar p}M$ such that $R_{\bar p}({\frak S})>6$ then $\inf_{\bP \in {\cal F}_{\Sp ^2}} (W+A)(\bP)< 4 \pi$ and  $\inf_{\bP \in {\cal F}_{\Sp ^2}} (F+A)(\bP)< 4 \pi$.
\hfill $\Box$
\end{Lm}

\begin{proof}
Let $r<Inj_{M,h}(\bar p)$ be less than the injectivity radius of $(M,h)$ at $\bar p$, then the exponential map $Exp_{\bar p}:B_r(0)\subset  T_{\bar p}M\to M$ is a diffeomorphism on the image. Call 
$$\tau:=Exp_{\bar p}({\frak S}\cap B_r(0))\quad,$$  
the image under the exponential map of the subspace ${\frak S}$. Observe that $\tau$ is a 3-dimensional submanifold which is geodesic at $\bar p$ (i.e. every geodesic in $\tau$ starting at $\bar p$ is a geodesic of $M$ at $\bar p$) so the second fundamental form ${\mathbb I}_{{\cal S \hookrightarrow M}}$ of $\tau$ as submanifold of $M$ vanishes at $\bar p$ (for the easy proof see for example \cite{DoC} Proposition 2.9 page 132). Endow $\tau$ with the metric induced by the immersion and observe that by the Gauss equations applied to $\tau\hookrightarrow M$ we get that the sectional curvatures of $\tau$ at $\bar p$ coincide with the corresponding sectional curvatures of $M$ at $\bar p$. Therefore the scalar curvature $R^\tau(\bar p)$ of $\tau$ at $\bar p$ coincide with $R_{\bar p}({\frak S})$ (see for example \cite{Chav} page 50 for the definition of scalar curvature via sectional curvature):
\be\label{eq:Scal}
R^\tau(\bar p)=R_{\bar p}({\frak S})\quad.
\ee
Now consider the geodesic sphere $S_{\bar p,\rho} \hookrightarrow \tau$ in the Riemannian manifold $\tau$ and observe that the composition of the immersions $S_{\bar p,\rho} \hookrightarrow \tau \hookrightarrow M$ coincides with $S_{\bar p,\rho}^{\frak S}$; call $\pi_{\bn_{S_\hookrightarrow M}}, \pi_{\bn_{S\hookrightarrow \tau}} $ and $\pi_{\bn_{\tau\hookrightarrow M}}$ the normal projections onto the normal bundles respectively of $S_{\bar p,\rho}$ relative to $M$, of $S_{\bar p,\rho}$ relative to $\tau$ and of $\tau$ relative to $M$ (i.e. for example in the second case we mean the intersection of the normal bundle of $S_{\bar p,\rho}$ as immersed in $M$ with the tangent bundle of $\tau$) then we have the orthogonal decomposition
\be\label{eq:pin}
\pi_{\bn_{S_\hookrightarrow M}}= \pi_{\bn_{S\hookrightarrow \tau}} + \pi_{\bn_{\tau\hookrightarrow M}}\quad.
\ee
By definition of second fundamental form we get for all $X,Y$ tangent vectors to $S_{\bar p,\rho}$
\be\label{eq:Dec2ff}
{\mathbb I}_{S_\hookrightarrow M}(X,Y):=\pi_{\bn_{S_\hookrightarrow M}}(\D_X Y)=\pi_{\bn_{S\hookrightarrow \tau}}(\D_X Y) + \pi_{\bn_{\tau\hookrightarrow M}} (\D_X Y)=:{\mathbb I}_{S_\hookrightarrow \tau}(X,Y)+{\mathbb I}_{\tau_\hookrightarrow M}(X,Y).
\ee
Therefore we obtain
\be\label{eq:est2ff}
|{\mathbb I}_{S_\hookrightarrow \tau} |^2\leq |{\mathbb I}_{S_\hookrightarrow M}|^2\leq |{\mathbb I}_{S_\hookrightarrow \tau} |^2+|{\mathbb I}_{\tau_\hookrightarrow M} |^2\quad,
\ee
and recalled that $\bH_{S_\hookrightarrow M}:= \frac{1}{2} \sum_{i=1}^2 \left[{\mathbb I}_{S_\hookrightarrow M}(\vec{e} _i, \vec{e}_i)\right]$ where $\{\vec{e} _1,\vec{e} _2\} $ is an orthonormal frame of $T_xS_{\bar p,\rho}$,
\be\label{eq:estH}
|\bH_{S_\hookrightarrow \tau} |^2\leq | \bH _{S_\hookrightarrow M}|^2\leq |\bH _{S_\hookrightarrow \tau} |^2+\frac{1}{2}|{\mathbb I}_{\tau_\hookrightarrow M} |^2\quad.
\ee
Since $S_{\bar p,\rho} \hookrightarrow \tau$ is a geodesic sphere in the $3$-dimensional manifold $\tau$,  we can use the expansions of \cite{Mon1}, \cite{Mon2}, \cite{KMS} for geodesic spheres in $3$-manifolds (more precisely see Proposition 3.1 in \cite{Mon1} and Lemma 2.3 in \cite{KMS}) and obtain that as $\rho \to 0$
\begin{eqnarray}
\frac{1}{2}\int_{S_{\bar p,\rho}} |{\mathbb I}_{S_\hookrightarrow \tau} |^2 d\mu_g&=& 4 \pi- \frac{2\pi}{3} R^\tau(\bar p) \rho^2+o(\rho^2) \label{eq:2ffSt}\\ 
\int_{S_{\bar p,\rho}} |\bH_{S_\hookrightarrow \tau}|^2 d\mu_g&=& 4 \pi- \frac{2\pi}{3} R^\tau(\bar p) \rho^2+o(\rho^2). \label{eq:HSt}
\end{eqnarray}
Observe that $\int_{S_{\bar p,\rho}} d\mu_g=O(\rho^2)$ and since ${\mathbb I}_{\tau_\hookrightarrow M}(\bar p)=0$ we have that $|{\mathbb I}_{\tau_\hookrightarrow M} |^2|_{S_{\bar p,\rho}}\to 0$ as $\rho\to 0$. Therefore $\int_{S_{\bar p,\rho}} |{\mathbb I}_{\tau_\hookrightarrow M}|^2 d\mu_g=o(\rho^2)$ and integrating the estimates \eqref{eq:est2ff},\eqref{eq:estH} on $S_{\bar p, \rho}$, using \eqref{eq:2ffSt},\eqref{eq:HSt}, we conclude that
\begin{eqnarray}
\frac{1}{2}\int_{S_{\bar p,\rho}} |{\mathbb I}_{S_\hookrightarrow M} |^2 d\mu_g&=& 4 \pi- \frac{2\pi}{3} R^\tau(\bar p) \rho^2+o(\rho^2) \label{eq:2ffSt}\\ 
\int_{S_{\bar p,\rho}} |\bH_{S_\hookrightarrow M}|^2 d\mu_g&=& 4 \pi- \frac{2\pi}{3} R^\tau(\bar p) \rho^2+o(\rho^2). \label{eq:HSt}
\end{eqnarray}
The expansion of the area is straightforward.
\end{proof}

The following Lemma is a variant for weak branched immersions of a Lemma proved by Simon \cite{SiL}, notice a similar statement is also present in \cite{KMS} in case of smooth immersions. We include it here for completeness.
  
\begin{Lm}\label{lem:MonFor}
Let $\bP\in {\mathcal F}_{\Sp^2}$ be a weak branched immersion with finite total curvature of $\Sp^2$ into the Riemannian manifold $(M^m,h)$. Assume  $W(\bP)+A(\bP) \leq \Lambda$.
Then there exists a constant $C=C(\Lambda,M)$ such that
\be\label{eq:Areadiam}
A(\bP)\leq C \left[\diam_M (\bP(\Sp^2))\right]^2\quad.
\ee
\hfill $\Box$
\end{Lm}

\begin{proof}
By Nash's theorem, there is an isometric embedding $I:M \hookrightarrow \R^s$ for some $s \in \N$.
The second fundamental forms of $\bP$, $I \circ \bP$ and $I$ are related by the formula holding $vol_g$-a.e. on $\Sp^2$
$$
{\mathbb I}_{I \circ \bP}(\cdot,\cdot) = 
dI|_{\bP} \circ  {\mathbb I}_{\bP}(\cdot,\cdot) \oplus ({\mathbb I}_{I} \circ \bP) (d\bP,d\bP)\quad.
$$
Taking the trace and squaring yields for an orthonormal basis $\bbe_i$ of $\bP_*(T\Sp^2)$ that $vol_g$-a.e. on $\Sp^2$
$$
|\bH_{I \circ \bP}|^2 = 
|H_{\bP}|^2 + \Big|\sum_{i = 1}^2 {\frac{1}{2}\mathbb I}_I \circ \bP (\bbe_i,\bbe_i)\Big|^2 
\leq |\bH_{\bP}|^2 + \frac{1}{2} |{\mathbb I}_I|^2 \circ \bP\quad.
$$
Analogously, taking the squared norms, one gets
$$
|{\mathbb I}_{I \circ \bP}|^2 = 
|{\mathbb I}_{\bP}|^2 + \Big|\sum_{i,j = 1}^2 {\mathbb I}_I \circ \bP (\bbe_i,\bbe_j)\Big|^2 
\leq |{\mathbb I}_{\bP}|^2 +  |{\mathbb I}_I|^2 \circ \bP.
$$ 
Integrating we obtain that $\bP$ is a weak branched immersion with finite total curvature and 
\be\label{eq:WIPhi}
W(I \circ \bP) \leq W(\bP) + C A(\bP)\leq C_{\Lambda,M},
\ee
where $C = \frac{1}{2} \max |{\mathbb I}_I|^2$.
\\Let $\{b^1,\ldots,b^N\}$ be the branch points of $\bP$ and for small $\ep>0$ let $K_\ep:=\Sp^2 \setminus \cup_{i=1}^{N} B_{\ep}(b^i)$. Then $\bP|_{K_\ep}$ is a weak immersion \emph{without branch points} of  the surface with smooth boundary $K_\ep$. Recall that for a smooth vector field $\bX$ on $\R^s$, the tangential divergence of $\bX$ on $(I\circ \bP) (\Sp^2)$ is defined by $$div_{I\circ \bP} \bX:= \sum_{i=1}^2 <d \bX \cdot \bbf_i, \bbf_i>,$$
where $\bbf_i$ is an orthormal frame on $(I\circ \bP)_* (T\Sp^2)$. Now, from the first part of the proof of Lemma A.3 of \cite{RivDegen}, the tangential divergence theorem ((A.18) of the mentioned paper) holds for a weak immersion of a surface with boundary in $\R^s$ withouth branch points and
\be\label{eq:DivThmBoundary}
\int_{(I\circ \bP) (K_\ep)} div_{I\circ \bP} \bX \, dvol_g = \int_{\cup_{i=1}^N [I\circ \bP (\p B_{\ep}(b^i))]} <\bX, \vec{\nu}> dl   -2 \int_{(I\circ \bP) (K_\ep)} <\bH_{I \circ \bP}, \bX> \, dvol_g,
\ee
where $\vec{\nu}$ is the unit limiting tangent vector to $(I \circ \bP) (K_\ep)$ on $I\circ \bP) (\p K_\ep)$ orthogonal to it and oriented in the outward direction. Since $\bP$ is Lipschitz  by assumption and since $\bX$ and $\vec{\nu}$ are trivially bounded, it follows that
$$\int_{\cup_{i=1}^N [I\circ \bP (\p B_{\ep}(b^i))]} <\bX, \vec{\nu}> dl \to 0 \quad \text{ as } \ep \to 0;$$
therefore the tangential divergence theorem still holds on a weak branched immersion:
\be\label{eq:DivThm}
\int_{(I\circ \bP) (Sp^2)} div_{I\circ \bP} \bX \, dvol_g =  - 2\int_{(I\circ \bP) (\Sp^2)} <\bH_{I \circ \bP}, \bX> \, dvol_g.
\ee
Now, as in \cite{SiL}, we choose $\bX(\vec{x}):=\vec{x}-\vec{x}_0$ where $\vec{x}_0\in (I\circ \bP) (\Sp^2)$. Then, observing that $div_{I\circ \bP} \bX=2$, by Schwartz inequality we get
$$ A(I\circ \bP)\leq \diam_{\R^s} [(I\circ \bP) (\Sp^2)]\; W(I\circ \bP)^{\frac{1}{2}} \; A(I\circ \bP)^{\frac{1}{2}}.$$
The last inequality, together with \eqref{eq:WIPhi}, the fact that $A(I\circ \bP)=A(\bP)$ and $\diam_{\R^s} [(I\circ \bP) (\Sp^2)]\leq \diam_M (\bP(\Sp^2))$ ensured by the isometry $I$, gives that
$$A(\bP) \leq C_{\Lambda,M} \left[\diam_M (\bP(\Sp^2))\right]^2\quad. $$

\end{proof} 
 
In the following lemma we collect some inequalities linking the geometric quantities  of two close metrics. This will be useful for working locally in normal coordinates (the analogous lemma in codimension one and for smooth immersions appears in \cite{KMS}). 

\begin{Lm} \label{lem:localcomparison}
Let $h_{1,2}$ be Riemannian metrics on a manifold $M^m$, with norms satisfying
$$
(1+\epsilon)^{-1} \|\cdot \|_1 \leq \|\,\cdot \,\|_2 \leq (1+\epsilon) \|\,\cdot \,\|_1
\quad \mbox{ for some } \epsilon \in (0,1].
$$
For any weak branched  immersion with finite total curvature $\bP\in {\mathcal F}_{\Sp^2}$, the following 
inequalities hold almost everywhere on $\Sigma$ for a universal $C < \infty$: 
\begin{itemize}
\item $vol_{g_1} \leq (1+C \epsilon) vol_{g_2}$, where $g_{1,2} = \bP^\ast (h_{1,2})$ and $vol_{g_{1,2}}$ are the associated area forms; 
\item $|{\mathbb I}_1|_{1}^2 \leq \big(1+ C(\epsilon + \delta)\big) |{\mathbb I}_2|_{2}^2 + C \delta^{-1} |\Gamma|^2_{h_1} \circ \bP$        for any $\delta \in (0,1]$, where $\Gamma := D^{h_1}-D^{h_2}$ and $D^{h_i}$ is the covariant derivative with respect to the metric $h_i$.
\item $|H_1|_1^2\leq \big(1+ C(\epsilon + \delta)\big) |H_2|_2^2 + C \delta^{-1} |\Gamma|_{h_1}^2 \circ \bP$ for any $\delta \in (0,1]$ and $\Gamma$ defined above.
\end{itemize}
\hfill$\Box$
\end{Lm}

\begin{proof} To compare the 
Jacobians of $\bP$ with respect to $h_{1,2}$, we use $|\,\cdot \,|_{g_1} \leq (1+\varepsilon) |\,\cdot \,|_{g_2}$
and compute for $v,w \in T_p \Sigma$ with $g_2(v,w) = 0$ 
$$
|v \wedge w|_{g_1}^2 = |v|_{g_1}^2 |w|_{g_1}^2 - g_1(v,w)^2 
\leq (1+\epsilon)^4  |v|_{g_2}^2 |w|_{g_2}^2 
= (1+\epsilon)^4  |v \wedge w|_{g_2}^2. 
$$
This proves the first inequality. Next we compare the norms for a bilinear map 
$B:T_p\Sigma \times T_p \Sigma \to T_{\bP(p)} M$ for $p$ not a  branch point. Choose a basis 
$v_\alpha$ of $T_p\Sigma$ such that $g_1(v_\alpha,v_\beta) = \delta_{\alpha \beta}$
and $g_2(v_\alpha,v_\beta) = \lambda_\alpha \delta_{\alpha \beta}$. 
Then
$$
\lambda_\alpha = |v_\alpha|_{g_2} \leq (1+ \epsilon) |v_\alpha|_{g_1} = 1 + \epsilon,
$$
and putting $w_\alpha = v_\alpha/\lambda_\alpha$ we obtain
$$
|B|_1^2 = \sum_{\alpha,\beta = 1}^2 
\lambda_\alpha^2 \lambda_\beta^2 |B(w_\alpha,w_\beta)|_{h_1}^2
\leq (1+C\epsilon) \sum_{\alpha,\beta = 1}^2 |B(w_\alpha,w_\beta)|_{h_2}^2
= (1+C\epsilon)|B|_2^2.
$$
Now denote by $\pi_{\bn_{1,2}}:T_{\bP(p)}M \to (\bP_*(T_p \Sigma))^{\perp_{h_{1,2}}}$ the orthogonal 
projections onto the normal spaces  with respect to $h_{1,2}$. Then  for 
any $\delta \in (0,1]$ and almost every $p \in \Sigma$ we have the following estimate (by approximation with smooth immersions locally away the branch points) 
\begin{eqnarray*}
\big|{\mathbb I}_1\big|_1^2 & = & \big|\pi_{\bn_{1}} (D^{h_1}(\nabla\bP))\big|_1^2\\
& \leq & \big|\pi_{\bn_{2}} (D^{h_1}(\nabla \bP))\big|_1^2\\
& \leq & \big| \pi_{\bn_{2}} \big( D^{h_2}(\nabla \bP) + \Gamma \circ \bP (\nabla \bP,\nabla \bP) \big)\big|_1^2\\
& \leq & (1+\delta) \big|\pi_{\bn_{2}} D^{h_2}(\nabla \bP)\big|_1^2 + C \delta^{-1} |\Gamma|_{h_1}^2 \circ \bP\\
& \leq & (1+\delta)(1+C\varepsilon) |{\mathbb I}_2|_2^2 + C \delta^{-1} |\Gamma|_{h_1}^2 \circ \bP\quad.
\end{eqnarray*}
This proves the second inequality.
The proof of the third inequality is analogous:
\begin{eqnarray*}
\big|H_1\big|_1^2 & = & \frac{1}{2}\big| {\mathbb I}_1(v_1,v_1)+ \mathbb{I}_1(v_2,v_2)\big|_1^2= \frac{1}{2} \big|\pi_{\bn_{1}} (D_{v_1}^{h_1}(\partial_{v_1}\bP)+D_{v_2}^{h_1}(\partial_{v_2}\bP))\big|_1^2\\
& \leq & \frac{1}{2} \big|\pi_{\bn_{2}} (D_{v_1}^{h_1}(\partial_{v_1}\bP)+D_{v_2}^{h_1}(\partial_{v_2}\bP))\big|_1^2\\
& \leq & \frac{1}{2} \big|\pi_{\bn_{2}} \left(D_{v_1}^{h_2}(\partial_{v_1}\bP)+D_{v_2}^{h_2}(\partial_{v_2}\bP)+ \Gamma \circ \bP (\partial_{v_1} \bP,\partial_{v_1} \bP)+ \Gamma \circ \bP (\partial_{v_2} \bP,\partial_{v_2} \bP)\right)\big|_1^2\\
& \leq & \frac{1}{2}(1+\delta)   \big|\pi_{\bn_{2}} (D_{v_1}^{h_2}(\partial_{v_1}\bP)+D_{v_2}^{h_2}(\partial_{v_2}\bP))\big|_1^2 + C \delta^{-1} |\Gamma|_{h_1}^2 \circ \bP\\
& \leq & \frac{1}{2}(1+\delta)(1+C\epsilon)   \big|\pi_{\bn_{2}} (D_{w_1}^{h_2}(\partial_{w_1}\bP)+D_{w_2}^{h_2}(\partial_{w_2}\bP))\big|_1^2 + C \delta^{-1} |\Gamma|_{h_1}^2 \circ \bP\\
& \leq & (1+\delta)(1+C\epsilon) |H_2|_2^2 + C \delta^{-1} |\Gamma|_{h_1}^2 \circ \bP\quad.
\end{eqnarray*}

\end{proof}

Since we are assuming an upper area bound, the lower diameter bound will follow combining  
Lemma \ref{Lm:ExpWE} and the fact below (which generalizes to arbitrary codimension and non smooth immersions Proposition 2.5 in \cite{KMS}, the proof is similar but we include it here for completeness). 
    
\begin{Prop}\label{prop:LBdiamEa}
Let $M^m$ be a compact Riemannian $m$-manifold and consider a sequence $\bP_k \in {\mathcal F}_{\Sp^2}$ such that $\sup_k (W+A)(\bP_k) \leq \Lambda$. If $\diam \bP_k(\Sp^2) \to 0$, then
$$
\lim_{k\to \infty} A(\bP_k)\to 0,\quad  \limsup_k F(\bP_k) \geq 4\pi \quad \text{ and } \quad \limsup_k W(\bP_k) \geq 4\pi \quad  \quad .
$$
\hfill $\Box$
\end{Prop}

\begin{proof} The first statement follows directly from Lemma \ref{lem:MonFor}. Let us prove the second one.  After passing to a subsequence, we may assume that the
$\bP_k(\Sp^2)$ converge to a point $\bar p \in M$. For given $\epsilon \in (0,1]$
we choose $\rho > 0$, such that in Riemann normal coordinates $x \in B_\rho(0) \subset \R^m$
$$
\frac{1}{1+\epsilon} |\,\cdot \,|_{\eucl} \leq |\,\cdot \,|_h \leq (1+\epsilon) |\,\cdot \,|_{\eucl} \quad \mbox{ and } \quad |\Gamma_{ij}^k(x)| \leq \varepsilon\quad, 
$$
where, of course, $|\,\cdot \,|_{\eucl}$ is the norm associated to the euclidean metric given by the coordinates and $|\,\cdot\,|_h$ is  the norm in metric $h$.
We have $\bP_k(\Sp^2) \subset B_\rho(x_0)$ for large $k$. Denoting by ${\mathbb I}^e,g^e_k$ the 
quantities with respect to the coordinate metric, we get from Willmore's inequality
and Lemma \ref{lem:localcomparison}
$$
4\pi \leq \frac{1}{2} \int_{\Sp^2} |{\mathbb I}^e_{\bP_k}|_e^2\,d\mu_{g^e_k}
\leq (1+C\epsilon) (1+\delta) \frac{1}{2} \int_{\Sp^2} |{\mathbb I}_{\bP_k}|^2\,d\mu_{g_k}
+ C(\delta) \epsilon^2\, \Area_{g_k}(\Sp^2)\quad.
$$
Since $\Area_{g_k}(\Sp^2) \leq C$ by assumption, we may let first $k \to \infty$, 
then $\epsilon \searrow 0$ and finally $\delta \searrow 0$ to obtain 
$$
\liminf_{k \to \infty} F(\bP _k) \geq 4\pi\quad.
$$
The proof for $W$ is analogous.  
\end{proof}

\section{Proof of the existence theorems}\label{Sec:Existence}
\reset
\noindent {\bf Proof of theorem~\ref{TeoExWillCurv}}
Let $\bP_k \subset {\mathcal F}_{\Sp^2}$ be a minimizing sequence of $F_1=F+A$, as before we can assume that $\bP_k$ are conformal; clearly there is a uniform upper bound on the areas and on the $L^2$ norms of the second fundamental forms $\bP_k$:
\begin{eqnarray}
\sup_k \int_{\Sp^2} |{\mathbb I}_k|^2 dvol_{g_{\bP_k}} &\leq& C <\infty, \label{boundA2} \\
\sup_k \Area_{g_{\bP_k}} (\Sp^2)  &\leq& C <\infty \quad.\label{boundArea}
\end{eqnarray}
 Since we are assuming that $R_{\bar p}({\frak S})>6$ for some point $\bar p$ and some 3-dimensional subspace ${\frak S}$, by Lemma \ref{Lm:ExpWE} we have 
\be\label{inf:L}
\inf_{\bP\in {\mathcal F}_{\Sp^2}} F_1(\bP) < 4\pi\quad.
\ee  
Therefore, Proposition \ref{prop:LBdiamEa} yelds
\be\label{LDB}
\liminf_k \diam (\bP_k)(\Sp^2) \geq \frac{1}{C}>0\quad.
\ee 
Now, thanks to \eqref{boundA2}, \eqref{boundArea} and \eqref{LDB}, we can apply the 'Good Gauge Extraction Lemma' IV.1 in \cite{MoRi1} and obtain that up to subsequences and up to repametrization of $\bP_k$ via positive Moebius transformations of $\Sp^2$ the following holds: there exists a finite set of points $\{a^1,\ldots,a^N\}\subset \Sp^2$ such that for every compact subset $K \subset \subset \Sp^2 \setminus \{a^1,\ldots,a^N\}$ (it is enough for our purpouses to take $K$ with smooth boundary ) there exists a constant $C_K$ such that
\be\label{BoundCF}
|\log |\nabla \bP_k|\,| \leq C_K \quad \text{on } K \text{ for every } k\quad.
\ee
Since the parametrization is conformal, then $|\nabla^2 \bP_k|^2= e^{4 \lambda_k} |{\mathbb I}_{\bP_k}|^2$ (where, as usual, $e^{\lambda_k}=|\p_{x^1}\bP_k|=|\p_{x^2}\bP_k|$), and the two estimates \eqref{boundA2}-\eqref{BoundCF} give that $\bP_k|_K$ are equibounded in $W^{2,2}(K)$, therefore by Banach-Alaoglu Theorem together with reflexivity and separability of $W^{2,2}(K)$ imply the existence of a map $\bP_\infty \in W^{2,2}(K)$ such that, up to subsequences,
\be
\bP_k \rightharpoonup \bP_\infty \quad \text{weakly in } W^{2,2}(K)\quad.
\ee
Now by Rellich-Kondrachov Theorem $\p_{x^i} \bP_k \to \p_{x^i} \bP_\infty$ as $k\to \infty$ strongly in $L^p (K)$ for every $1<p<\infty$ and a.e. on $K$. It follows that $\bP_\infty$ is a $W^{1,\infty}\cap W^{2,2}$ conformal immersion of $K$. 
Moreover by the lower semicontinuity under $W^{2,2}$-weak convergence proved in Lemma \ref{lem:LSC} we have
\be\label{eq:lscF}
\int_K |{\mathbb I}_{\bP_\infty}|^2 dvol_{g_{\bP_\infty}} \leq \liminf_{k} \int_K  |{\mathbb I}_{\bP_k}|^2 dvol_{g_{\bP_k}}\quad, 
\ee
and the strong $L^p(K)$ convergence of the gradients implies
\be\label{eq:ConvArea}
\Area_{g_{\bP_k}}(K) \to \Area_{g_{\bP_\infty}}(K)\quad. 
\ee 
Iterating the procedure on a countable increasing family of compact subsets with smooth boundary invading $\Sp^2\setminus\{a^1,\ldots,a^N\}$, via a diagonal argument we get the existence of a $W^{1,\infty}_{loc}\cap W^{2,2}_{loc}$ conformal immersion $\bP_\infty$ of $\Sp^2\setminus\{a^1,\ldots,a^N\}$ into $M$ such that, up to subsequences, 
\be\label{eq:lscF}
\int_{\Sp^2\setminus\{a^1,\ldots,a^N\}} \left(\frac{1}{2}|{\mathbb I}_{\bP_\infty}|^2+1\right) dvolg_{\bP_\infty} \leq \liminf_k  \int_{\Sp^2\setminus\{a^1,\ldots,a^N\}} \left(\frac{1}{2}|{\mathbb I}_{\bP_k}|^2 +1\right) dvolg_{\bP_k} \leq C\quad.
\ee
Now, thanks to the conformality of $\bP_\infty$ on $\Sp^2\setminus\{a^1,\ldots,a^N\}$ and the estimate \eqref{eq:lscF}, we can apply Lemma A.5 of \cite{Riv2} and extend $\bP_{\infty}$ to a weak conformal immersion in $\mathcal F_{\Sp^2}$ possibly branched in a subset of $\{a^1,\ldots,a^N\}$. Since ${\mathbb I}_{\bP_\infty} \in L^2(\Sp^2, vol_{g_{\bP_\infty}})$, inequality \eqref{eq:lscF} implies 
\be\label{InfRealized}
F_1(\bP_{\infty})\leq \liminf_k F_1(\bP_k)=\inf_{\bP \in {\mathcal F}_{\Sp^2}} F_1(\bP)\quad,
\ee
therefore $\bP_{\infty}$ is a minimizer of $F_1$ in ${\mathcal F}_{\Sp^2}$. By  Lemma \ref{lem:dW}, the  functional $F_1$ is Frech\'et differentiable at $\bP_\infty$ with respect to variations $\bw \in W^{1,\infty}\cap W^{2,2} (D^2,T_{\bP_{\infty}} M)$  with compact support in $\Sp^2\setminus \{b^1,\ldots,b^{N_\infty}\}$, where $\{b^1,\ldots,b^{N_{\infty}}\}$ are the branched points of $\bP_\infty$. From the expression of the differentials given in Lemma \ref{lem:dW} we deduce that $\bP_{\infty}$ satisfies the following  area constraint Willmore like equation in conservative form away the branch points, and since $\bP_\infty$ is conformal, the equation writes  
\begin{eqnarray}
8\,e^{-2\la}\,\Re\lf(\D_{\ov{z}}\lf[\pi_{\vec{n}}(\D_z\vec{H})+<\vec{H},\vec{H}_0>\ \p_{\ov{z}}\vec{\Phi}\rg]\rg)=2\ti{R}(\bH)+16 \Re \lf(<\Riem^h(\vec{e}_{\ov{z}},\vec{e}_z) \vec{e}_z, \vec{H}> \vec{e}_{\ov{z}}\rg) \nonumber \\[5mm]
\quad\quad+2\bH+ (D\, R)(T\bP)+2{\frak R}_{\bP}(T\bP)+2 \bar{K}(T\bP) \bH \quad . \label{eq:ConsF1}
\end{eqnarray}
Observe that the difference between this last equation and the Willmore equation \eqref{eq:ConsW} are just terms of the second line which are completely analogous to the curvature terms of the right hand side already appearing in \eqref{eq:ConsW} (see the definitions \eqref{def:frakR} and \eqref{def:DR}). Therefore all the arguments of Sections \ref{Sec:SystY} and \ref{Sec:Regularity} can be repeated including these new terms   and we conclude with the smoothness of $\bP_\infty$ away the branched points. 
\hfill $\Box$
\medskip

\noindent {\bf Proof of theorem~\ref{TeoExWill}}
The proof is completely analogous to the proof of Theorem \ref{TeoExWillCurv} once we observe that the lower bound on the areas $A(\bP_k)\geq\frac{1}{C}>0$ together with Lemma \ref{lem:MonFor} yelds a lower bound on the diameters:
\be\label{LoBDi}
\diam_M  \bP_k(\Sp^2) \geq \frac{1}{C}.
\ee 
Indeed we still have \eqref{boundA2}, \eqref{boundArea} and \eqref{LDB}. Thereofore, as above, we obtain the existence of a minimizer $\bP_\infty \in {\cal F}_{\Sp^2}$,
$$F(\bP_\infty)=\inf_{\bP \in {\cal F}_{\Sp^2}} F(\bP),$$
satisfing the equation  in conservative form 
\be
\label{eq:ConsFConf}
\begin{array}{l}
\ds 8\,e^{-2\la}\,\Re\lf(\D_{\ov{z}}\lf[\pi_{\vec{n}}(\D_z\vec{H})+<\vec{H},\vec{H}_0>\ \p_{\ov{z}}\vec{\Phi}\rg]\rg)=2\ti{R}(\bH)\\[5mm]
\ds \quad\quad+16 \Re \lf(<\Riem^h(\vec{e}_{\ov{z}},\vec{e}_z) \vec{e}_z, \vec{H}> \vec{e}_{\ov{z}}\rg)+ (D\, R)(T\bP)+2{\frak R}_{\bP}(T\bP)+2 \bar{K}(T\bP)\, \bH 
\end{array}
\ee
(now without the Lagrange multiplier $2\bH$) outside the finitely many branched points. The smoothness of $\bP_\infty$ outside the branched points follows as before.
\hfill $\Box$
\medskip

\noindent {\bf Proof of Theorem~\ref{Th:ExWillHom}}
Recall the discussion after the statement of the Theorem, here we just formalize that idea. First of all recall the precise Definition VII.2 in \cite{MoRi1} of a  {\it bubble tree of weak immersions}; for the proof of the present Theorem we just need to recall (actually the rigorous definition is more precise and intricate) that a bubble tree of weak immersions is  an $N+1$-tuple $\vec{T}:=(\vec{f},\vec{\Phi}^1\cdots\vec{\Phi}^N)$, where $N$ is an arbitrary integer, $\vec{f}\in W^{1,\infty}(\Sp^2,M^m)$
and $\vec{\Phi}^i\in{\mathcal F}_{\Sp^2}$ for $i=1\cdots N$  such that $\vec{f}(\Sp^2)=\cup_{i=1}^N \bP^i(\Sp^2)$ and $\vec{f}_*[\Sp^2]=\sum_{i=1}^N \bP^i_*[\Sp^2]$, where for a lipschitz map $\vec{a}\in W^{1,\infty}(\Sp^2,M)$ we denote $\vec{a}_*[\Sp^2]$ the push forward of the current of integration over $\Sp^2$. The set of bubble trees is denoted by ${\cal T}$ and, considered a nontrivial homotopy class $0\neq\gamma \in \pi_2(M^m)$ , the set  of bubble trees such that the map $\vec{f}$ belongs to the homotopy group $\gamma$ is denoted by ${\cal T}_{\gamma}$. 

Consider the lagrangian $\LW$ defined in \eqref{def:LW} and \eqref{def:LWTree}; up to rescaling the metric $h$ by a positive constant we can assume that $\bar{K}\leq 1$ (or analogously instead of 1, in the definition of $\LW$, take a constant $C>max_{M} \bar{K}$). Consider a minimizing sequence $\vec{T}_k \in {\cal T}_{\gamma}$, of bubble trees realizing the homotopy class $\gamma$, for the functional $\LW$. Observe that by Proposition \ref{pr-I.1} we can assume the $\bP_k$ are conformal. By the expression of $\LW$, there is a uniform bound on the $F_1$ functional
\be\label{eq:F1bound}
\limsup_{k\to \infty} F_1(\vec{T}_k)= \limsup_k \int_{\Sp^2} \lf(1+\frac{|{\mathbb I}|^2}{2} \rg) dvol_{g_k} <+\infty\quad,
\ee
moreover, since $\vec{f}_k \in \gamma \neq 0$, we also have
\be\label{eq:diamBound}
\liminf_{k\to \infty} \sum_{i=1}^{N_k} \diam_{M}\lf(\bP_k^i(\Sp^2)\rg)>0\quad,
\ee
therefore we perfectly fit in the assumptions of the compactness theorem for bubble trees (Theorem VII.1 in \cite{MoRi1}). It follows, recalling also Lemma \ref{lem:LSC},  that there exists a limit bubble tree $\vec{T}_{\infty}=(\vec{f}_\infty, \bP^1_{\infty},\ldots, \bP^{N_{\infty}}_{\infty})$ minimizing the Lagrangian $\LW$ in ${\cal T}_{\gamma}$. By the minimality, using Lemma \ref{lem:dW}, we have that each $\bP^i_{\infty}$ satisfies the Euler Lagrange equation of $\LW$ outside the branch points. As remarked in the introduction, the Euler Lagrange equation of $\LW$ coincides with the area-constrained Willmore equation. By the Regularity Theorem \ref{Th:Regularity}, we conclude that each $\bP^i_{\infty}$ is a branched conformal immersion of $\Sp^2$ which is smooth and satisfies the area-constraned Willmore equation outside the finitely many branched points.  
\hfill $\Box$

\medskip

\noindent {\bf Proof of Theorem~\ref{Th:AreaConstraintWill}.}
The arguments are analogous to the proof of Theorem ~\ref{Th:ExWillHom}. Indeed observe that, fixed any ${\cal A}>0$, for a minimizing sequence $\vec{T}_k \in {\cal T}$ of the functional $W_K$, defined in \eqref{def:WK}, under the $\cal{A}$-area constraint
\be\label{ATk}
A(\vec{T}_k):= Area(\vec{f}_k(\Sp^2))={\cal A}\quad,
\ee 
the bound \eqref{eq:F1bound} still holds (by the constrained on the total area and by the boundness of $\bar{K}$ ensured by the compactness of $M$). Moreover, by the monotonicity formula given in Lemma \ref{lem:MonFor}, the area constraint \eqref{ATk} also implies \eqref{eq:diamBound}. Then, as before, we apply the compactness theorem for bubble trees and the thesis follows as above by recalling that the area constraint is preserved in the limit: $$A(\vec{T}_{\infty}):=Area(\vec{f}_\infty(\Sp^2))=\lim_{k\to \infty} Area(\vec{f}_k(\Sp^2))={\cal A}\quad.$$ 
\hfill $\Box$

\section{Appendix}
\reset
\subsection{Useful lemmas for proving the regularity}
In the appendix we prove some technical lemmas used in the paper. In the following we denote $\mathring{H}^{-1}(\C)$ the dual of the homogeneous Sobolev space $\mathring{H}^1(\C)$ (for the standard definition see for instance \cite{Gra2} Definition 6.2.5 )

\begin{Lm}\label{lemmaA2}
For $j,l\in\{1,\ldots,m\}$ let $\gamma^j_l\in (C^0\cap W^{1,2})(\C)$ be such that $\supp \gamma^j_l \in B_2(0)$ and $\|\gamma^j_l\|_{L^\infty(\C)}\leq \epsilon$. For every $\bU \in (L^1_{loc})(\C)$ denote, in distributional sense, 
\be\label{eq:defDz}
(\D_z U)^j:= \p _z U^j+\sum_{k=1}^m \gamma^j_k U^k.
\ee
Then for every $\bY \in (\mathring{H}^{-1}+L^1)(\C)$ with $\Im(D_{\bar z} \bY) \in (\mathring{H}^{-1}+L^1)(\C)$
there exists a unique $\bU \in L^{2,\infty}(D^2)$ with $\Im(\bU) \in W^{1,(2,\infty)}(D^2)$ satisfying

\be \label{eq:DzU=Y}
\lf\{
\begin{array}{l}
\D _z \bU= \bY \quad 	\text{in } {\cal D}'(D^2)\\[5mm]
\Im \bU=0 \quad \text{on } \p D^2\quad.
\end{array}
\rg.
\ee
Moreover the following estimate holds:
$$\|\bU\|_{L^{2,\infty}(D^2)}+ \|\nabla \Im(\bU)\|_{L^{2,\infty}(D^2)}\leq C \left(\|\bY\|_{H^{-1}+L^1(\C)}+\|\Im(\D _{\bar z} \bY)\|_{H^{-1}+L^1(\C)} \right)\quad. $$
\hfill $\Box$
\end{Lm}

\begin{proof}
Let us first construct  $\tilde U^j \in L^{2,\infty}(\C)$ satifying $\D _z \tilde U^j=Y^j $ on $\C$; observe this is equivalent to solve the fixed point problem in $L^{2,\infty}(\C)$
\be\label{eq:Utilde}
\tilde U^j=-\frac{1}{\pi \bar z} * \left(Y^j-\sum_{k=1}^m \gamma^j_k \tilde{U}^k\right):=T(\tilde \bU)\quad.
\ee
We prove that the problem has unique solution by the contraction mapping principle in $L^{2,\infty}$. By the  Young and H\"older inequalities for weak type spaces (see Theorem 1.2.13 and Exercise 1.4.19 in \cite{Gra}) we have
\be \label{eq:estU}
\left\|-\frac{1}{\pi \bar z} * (\sum_{k=1}^m \gamma^j_k \tilde{U}^k )\right\|_{L^{2,\infty}(\C)}\leq C \left\|-\frac{1}{\pi \bar z}\right\|_{L^{2,\infty}} \left\|\sum_{k=1}^m \gamma^j_k \tilde{U}^k \right\|_{L^1(\C)} \leq C \epsilon \|\tilde \bU\|_{L^{2,\infty}(\C)}\quad.
\ee
We Choose $\epsilon>0$ such that $C\epsilon \leq \frac 1 2$. Now we claim that
\be\label{eq:estY}
\left\|-\frac{1}{\pi \bar z} * Y^j\right\|_{L^{2,\infty}(\C)}\leq C \|Y^j\|_{L^1+\mathring{H}^{-1}(\C)}\quad.
\ee
 Recall that $Y^j \in L^1+\mathring{H}^{-1}(\C)$ and $\|Y^j\|_{L^1+\mathring{H}^{-1}(\C)}:=\inf\{\|Y^j_1\|_{L^1(\C)}+\|Y^j_2\|_{\mathring{H}^{-1}(\C)}: Y^j=Y^j_1+Y^j_2 \}$; since we can assume $Y^j\neq 0$ otherwise trivially $-\frac{1}{\pi \bar z} * Y^j=0$, we can find $Y^j_1 \in L^1(\C)$ and $Y^j_2\in \mathring{H}^{-1}(\C)$ such that 
\be \label{eq:estYY1Y2}
\|Y^j_1\|_{L^1(\C)}+\|Y^j_2\|_{\mathring{H}^{-1}(\C)}\leq \frac 3 2 \|Y^j\|_{L^1+\mathring{H}^{-1}(\C)}\quad.
\ee  
As before, by the Young inequality, we have 
\be\label{eq:estY1}
\left\|-\frac{1}{\pi \bar z} * Y^j_1\right\|_{L^{2,\infty}(\C)}\leq C \|Y^j_1\|_{L^1(\C)}\quad.
\ee
On the other hand called $\hat{Y} ^j_2$ the Fourier transform of $Y^j_2$,and observed that the Fourier transform of $-\frac{1}{\pi \bar z}$ is (up to a multiplicative constant) $\frac{1}{\xi}$  we have by the convolution theorem
$$\int_\C \left|-\frac{1}{\pi \bar z} * Y^j_2 \right|^2= C \int_\C\left|\hat{Y} ^j_2(\xi) \frac{1}{\xi}\right|^2\quad.$$
Moreover recalling that $\|h\|_{\mathring{H}^1}=\int_{\C}|\xi \hat{h}(\xi)|^2$ and that
$$\|{Y}^j_2\|_{\mathring{H}^{-1}(\C)}=\sup_{\|h\|_{\mathring{H}^1}\leq 1} \int_\C  \hat{Y} ^j_2(\xi) \bar{\hat{h}}(\xi)=\sup_{\|h\|_{\mathring{H}^1}\leq 1} \int_\C  \frac{\hat{Y} ^j_2(\xi)}{\xi} \; \bar{\hat{h}}(\xi)\xi=\int_{\C} \left|\frac{\hat{Y} ^j_2(\xi)}{\xi} \right|^2 $$
we get 
\be\label{eq:estY2}
\left\|-\frac{1}{\pi \bar z} * Y^j_2\right\|_{L^{2,\infty}(\C)}\leq \left\|-\frac{1}{\pi \bar z} * Y^j_2\right\|_{L^{2}(\C)} \leq C \|Y^j_2\|_{\mathring{H}^{-1}(\C)}.
\ee
Combining \eqref{eq:estYY1Y2}, \eqref{eq:estY1} and \eqref{eq:estY2} we get \eqref{eq:estY} which was our claim.
Now the estimates \eqref{eq:estU} and \eqref{eq:estY} implies that $T:L^{2,\infty}(\C)\to L^{2,\infty}(\C)$ is well defined and is a contraction; the existence of a unique $\tilde{\bU}$ satisfying \eqref{eq:Utilde} follows by the contraction mapping principle. Notice moreover we have the estimate
\be\label{eq:EstTildeU}
\|\tilde \bU\|_{L^{2,\infty}(\C)}\leq C \|\bY \|_{L^1+\mathring{H}^{-1}(\C)}.
\ee
Now let us consider $\Im (\tilde{\bU})$. From the equation satisfyied by $\tilde \bU$ we obtain
$$\frac 1 4 \triangle \tilde U^j=\p_ {\bar z} \p_ z \tilde{U}^j= \p_{\bar z} Y^j-\p_{\bar z} \left(\sum_{k=1}^m \gamma^j_k \tilde U^k \right)\quad,$$
whose imaginary part gives
\be\label{eq:triangleUtilde}
\frac 1 4 \triangle \Im(\tilde U^j)=\Im(\p_{\bar z} Y^j)-\Im\left(\p_{\bar z} \left(\sum_{k=1}^m \gamma^j_k \tilde U^k \right)\right)\quad. 
\ee
Since by assumption $\|\gamma^j_k\|_{L^\infty\cap W^{1,2}(\C)}\leq C$, then $\|\sum_{k=1}^{m}\gamma^j_k Y^k\|_{L^1+\mathring{H}^{-1}(\C)}\leq C \|\bY\|_{L^1+\mathring{H}^{-1}(\C)}$  and we have
$$\|\Im(\p_{\bar{z}}Y^j)\|_{L^1+\mathring{H}^{-1}(\C)}= \|\Im(D_{\bar z} Y^j)-\Im(\sum_{k=1}^{m}\gamma^j_k Y^k)\|_{L^1+\mathring{H}^{-1}(\C)}\leq \|\Im(D_{\bar z} Y^j)\|_{L^1+\mathring{H}^{-1}(\C)}+ C \|\bY\|_{L^1+\mathring{H}^{-1}(\C)}\quad.$$
Equation \eqref{eq:triangleUtilde} together with \eqref{eq:EstTildeU} and the last estimate gives

$$\|\triangle \Im (\tilde U^j)\|_{\mathring{H}^{-1}+L^1+W^{-1,(2,\infty)}(\C)}\leq  \|\Im(D_{\bar z} Y^j)\|_{L^1+\mathring{H}^{-1}(\C)}+ C \;\|\bY\|_{L^1+\mathring{H}^{-1}(\C)}$$ 

which implies, since $\|\Im (U)^j\|_{L^{2,\infty}(\C)}\leq C \|\bY\|_{L^1+\mathring{H}^{-1}(\C)}$, that
\be \label{eq:estNablaImU}
\|\nabla \Im (\tilde U^j)\|_{L^{2,\infty}(\C)}\leq C \left(\|\Im(D_{\bar z} Y^j)\|_{L^1+\mathring{H}^{-1}(\C)}+ \|\bY\|_{L^1+\mathring{H}^{-1}(\C)}\right). 
\ee
Now, since $\nabla \Im (\tilde U^j) \in L^{2,\infty}(\C)$, the function $ \Im (\tilde U^j)$ leaves a trace in $H^{\frac{1}{2},(2,\infty)}(\p D^2)$ and we can consider the homogeneous Dirichelet problem
\be\label{eq:DiricheletV}
\lf\{
\begin{array}{l}
\p _z V^j+ \sum_{k=1}^m \gamma^j_k V^k= 0 \quad \text{on } D^2\\[5mm]
\Im V^j= \Im \tilde U^j \quad \text{on } \p D^2\quad.
\end{array}
\rg.
\ee
We solve it again by contraction mapping principle; given $\bW \in W^{1,(2,\infty)}(D^2)$ with $\Im W^j= \Im \tilde U^j$ consider $\bV =:S(\bW)$ solving
\be\nonumber
\lf\{
\begin{array}{l}
\p _z V^j=- \sum_{k=1}^m \gamma^j_k W^k \quad \text{on } D^2\\[5mm]
\Im V^j= \Im \tilde U^j \quad \text{on } \p D^2\quad.
\end{array}
\rg.
\ee
Then the following estimate holds (see hand notes: bring right hand side to the left using convolution with $\frac{1}{\pi \bar z}$, so get homogeneous equation with different boundary data but still controlled in $H^{\frac{1}{2},(2,\infty)}$ both real and imaginary part using Hilbert tranform, the estimates then follow from the estimates for the laplace equation, using Calderon Zygmund theory for estimating the gradients)
$$\|\bV\|_{W^{1,(2,\infty)}(D^2)}\leq C\left(\epsilon \|\bW\|_{L^{2,\infty}(D^2)}+\|\Im \tilde{\bU}\|_{W^{1,(2,\infty)}(\C)} \right)$$
and 
$$\|S(\bW_1)-S(\bW_2) \|_{W^{1,(2,\infty)}(D^2)}\leq C\epsilon \|\bW_1-\bW_2 \|_{L^{2,\infty}(D^2)}.  $$
Therefore, for $\epsilon>0$ small, $S:W^{1,(2,\infty)}(D^2)\to W^{1,(2,\infty)}(D^2)$ is a contraction and there exists a unique solution of problem \eqref{eq:DiricheletV} satisfying the estimate
\be\label{estV}
\|\bV\|_{W^{1,(2,\infty)}(D^2)}\leq C \|\Im(\tilde{\bU})\|_{W^{1,(2,\infty)}(\C)}.
\ee
Now we conclude observing that $U^j:=\tilde U^j-V^j \in L^{2,\infty}(D^2)$ is the unique solution to the problem \eqref{eq:DzU=Y} and, combining \eqref{eq:EstTildeU}, \eqref{eq:estNablaImU} and \eqref{estV}, it satisfyes the estimates
$$\|\bU\|_{L^{2,\infty}(D^2)}+ \|\nabla \Im(\bU) \|_{L^{2,\infty}}\leq C  \left(\|\Im(D_{\bar z} Y^j)\|_{L^1+\mathring{H}^{-1}(\C)}+ \|\bY\|_{L^1+\mathring{H}^{-1}(\C)}\right) $$
as desired. 
\end{proof}

\begin{Lm}\label{lemmaA2bis}
For $j,l\in\{1,\ldots,m\}$ let $\gamma^j_l\in (C^0\cap W^{1,2})(\C)$ be such that $\supp \gamma^j_l \in B_2(0)$ and $\|\gamma^j_l\|_{L^\infty(\C)}\leq \epsilon$. For every $\bU \in (L^1_{loc})(\C)$ denote, in distributional sense, 
$$ (\D_z U)^j:= \p _z U^j+\sum_{k=1}^m \gamma^j_k U^k. $$
Let $\bY \in (L^1\cap L^{2,\infty})(\C)$ with $\Im(D_{\bar z} \bY) \in L^q(\C)$ for some $1<q<2$. Then
there exists a unique $\bU \in W^{1,(2,\infty)}(D^2)$ with $\Im(\bU) \in W^{2,q}(D^2)$ satisfying

\be\label{eq:DzU=Ybis}
\lf\{
\begin{array}{l}
\D _z \bU= \bY \quad 	\text{in } {\cal D}'(D^2)\\[5mm]
\Im \bU=0 \quad \text{on } \p D^2\quad.
\end{array}
\rg.
\ee
Moreover the following estimate holds:
$$\|\bU\|_{L^{2,\infty}(D^2)}+ \|\nabla \bU \|_{L^{2,\infty}(D^2)}+ \|\nabla^2 \Im(U) \|_{L^q(D^2)} \leq C \left(\|\bY\|_{L^1\cap L^{2,\infty}(\C)}+\|\Im(\D _{\bar z} \bY)\|_{L^q(\C)} \right)\quad. $$
\hfill $\Box$
\end{Lm}

\begin{proof}
As in the proof of Lemma \ref{lemmaA2} we first solve the equation $\D _z \tilde{\bU}= \bY$ in $\C$ proving  existence and uniqueness of solutions to the fixed point problem  in $W^{1,(2,\infty)}(\C)$
\be\label{eq:Utilde1}
\tilde U^j=-\frac{1}{\pi \bar z} * \left(Y^j-\sum_{k=1}^m \gamma^j_k \tilde{U}^k\right)=:T(\tilde \bU)\quad.
\ee
Analogously to the proof of Lemma \ref{lemmaA2}, for $\epsilon>0$ small but depending just on universal constants, the $L^{2,\infty}(\C)$ norm of $T(\tilde \bU)$ can be bounded as
\be\label{eq:TuL2inf}
\|T(\tilde \bU) \|_{L^{2,\infty}(\C)}\leq C \|\bY\|_{L^1(\C)}\quad,
\ee 
and for $\tilde{\bU}_1,\tilde{\bU}_2 \in L^{2,\infty}(\C)$ it holds
\be\label{eq:ContrL}
\|T(\tilde{\bU}_1)-T(\tilde{\bU}_2)\|_{L^{2,\infty}(\C)}\leq C\epsilon \|\tilde{\bU}_1-\tilde{\bU}_2\|_{L^{2,\infty}(\C)}.
\ee
$L^{2,\infty}$-\emph{Gradient estimate}: we have
\be\label{eq:nablaTildeU}
\|\nabla T(\tilde U^j)\|_{L^{2,\infty}(\C)}=\left\|\left( \nabla \frac{1}{\pi \bar z}\right)* \left(Y^j-\sum_{k=1}^m \gamma^j_k \tilde U^k \right)\right\|_{L^{2,\infty}(\C)}.
\ee
Observe that the Fourier trasform of the convolution kernel $\nabla \frac{1}{\bar z}=\nabla(\p_{\bar z} \log |z|)$ satisfies the assumptions of Theorem 3 pag. 96 in \cite{Stein}, therefore
$$\left\|\left( \nabla \frac{1}{\pi \bar z}\right)* \left(Y^j-\sum_{k=1}^m \gamma^j_k \tilde U^k \right)\right\|_{L^s(\C)}\leq C_s \left\|\left(Y^j-\sum_{k=1}^m \gamma^j_k \tilde U^k \right) \right\|_{L^s(\C)} \quad \forall 1<s<\infty $$
and by interpolation (see for instance Theorem 3.15 in \cite{Stein-Weiss} of Theorem 3.3.3 in \cite{Hel})
\be \label{eq:EstCZ}
\begin{array}{l}
\ds\left\|\left( \nabla \frac{1}{\pi \bar z}\right)* \left(Y^j-\sum_{k=1}^m \gamma^j_k \tilde U^k \right)\right\|_{L^{2,\infty}(\C)}\leq C \left\|\left(Y^j-\sum_{k=1}^m \gamma^j_k \tilde U^k \right) \right\|_{L^{2,\infty}(\C)}
\\[5mm]
\ds\quad\quad \leq C \|Y\|_{L^{2,\infty}(\C)}+ C\epsilon \|\tilde \bU \|_{L^{2,\infty}(\C)}\quad.
\end{array}
\ee
Combining \eqref{eq:EstCZ}, \eqref{eq:nablaTildeU} and \eqref{eq:TuL2inf} we get, for small $\epsilon>0$,
\be\label{eq:estTW12}
\|T(\tilde \bU) \|_{W^{1,(2,\infty)}(\C)}\leq C \left( \|\bY\|_{L^1(\C)}+ \|\bY\|_{L^{2,\infty}(\C)} \right).
\ee
So $T:W^{1,(2,\infty)}(\C)\to W^{1,(2,\infty)}(\C)$  is a well defined linear operator and the same arguments imply that $T$ is a contraction. Therefore there exists a unique $\tilde U^j \in W^{1,(2,\infty)}(\C)$ satisfying \eqref{eq:Utilde1} and
\be\label{eq:estTW12}
\|\tilde \bU \|_{W^{1,(2,\infty)}(\C)}\leq C \left( \|\bY\|_{L^1(\C)}+ \|\bY\|_{L^{2,\infty}(\C)} \right).
\ee
Noticing that $\Im(\tilde U^j)$ satisfies also
$$
\begin{array}{l}
\triangle(\Im(\tilde U^j))=4 \Im(\p_{\bar z}Y^j)-4\Im\left[\p_{\bar z}\left(\sum_{k=1}^m \gamma^j_k \tilde U^k\right)\right]\\[5mm]
\quad\quad\ds=4 \Im(D_{\bar z}Y^j)-4\Im\left(\sum_{k=1}^m \gamma^j_k \tilde Y^k\right)-4\Im\left[\p_{\bar z}\left(\sum_{k=1}^m \gamma^j_k \tilde U^k\right)\right]\quad,
\end{array}
 $$
estimate \eqref{eq:estTW12}, the assumptions on $\gamma^j_k$, H\"older inequality and standand elliptic estimates imply
\be \label{eq:estHessImU}
\|\nabla^2 \Im(\tilde{U}^j)\|_{L^q(D^2)}\leq C\left(\|\Im(D_{\bar z}Y^j) \|_{L^q(D^2)}+ \|\bY\|_{L^1(\C)}+ \|\bY\|_{L^{2,\infty}(\C)}\right). 
\ee
Now exactly as in the previous lemma it is possible to solve the corresponding homogeneous problem on $D^2$
\be\label{eq:Homogen}
\lf\{
\begin{array}{l}
\p _z V^j=- \sum_{k=1}^m \gamma^j_k V^k \quad \text{on } D^2\\[5mm]
\Im V^j= \Im \tilde U^j \quad \text{on } \p D^2\quad.
\end{array}
\rg.
\ee
and the solution $V^j$ satisfies the estimates
\be\label{estVbis}
\|\bV\|_{W^{1,(2,\infty)}(D^2)}\leq C \|\Im(\tilde{\bU})\|_{W^{1,(2,\infty)}(\C)}.
\ee
Moreover the imaginary part $\Im(V^j)$ solves the following problem
\be\label{eq:HomogenImV}
\lf\{
\begin{array}{l}
\triangle \Im(V^j) + 4 \Im\left[\p_{\bar{z}} \left(\sum_{k=1}^m \gamma^j_k V^k  \right)\right] =0\quad \text{on } D^2\\[5mm]
\Im V^j= \Im \tilde U^j \quad \text{on } \p D^2\quad;
\end{array}
\rg.
\ee
then, estimate \eqref{eq:estHessImU}  and elliptic regularity imply 
\be \label{eq:estHessImV}
\|\nabla^2 \Im(V^j)\|_{L^q(D^2)}\leq C\left(\|\Im(D_{\bar z}Y^j) \|_{L^q(D^2)}+ \|\bY\|_{L^1(\C)}+ \|\bY\|_{L^{2,\infty}(\C)}\right). 
\ee 
Now, as in the previous lemma, the function $U^j=\tilde U^j-V^j$ is a solution to the original problem \eqref{eq:DzU=Ybis}; moreover collecting \eqref{eq:estTW12}, \eqref{eq:estHessImU},\eqref{estVbis} and \eqref{eq:estHessImV} we obtain the desired estimate
$$
\|\bU\|_{W^{1,(2,\infty)}(D^2)}+ \|\nabla^2 \Im(\bU)\|_{L^q(D^2)} \leq C \left(  \|\Im(D_{\bar z}Y^j) \|_{L^q(D^2)}+ \|\bY\|_{L^1(\C)}+ \|\bY\|_{L^{2,\infty}(\C)}\right). 
 $$
\end{proof}

\subsection{Differentiability of the Willmore functional in ${\mathcal F}_{\Sp^2}$, identification of the first differential and lower semicontinuity under $W_{loc}^{2,2}$ weak convergence}

Let us start with two computational lemmas whose utility will be clear later in the subsection.

\begin{Lm}\label{lem:PreInt}
Let $\bP$ be a smooth immersion of the disc $D^2$ into the Riemannian manifold $(M^n,h)$. Let $\bX\otimes \bv \in \Gamma_{D^2}(T_{\bP}M \otimes TD^2)$ and $\bw \in \Gamma_{D^2}(T_{\bP}M)$, recall the notation introduced in \eqref{def:Dg}, \eqref{def:*g}, \eqref{def:D*} and \eqref{def:ScalProd}. Then
\be\label{eq:PreInt}
<D_g^{*_g}(\bX \otimes \bv), \bw>=<\bX \otimes \bv,D_g \bw>-div_g \bu+<\bX,\bw> div_g \bv,  
\ee
where $\bu\in \Gamma_{D^2}(TD^2)$ is the vector field defined below. Let $\bbf_1,\bbf_2$ be a positve orthonormal frame of $TD^2$, write $\bv=v^1 \bbf_1+v^2 \bbf_2$, then define
$$\bu:=<v_1 \bX, \bw>_h \bbf_1+ <v_2 \bX, \bw>_h \bbf_2.$$
Notice that $\bu$ is independent of the choice of the frame $f_i$, i.e. it is a well defined vector field on $D^2$.
\hfill $\Box$
\end{Lm}

\begin{proof}
Call $\bbe_i:=\bP_*(\bbf_i)$ the positive orthonormal frame of $\bP_* (TD^2)$ associated to $\bbf_1,\bbf_2$; a straightforward computation using just the definitions  \eqref{def:Dg}, \eqref{def:*g} and \eqref{def:D*} gives
\be \label{eq:Dg*X}
<D_g^{*_g}[\bX \otimes ( v^1 \bbf_1+ v^2 \bbf_2)], \bw> = -<v_1 D_{\bbe_1} \bX+ v^2 D_{\bbe_2} \bX, \bw>.
\ee
Writing the right hand side as $-< D_{\bbe_1} (v^1 \bX)+  D_{\bbe_2} (v^2 \bX), \bw> + <\bX,\bw> \left(\bbe_1[v^1]+\bbe_2[v^2]  \right) $, where $\bbe_i[v_i]$ denotes the derivative of the function $v^i$ with respect to $\bbe_i$,  we can express \eqref{eq:Dg*X} as
\begin{eqnarray} 
<D_g^{*_g}[\bX \otimes ( v^1 \bbf_1+ v^2 \bbf_2)], \bw> &=& < v^1 \bX, D_{\bbe_1}\bw> + < v^2 \bX, D_{\bbe_2}\bw>\nonumber \\
&&- \bbe_1[< v^1 \bX, \bw>]-  \bbe_2[< v^2 \bX, \bw>] \nonumber \\
&&+ <\bX,\bw> \left(\bbe_1[v^1]+\bbe_2[v^2]  \right). \label{eq:Dg*X1}
\end{eqnarray}
Notice that the first line of the right hand side is exactly $<\bX\otimes \bv, D_g \bw >.  $
Observe that, through the parametrization $\bP$, we can identify $TD^2$ and $\bP_*(TD^2)$, moreover noticing that for fixed $i=1,2$  we have $<D_{\bbe_i} \bbe_i, \bbe_i>= \frac{1}{2}\bbe_i [<\bbe_i, \bbe_i >]=0 $, after some easy computations we get
that 
\be \label{eq:divbv}
\bbe_1[v^1]+\bbe_2[v^2]=\bbf_1[v^1]+\bbf_2[v^2]= div_g (\bv)+ <\bv, D_{\bbf_1} \bbf_1+D_{\bbf_2} \bbf_2>,
\ee
where, by definition, $div_g (\bv):=\sum_{i=1,2} <D_{\bbf_i} \bv, \bbf_i>$ and $D$ is intented as the covariant derivative on $TD^2$ endowed with the metric $g:=\bP^*h$ (notice that the covariant derivative in $M$ along $\bP(D^2)$ projected on $\bP_*(TD^2)$ correspond to the covariant derivative on $(D^2,g)$ via the identification given by the immersion $\bP$). 
\\Recall we defined $\bu:=<v_1 \bX, \bw>_h \bbf_1+ <v_2 \bX, \bw>_h \bbf_2=u^1 \bbf_1+u^2 \bbf_2 \in TD^2$; an easy computation gives
\be\label{eq:divbu}
\bbf_1[< v^1 \bX, \bw>]+  \bbf_2[< v^2 \bX, \bw>]= div_g \bu + <\bX, \bw> <\bv, D_{\bbf_1} \bbf_1+D_{\bbf_2} \bbf_2>.
\ee
Now combining \eqref{eq:Dg*X1}, \eqref{eq:divbv} and \eqref{eq:divbu} we get the thesis.
\end{proof}

\begin{Lm}{\bf [Integration by parts in Willmore equation]} \label{Lm:IntPart}
Let $\bP$ be a smooth immersion of the disc $D^2$ into the Riemannian manifold $(M^n,h)$ and let $\bw \in \Gamma_{D^2}(T_{\bP}M)$ smooth with compact support in $D^2$. Then
\be
\int_{D^2} \left[<\frac{1}{2} D^{*_g}_g \left[D_g \bH -3 \pi_{\bn} (D_g \bH) + \star_h\left((*_g D_g \bn) \wedge_M \bH \right)  \right] -\ti{R}(\bH)+R^\perp_{\vec{\Phi}}(T\vec{\Phi}), \; \bw> \right] dvol_g=
\ee
$$
\begin{array}{l}
\ds\int_{D^2} \left(< \bH, \frac{1}{2} D^{*_g}_g \left[D_g \bw -3 \pi_{\bn} (D_g \bw)\right]> +  <\star_h\left((*_g D_g \bn) \wedge_M \bH \right), D_g \bw>\right)\ dvol_g\\[5mm]
+\ds\int_{D^2}\ < -\ti{R}(\bH)+R^\perp_{\vec{\Phi}}(T\vec{\Phi}), \bw> \ dvol_g\quad.
\end{array}
 $$
 \hfill $\Box$
\end{Lm}

\begin{proof}
Let us start considering $\int_{D^2}< D^{*_g}_g [D_g \bH],\bw> dvol_g$. Fix a point $x_0\in D^2$, take normal coordinates $x^i$ centred in $x_0$ with respect to the metric $g=\bP^* h$, and call $f_i:=\frac{\p}{\p x^i}$ the coordinate frame; observe it is orthonormal at $x_0$ and $D f_i=0$ at $x_0$. By Lemma \ref{lem:PreInt}, we have
$$< D^{*_g}_g [D_g \bH], \bw>= < D_g \bH, D_g \bw>-div_g \bu+<D_{\bbf_1} \bH,\bw> div_g \bbf_1+ <D_{\bbf_2} \bH,\bw> div_g \bbf_2$$
for some vector field $\bu$ compactly supported in $D^2$. Observe that, at $x_0$, the condition $D\bbf_i=0$ implies 
$$div_g \bbf_i=<D_{\bbf_1} \bbf_i, \bbf_1>+<D_{\bbf_2} \bbf_i, \bbf_2>=0.$$
Therefore taking $\bbf_i$ to coincide at $x_0$ with the frame associated to normal coordinates centred at $x_0$ we obtain 
\be\label{eq:IntPartsP1}
< D^{*_g}_g [D_g \bH], \bw>= < D_g \bH, D_g \bw>-div_g \bu;
\ee
since all the terms are defined intrinsecally, the identity is true intrinsecally at every point $x_0 \in D^2$. Now integrate \eqref{eq:IntPartsP1} on $D^2$ and, observing that $\bu$ is compactly supported, use the divergence theorem to infer
\be\nonumber 
\int_{D^2}< D^{*_g}_g [D_g \bH], \bw> dvol_g=\int_{D^2} < D_g \bH, D_g \bw> dvol_g.
\ee
Repenting the same argument we have also $\int_{D^2}< D^{*_g}_g [D_g \bw], \bH> dvol_g=\int_{D^2} < D_g \bw, D_g \bH> dvol_g$, so
\be\label{eq:IntParts1}
\int_{D^2}< D^{*_g}_g [D_g \bH], \bw> dvol_g=\int_{D^2} < \bH,D^{*_g}_g[ D_g \bw]> dvol_g.
\ee
Using analogous arguments one checks that also
\begin{eqnarray}
\int_{D^2}< D^{*_g}_g [\pi_{\bn}(D_g \bH)], \bw> dvol_g&=&\int_{D^2} < \pi_{\bn}(D_g \bH), D_g \bw> dvol_g=\int_{D^2} < D_g \bH,\pi_{\bn}( D_g \bw)> dvol_g \nonumber \\
&=&\int_{D^2} <\bH, D^{*_g}_g[\pi_{\bn}( D_g \bw)]> dvol_g. \label{eq:IntParts2}
\end{eqnarray}
Finally, along the same lines, one has
\be\label{eq:IntParts3}
\int_{D^2} <D^{*_g}_g\left[\star_h\left((*_g D_g \bn) \wedge_M \bH \right)\right], \bw> dvol_g= \int_{D^2} <\star_h\left((*_g D_g \bn) \wedge_M \bH \right), D_g \bw> dvol_g.
\ee
The thesis follows collecting \eqref{eq:IntParts1}, \eqref{eq:IntParts2} and \eqref{eq:IntParts3}.
\end{proof}

\begin{Lm}{\bf [Differentiability of $W$ and identification of $dW$]} \label{lem:dW}
Let $\bP \in \mathcal{F}_{\Sp^2}$ be a weak branched immersion of $\Sp^2$ into the $m$-dimensional Riemannian manifold $(M^m,h)$ with branched points $\{b^1,\dots,b^N\}$ and let $W(\bP):=\int_{\Sp^2}|H_{\bP}|^2 dvol_{g_{\bP}}$ be the Willmore functional. Then $W$ is Frech\'et differentiable with respect to   variations $\bw \in W^{1,\infty}\cap W^{2,2} (D^2,T_{\bP} M)$  with compact support in $\Sp^2\setminus \{b^1,\ldots,b^N\}$ in the sense that
\be\label{eq:WFrechet}
W( Exp_{\bP}[t\bw])=W (\bP)+t \;d_{\bP}W[\bw]+R^{\bP}_{\bw}[t],
\ee
where $Exp_{\bP}[t\bw](x_0)$  denotes the exponential map in $M$ centered in $\bP(x_0) \in M$ applied to the tangent vector $t\bw \in T_{\bP(x_0)}M$ and where the remainder $R^{\bP}_{\bw}[t]$ satisfyes
$$\sup \; \left\{ \left|R^{\bP}_{\bw}[t]\right| \;: \; \|\bw\|_{W^{2,2}}+\|\bw\|_{W^{1,\infty}}\leq 1 \text{ and } \supp \bw \subset K \subset \subset \Sp^2 \setminus \{b^1,\ldots,b^{N_{\bP}}\} \right\} \leq C_{\bP,K} t^2. $$
Moreover the differential $d_{\bP}W$ coincides with the Willmore equation in conservative form: for every  $\bw \in W^{1,\infty}\cap W^{2,2} (D^2,T_{\bP} M)$   with compact support in  $\Sp^2\setminus \{b^1,\ldots,b^N\}$,
\begin{eqnarray}
d_{\bP}W[\bw]&=&\int_{\Sp^2}\Big(< \bH, \frac{1}{2} D^{*_g}_g \left[D_g \bw -3 \pi_{\bn} (D_g \bw)\right]> +  <\star_h\left((*_g D_g \bn) \wedge_M \bH \right), D_g \bw>\nonumber \\
&&\quad \quad +< -\ti{R}(\bH)+R^\perp_{\vec{\Phi}}(T\vec{\Phi}), \bw> \Big) dvol_g \quad . \label{eq:dWCons}
\end{eqnarray}
Also the functional $F(\bP)=\int_{\Sp^2}|{\mathbb I}_{\bP}|^2 dvol_g$ is Frech\'et differentiable  with respect to  variations $\bw \in W^{1,\infty}\cap W^{2,2} (\Sp^2,T_{\bP} M)$  with compact support in $\Sp^2\setminus \{b^1,\ldots,b^N\}$ in the sense above, and 
\begin{eqnarray}
d_{\bP}F[\bw]&=&\int_{\Sp^2}\Big(< \bH,  D^{*_g}_g \left[D_g \bw -3 \pi_{\bn} (D_g \bw)\right]> + 2 <\star_h\left((*_g D_g \bn) \wedge_M \bH \right), D_g \bw>\nonumber \\
&& +2< -\ti{R}(\bH)+R^\perp_{\vec{\Phi}}(T\vec{\Phi})-(D\,R)(T\bP)-2{\frak R}_{\bP}(T\bP)-2\bar{K}(T\bP)\bH, \bw> \Big) dvol_g \quad . \label{eq:dF}
\end{eqnarray}
Finally also the area functional $A(\bP)=\Area_{g_{\bP}}(\Sp^2)$ is Frech\'et differentiable  with respect to   variations $\bw \in W^{1,\infty}\cap W^{2,2} (D^2,T_{\bP} M)$  with compact support in $\Sp^2\setminus \{b^1,\ldots,b^N\}$ in the sense above, and 
\be\label{eq:dA}
d_{\bP}A[\bw]=-\int_{\Sp^2} <2\bH,\bw> dvol_g.
\ee
\hfill $\Box$
\end{Lm}

\begin{proof}
Let $\bP \in \mathcal{F}_{\Sp^2}$ and observe that the mean curvature $\bH_{\bP} \in T_{\bP}M^n$ is a function of $(\nabla^2 \bP, \nabla \bP, \bP)$, where $\nabla^2 \bP$ and $\nabla \bP$ are respectively the hessian and the gradient of $\bP$:
\begin{equation}
\bH_{\bP}= \tilde{\bH} (\nabla^2 \bP, \nabla \bP, \bP),
\end{equation}
where 
\be
\tilde{\bH}:((T \Sp^2)^2 \otimes T_{\bP}M, T \Sp^2 \otimes T_{\bP}M, M ) \to T_{\bP}M, \quad (\vec{\xi},\vec{q},\vec{z}) \mapsto  \tilde{\bH}(\vec{\xi},\vec{q},\vec{z}).
\ee
Observe that $ \tilde{\bH}$ is smooth on the open set given by $|\vec{q} \wedge \vec{q}| >0$; moreover, for every $\vec{q}_0$ and $\vec{z}_0$, the map $\vec{\xi} \mapsto  \tilde{\bH}(\vec{\xi},\vec{q}_0,\vec{z}_0)$ is linear.
Recall also that the area form $dvol_{g_{\bP}}$ associated to the pullback metric $g_{\bP}:=\bP^*h$ is of the form
$$ dvol_{g_{\bP}} = f(\nabla \bP, \bP) \; dvol_{g_0} $$
where $dvol_{g_0}$ is the area form associated to the reference metric $g_0$ on $(\Sp^2,c_0)$, and  
$$
f:(T \Sp^2 \otimes T_{\bP}M, M ) \to \R, \quad (\vec{q},\vec{z}) \mapsto  f(\vec{q},\vec{z})
$$
is smooth on the open subset  $|\vec{q} \wedge \vec{q}| >0$. Therefore the integrand of the Willmore functional can be written as
\be\label{eq:H2F}
|\bH_{\bP}|^2 dvol_{g_{\bP}}= |\bF|^2(\nabla^2 \bP, \nabla \bP, \bP) dvol_{g_0},
\ee
where $\bF(\vec{\xi},\vec{q},\vec{z}):=\tilde{\bH}(\vec{\xi},\vec{q},\vec{z}) \sqrt{f}(\vec{q},\vec{z})$; clearly $\bF$ is smooth on the subset $|\vec{q} \wedge \vec{q}| >0$ and, for every $\vec{q}_0$ and $\vec{z}_0$, the map $\vec{\xi} \mapsto  {\bF}(\vec{\xi},\vec{q}_0,\vec{z}_0)$ is linear.

Let $\bw \in W^{2,2}\cap W^{1,\infty}(\Sp^2,T_{\bP}M)$  be an infinitesimal perturbation  supported in $\Sp^2 \setminus \{b^1,\ldots,b^{N_{\bP}}\}$, where $\{b^1,\ldots,b^{N_{\bP}}\}$ are the branch points of $\bP$;  consider, for small $t>0$, the perturbed weak branched immersion $Exp_{\bP}[t\bw]$, where $Exp_{\bP}[t\bw](x_0)$  denotes the exponential map in $M$ centered in $\bP(x_0) \in M$ applied to the tangent vector $t\bw \in T_{\bP(x_0)}M$. Observe that, by definition,
$$\int_{\Sp^2} |\bH_{Exp_{\bP}[t\bw]}|^2 dvol_{g_{Exp_{\bP}[t\bw]}}=\int_{\Sp^2} |\bF|^2 (\nabla^2 (Exp_{\bP}[t\bw]), \nabla (Exp_{\bP}[t\bw]), Exp_{\bP}[t\bw]) \; dvol_{g_0}.$$
Recall that, using the construction of conformal coordinates with estimates by Chern-Hel\'ein-Rivi\'ere, we can assume that on every compact subset $K\subset \subset \Sp^2 \setminus \{b^1,\ldots,b^{N_{\bP}}\}$, the immersion $\bP$ is  conformal  with  $\|(\log|\nabla \bP|)\|_{L^\infty(K)} \leq C_K$ for some constant $C_K$ depending on $K$. By conformality, it follows that on every compact subset $K\subset \subset \Sp^2\setminus\{b^1,\ldots,b^{N_{\bP}}\}$ there exist a positive constant $c_K$ such that $|d\bP\wedge d \bP |\geq c_K >0$. Since $\bw$ is supported away the branch points it follows that, for $t$ small enough,  $(\nabla^2 (Exp_{\bP}[t\bw]), \nabla (Exp_{\bP}[t\bw]), Exp_{\bP}[t\bw])|_{\supp (\bw)}$ is in the domain of smoothness of $\bF$. By a Taylor expansion in $t$ we get
\begin{eqnarray}
\int_{\Sp^2} |\bH_{Exp_{\bP}[t\bw]}|^2 dvol_{g_{Exp_{\bP}[t\bw]}}&=& \int_{\Sp^2} |\bH_{\bP}|^2 dvol_{g_{\bP}} \nonumber \\
&& +2t \int_{\Sp^2} \bF(\nabla^2 \bP, \nabla \bP, \bP) \cdot \partial_{\xi^k_{ij}} \bF(\nabla \bP, \bP) \partial^2_{x^i x^j} w^k dvol_{g_0} \nonumber \\
&& +2t \int_{\Sp^2} \bF(\nabla^2 \bP, \nabla \bP, \bP) \cdot \partial_{q^k_{i}} \bF(\nabla^2 \bP,\nabla \bP, \bP) \partial_{x^i} w^k dvol_{g_0} \nonumber \\
&&+2t \int_{\Sp^2} \bF(\nabla^2 \bP, \nabla \bP, \bP) \cdot \partial_{z^k} \bF(\nabla^2 \bP,\nabla \bP, \bP) w^k dvol_{g_0} \nonumber \\
&&+ t^2 \int_{\Sp^2} \p^2_{x^i x^j} w^k \p^2_{x^r x^s} w^l P^{kl}_{ijrs} (\nabla \bP, \bP, \nabla \bw, \bw)\nonumber \\
&&+t^2 \int_{\Sp^2} \p^2_{x^i x^j} w^k Q^{k}_{ij} (\nabla^2 \bP,\nabla \bP, \bP, \nabla \bw, \bw) \nonumber\\
&&+t^2 \int_{\Sp^2} S(\nabla^2 \bP,\nabla \bP, \bP, \nabla \bw, \bw) \quad , \label{eq:TaylorW}
\end{eqnarray} 
where, in the second line $\partial_{\xi^k_{ij}} \bF$ depends just on $(\nabla \bP, \bP)$ since $\bF$ is linear in $\vec{\xi}$, in the $5^{th}$ line the function $\vec{P}$ is smooth in its arguments with $\vec{P}(\nabla \bP, \bP, 0, 0)=0$, in the  $6^{th}$ line the function $\vec{Q}$ is smooth in its arguments and linear in $\nabla^2 \bP$ with $\vec{Q}(\nabla^2 \bP,\nabla \bP, \bP, 0, 0)=0$ and in the $7^{th}$ line the function $S$ is smooth in its arguments and quadratic in $\nabla^2 \bP$ with $S(\nabla^2 \bP,\nabla \bP, \bP, 0, 0)=0$.  Therefore, called $R^{\bP}_{\bw}[t]$ the sum of the last three lines  of \eqref{eq:TaylorW}, we have that
$$\sup \; \left\{ \left|R^{\bP}_{\bw}[t]\right| \;: \; \|\bw\|_{W^{2,2}}+\|\bw\|_{W^{1,\infty}}\leq 1 \text{ and } \supp \bw \subset K \subset \subset \Sp^2 \setminus \{b^1,\ldots,b^{N_{\bP}}\} \right\} \leq C_{\bP,K} t^2. $$
It follows that $\int |H_{\bP}|^2 dvol_{g_{\bP}}$ is Frech\'et-differentiable with respect to ${W^{2,2}}\cap W^{1,\infty}$  variations compactly supported away from the branch points, and the first variation $dW_{\bP_\e}[t\bw]$ is given by the sum of lines $2,3,4$ of  \eqref{eq:TaylorW}.

Now we identify the first order term in the expansion  of $\int |H_{\bP}|^2 dvol_{g_{\bP}}$   with the conservative Willmore equation we derived before in the paper. Observe that it is not completely trivial since the conservative Willmore equation has been proved for \emph{smooth}  immersions, while now $\bP$ is a \emph{weak} branched  immersion. First of all, recall that if $\vec{\Psi}$ is a smooth  immersion of the disc $D^2$ taking values in a coordinate chart of $M$, then for a smooth variation $\bw \in C^{\infty}_0 (D^2, \R^m)$ with compact support in $D^2$  we have that
\be\label{expWsmooth}
\int_{D^2} |\bH_{{\vec{\Psi}}+t\bw}|^2 dvol_{g_{{\vec{\Psi}}+t\bw}}= \int_{D^2} |\bH_{\vec{\Psi}}|^2 dvol_{g_{\vec{\Psi}}}+t dW_{\vec{\Psi}}[\bw]+\tilde{R}^{\vec{\Psi}}_{\bw}[t] ;
\ee
where the remainder $\tilde{R}^{\vec{\Psi}}_{\bw}[t]$ has the same form as the sum of the last three line of \eqref{eq:TaylorW}, and where the differential $dW_{\vec{\Psi}}[\bw]$, after the integration by parts procedure carried in Lemma \ref{Lm:IntPart}, can be written as
\be
\label{eq:DiffSmooth}
\begin{array}{l}
\ds dW_{\vec{\Psi}}[\bw]=\int_{D^2} \left(< {\bH_{\vec{\Psi}}}, \frac{1}{2} D^{*_{g_{\vec{\Psi}}}}_{g_{\vec{\Psi}}} \left[D_{g_{\vec{\Psi}}} \bw -3 \pi_{\bn_{\vec{\Psi}}} (D_{g_{\vec{\Psi}}} \bw)\right]>\right) dvol_{g_{\vec{\Psi}}} \nonumber \\[5mm]
\ds\quad + \int_{D^2}  \left( <\star_h\left((*_{g_{\vec{\Psi}}} D_{g_{\vec{\Psi}}} \bn) \wedge_M {\bH_{\vec{\Psi}}} \right), D_{g_{\vec{\Psi}}} \bw>+< -\ti{R}({\bH_{\vec{\Psi}}})+R^\perp_{\vec{\Psi}}(T{\vec{\Psi}}), \bw> \right) dvol_{g_{\vec{\Psi}}}. 
 \end{array}
 \ee

Now let us start considering the case of $\bP \in W^{1,\infty}\cap W^{2,2} (D^2)$ a weak conformal immersion with finite total curvature without branch points taking values in a coordinate chart of $M$, and let $\bw \in C^\infty_0 (D^2, \R^m)$ be  a smooth variation with compact support in $D^2$.
\\ Let $\varphi$ be a non negative compactly supported function of $C^\infty_0({\R})$ such that $\varphi$ is identically equal to 1 in a neighborhood of $0$ and
\[
2\pi \int_{\R}\varphi(t)\,t\ dt=1.
\]
Call $\varphi_\ep(t):=\ep^{-2}\ \varphi(t/\ep)$. Denote for $\ep<1/4$ and for any $x\in D^2_{1/2}$, 
\[
\vec{\Phi}_\ep(x):=\varphi_\ep(|x|)\star\vec{\Phi}:=\int_{D^2}\varphi_{\ep}(|x-y|)\ \vec{\Phi}(y)\ dy\quad.
\]
By Lemma \ref{approx-lemma} there exists $0<\ep_{\vec{\Phi}}<1/4$ such that for any $\ep<\ep_{\vec{\Phi}}$ the  map $\vec{\Phi}_\ep$ realizes a smooth immersion from $D^2_{1/2}$ into the coordinate chart, moreover we have (notice that in order to keep the notation not too heavy, in the following  we replaced $D^2_{1/2}$ by $D^2$)
\begin{eqnarray}
\bP_\e &\to& \bP \quad \text{strong in } W^{2,2}(D^2) \label{PhietoPhi} \\
\bH_\e  &\to &\bH \quad \text{strong in } L^2(D^2) \label{HetoH} \\
\bn_\e &\to& \bn \quad \text{strong in } W^{1,2}(D^2) \label{neton} 
\end{eqnarray}
and 
\begin{eqnarray}
\sup_{0<\e\leq \e_0} \|\bP_\e\|_{W^{1,\infty}} (D^2) &\leq& C < \infty \label{PhiBounded} \\
\inf_{x \in D^2} \inf_{0<\e\leq \e_0} |d\bP_\e \wedge d\bP_\e| &\geq& \frac{1}{C}>0. \label{eq:dPhibb}
\end{eqnarray}
Since $\bP_\e$ is smooth, the Willmore functional computed on $\bP_\e+t \bw$ expands as in \eqref{expWsmooth}; observe that, thanks to \eqref{PhietoPhi}, \eqref{PhiBounded} and \eqref{eq:dPhibb}, the remainder $\tilde{R}^{\vec{\bP_\e}}_{\bw}[t]$ satisfyes
\be\label{eq:EstRemTilde}
\sup_{0<\e\leq \e_0}\; \sup_{\|\bw\|_{W^{1,\infty}\cap W^{2,2}(D^2)}\leq 1} \left|\tilde{R} ^{\bP_\e}_{\bw}[t]\right| \leq C_{\bP} t^2.
\ee
Observe moreover that, by \eqref{PhietoPhi}, \eqref{HetoH} and \eqref{PhiBounded}, we have that $|\bH_\e|^2 dvol_{g_\e}$ is dominated in $L^1(D^2)$ for $\e\leq \e_0$, and converges almost everywhere on $D^2$ to $|\bH|^2 dvol_g$; therefore, by Dominated Convergence Theorem, we have 
\be\label{eq:convWPhie}
\int_{D^2} |\bH_\e|^2 dvol_{g_\e} \to \int_{D^2} |\bH|^2 dvol_g .
\ee 
Moreover, using \eqref{PhietoPhi} and \eqref{eq:dPhibb} we have that, for $\|\bw\|_{W^{1,\infty}\cap W^{2,2}(D^2)}\leq 1$ and $t$ small enough, $\bP_\e+t \bw \to \bP+t \bw$ strongly in  $W^{2,2}(D^2)$ and $\bH_{\bP_\e+t\bw}  \to \bH_{\bP+t \bw}$ strongly in $L^2(D^2)$; of course it  still holds  $\sup_{0<\e\leq \e_0} \|\bP_\e+t \bw\|_{W^{1,\infty}} (D^2) \leq C < \infty$. Therefore, with the same argument above, we get
\be\label{eq:convWPhiew}
\int_{D^2} |\bH_{\bP_{\e}+t\bw}|^2 dvol_{g_{\bP_{\e}+t\bw}} \to \int_{D^2} |\bH_{\bP+t\bw}|^2 dvol_{g_{\bP+t\bw}} .
\ee
Combining \eqref{eq:TaylorW}, \eqref{eq:convWPhie} and \eqref{eq:convWPhiew} gives
$$
dW_{\bP}[\bw]=\lim_{t\to 0} \frac{W(\bP+t\bw)-W(\bP)}{t}=\lim_{t\to 0} \lim_{\e\to 0} \frac{W(\bP_\e+t\bw)-W(\bP_\e)}{t}=\lim_{t\to 0} \lim_{\e\to 0} \left( dW_{\bP_\e}[\bw]+ \frac{\tilde{R} ^{\bP_\e}_{\bw}[t]}{t}\right);
$$ 
recalling \eqref{eq:EstRemTilde}, we obtain
\be\label{dWetodW}
dW_{\bP_\e}[\bw] \to dW_{\bP}[w] \quad \text{as } \e\to 0.
\ee
Therefore in order to prove that, as in the smooth situation, $dW_{\bP}$ is the Willmore equation in conservative form, it is sufficient to show that 
\begin{eqnarray}
&&\int_{D^2} \left(< {\bH_{\e}}, \frac{1}{2} D^{*_{g_{\e}}}_{g_{\e}} \left[D_{g_{\e}} \bw -3 \pi_{\bn_{\e}} (D_{g_{\e}} \bw)\right]>\right) dvol_{g_{\e}} \nonumber \\
&&\qquad + \int_{D^2}  \left( <\star_h\left((*_{g_{\e}} D_{g_{\e}} \bn_\e) \wedge_M {\bH_{\e}} \right), D_{g_{\e}} \bw>+< -\ti{R}({\bH_{\e}})+R^\perp_{\bP_\e}(T{\bP_\e}), \bw> \right) dvol_{g_{\e}}. \nonumber \\
&& \to\int_{D^2} \left(< {\bH_{}}, \frac{1}{2} D^{*_{g_{}}}_{g_{}} \left[D_{g_{}} \bw -3 \pi_{\bn_{}} (D_{g_{}} \bw)\right]>\right) dvol_{g_{}} \nonumber \\
&&\qquad + \int_{D^2}  \left( <\star_h\left((*_{g_{}} D_{g_{}} \bn) \wedge_M {\bH_{}} \right), D_{g_{}} \bw>+< -\ti{R}({\bH_{}})+R^\perp_{\bP}(T{\bP}), \bw> \right) dvol_{g_{}}. \label{eq:dWetodW1}
\end{eqnarray}
We are going to check the convergence term by term. 

Observe that $D^{*_{g_{\e}}}_{g_{\e}}[D_{g_{\e}} \bw]= g^{ij}_\e D_{\p_{x^i}\bP_\e}D_{\p_{x^j}\bP_\e} \bw$ and, using the definitions, one computes
\begin{eqnarray}
\left(g^{ij}_\e D_{\p_{x^i}\bP_\e}D_{\p_{x^j}\bP_\e} \bw\right)^k&=&g^{ij}_\e \Big[ \p^2_{x^i x^j} w^k+ (\Gamma^k_{pq}\circ \bP_\e) \p_{x^j}w^p \p_{x^i}\Phi_\e^q+ (<grad_h \Gamma^k_{pq},\p_{x^i}\bP_\e>_h \circ \bP_\e) w^p \p_{x^j}\Phi^q \nonumber \\
&& \quad + (\Gamma^k_{pq}\circ \bP_\e) \p_{x^i}w^p \p_{x^j}\Phi^q+(\Gamma^k_{pq}\circ \bP_\e) w^p \p^2_{x^i x^j}\Phi^q\nonumber\\
&&\quad +(\Gamma^k_{lm}\circ \bP_\e)(\Gamma^l_{pq}\circ \bP_\e) w^p \p_{x^j}\Phi^q \p_{x^i}\Phi^m \Big] \nonumber\\
&=& f^k_{1,\e}+f^k_{2,\e} \quad \text{with } |f^k_{1,\e}|\leq F^k_1 \in L^{\infty}(D^2) \text{ and } |f^k_{2,\e}|\leq F^k_2 \in L^{2}(D^2);\label{Bounda)}
\end{eqnarray}
where $\Gamma^k_{pq}$ are the Christoffel symbols of $(M,h)$ which are smooth and $C^1$ bounded by the compactness of $M$. Notice that in the last equality we used \eqref{PhietoPhi} and  \eqref{PhiBounded}. Combining  \eqref{PhietoPhi}, \eqref{HetoH} and \eqref{Bounda)} we get therefore that $<\bH_\e,D^{*_{g_{\e}}}_{g_{\e}}[D_{g_{\e}} \bw]> dvol_{g_\e}$ is dominated in $L^1(D^2)$ and converges almost everywhere to $<\bH,D^{*_{g}}_{g}[D_{g} \bw]> dvol_{g}$, then by Dominated Convergence Theorem
\be\label{eq:Conva)}
\int_{D^2}<\bH_\e,D^{*_{g_{\e}}}_{g_{\e}}[D_{g_{\e}} \bw]> dvol_{g_\e} \to \int_{D^2}<\bH,D^{*_{g}}_{g}[D_{g} \bw]> dvol_{g} \quad \text{as } \e\to 0.
\ee 
Now let us consider the second summand in the first line of \eqref{eq:dWetodW1}. Observe that $D^{*_{g_{\e}}}_{g_{\e}}[\pi_{n_\e}(D_{g_{\e}} \bw)]= g^{ij}_\e D_{\p_{x^i}\bP_\e}[\pi_{\bn_\e}(D_{\p_{x^j}\bP_\e} \bw)]$; using \eqref{VI.65ab} we can write
\begin{eqnarray}
D^{*_{g_{\e}}}_{g_{\e}}[\pi_{n_\e}(D_{g_{\e}} \bw)]&=&(-1)^{m-1} g_\e^{ij} \Big[(D_{\p_{x^i}\bP_\e} \bn_\e) \res (\bn_\e \res  D_{\p_{x^j}\bP_\e} \bw)+ \bn_\e \res((D_{\p_{x^i}\bP_\e} \bn_\e)\res D_{\p_{x^j}\bP_\e} \bw) \nonumber \\
&&\quad \quad \quad \quad \quad \quad +\bn_\e \res (\bn_\e \res  (D_{\p_{x^i}\bP_\e} D_{\p_{x^j}\bP_\e}\bw)) \Big].
\end{eqnarray}
Writing explicitely the right hand side as done for \eqref{Bounda)}, one checks that
\be\label{Boundb)}
D^{*_{g_{\e}}}_{g_{\e}}[\pi_{n_\e}(D_{g_{\e}} \bw)]=\vec{f}_{3,\e}+\vec{f}_{4,\e} \quad \text{with } |\vec{f}_{3,\e}|\leq \bF_3 \in L^{\infty}(D^2) \text{ and } |\vec{f}_{4,\e}|\leq \bF_4 \in L^{2}(D^2).  
\ee 
Combining  \eqref{PhietoPhi}, \eqref{HetoH} and \eqref{Boundb)} we get therefore that $<\bH_\e,D^{*_{g_{\e}}}_{g_{\e}}[\pi_{n_\e}(D_{g_{\e}} \bw)]> dvol_{g_\e}$ is dominated in $L^1(D^2)$ and converges almost everywhere to $<\bH,D^{*_{g}}_{g}[\pi_{\bn}(D_{g} \bw)]> dvol_{g}$, then by Dominated Convergence Theorem
\be\label{eq:Convb)}
\int_{D^2}<\bH_\e,D^{*_{g_{\e}}}_{g_{\e}}[\pi_{n_\e}(D_{g_{\e}} \bw)]> dvol_{g_\e} \to \int_{D^2}<\bH,D^{*_{g}}_{g}[\pi_{\bn}(D_{g} \bw)]> dvol_{g} \quad \text{as } \e\to 0.
\ee 
Now let us consider the first summand of the second line of \eqref{eq:dWetodW1}. Observe that 
\be\label{*gcoord}
*_{g_\e} \frac{\p}{\p x^j}= \sqrt{\det g_\e} \; \varepsilon_{jp} \;g_\e^{pq} \,\frac{\p}{\p x^q},
\ee
where $\varepsilon_{jp}$ is null if $j=p$ and equals the signature of the permutation $(1,2)\mapsto (j,p)$ if $j\neq p$;
after some straightforward computations using the definitions \eqref{def:Dg}, \eqref{def:*g},\eqref{def:starh}, \eqref{def:wedgeM}, \eqref{def:ScalProd}, we get
\begin{eqnarray}\label{eq:2line1sum}
<\star_h\left((*_{g_{\e}} D_{g_{\e}} \bn_\e) \wedge_M {\bH_{\e}} \right), D_{g_{\e}} \bw>&=& \sqrt{\det g_\e}\, g_\e^{ij}\, \varepsilon_{jp} \,g_\e^{pq} \,<\star_h\left(D_{\p_{x^i}\bP_\e} \bn_\e) \wedge_M {\bH_{\e}} \right), D_{\p_{x^q}\bP_\e} \bw> \nonumber \\
&=& f_{5,\e} \text{ with } |f_{5,\e}|\leq F_5 \in L^1(D^2).
\end{eqnarray}
Using analogous arguments as before, by dominated convergence theorem we obtain
\be\label{eq:Convc)}
\int_{D^2}<\star_h\left((*_{g_{\e}} D_{g_{\e}} \bn_\e) \wedge_M {\bH_{\e}} \right), D_{g_{\e}} \bw> dvol_{g_\e}\to \int_{D^2}<\star_h\left((*_{g} D_{g} \bn) \wedge_M {\bH} \right), D_{g} \bw> dvol_{g}.
\ee
Finally consider the last two curvature terms in \eqref{eq:dWetodW1}. By the definition  \eqref{eq:defR}, the first one writes as
\be\label{eq:d)}
<\tilde{R} (\bH_\e),\bw>=-<\sum_{i=1}^2 \Riem^h(\bH_\e, \bbe^\ep_i)\bbe^\ep_i, \pi_{\bn_\e}(\bw)>
\ee
for an orthonormal frame $\bbe^\ep_i$ of $\bP_{\ep,*}(TD^2)$; observe it is dominated in $L^1(D^2)$ and converges a. e. to  $<\tilde{R}( \bH),\bw>$ on $D^2$, so as before
\be\label{eq:Convd)}
\int_{D^2}<\tilde{R} (\bH_\e),\bw> dvol_{g_\e}\to \int_{D^2}<\tilde{R} (\bH),\bw> dvol_{g}.
\ee
Finally, by definition \eqref{Rperp} and identity \eqref{*gcoord},  
\be\label{eq:e)}
R^\perp_{\bP_\e}(T\bP_\e):= \left[\pi_T\left(\Riem^h(\bbe_1,\bbe_2) \bH \right) \right]^\perp=\sqrt{\det g_\e}\, g_\e^{ij}\, g_\e^{kl}\, \e_{lp} \, g_\e^{pq} <\Riem^h(\p_{x^i}\bP_\e, \p_{x^j}\bP_\e)\bH_\e, \p_{x^k}\bP_\e> \p_{x^q}\bP_\e.
\ee
From this explicit formula, as before, one checks that $<R^\perp_{\bP_\e}(T\bP_\e), \bw> dvol_{g_\e}$ is dominated in $L^1(D^2)$ and converges   to $<R^\perp_{\bP}(T\bP), \bw> dvol_{g}$ a.e. on $D^2$, then by Dominated Convergence Theorem
\be\label{eq:Conve)}
\int_{D^2}<R^\perp_{\bP_\e}(T\bP_\e), \bw> dvol_{g_\e} \to \int_{D^2}<R^\perp_{\bP}(T\bP), \bw> dvol_{g}.
\ee 
Combining  \eqref{eq:Conva)}, \eqref{eq:Convb)},\eqref{eq:Convc)},\eqref{eq:Convd)} and  \eqref{eq:Conve)} we obtain \eqref{eq:dWetodW1} as desired.
Let us recap what we have just proved: if $\bP$ is a $W^{1,\infty}\cap W^{2,2}$ immersion of the disc $D^2$ into a coordinate neighbourood in $M$ and $\bw\in C^{\infty}_0(D^2,\R^m)$ is a smooth variation with compact support in $D^2$, then the differential of the Willmore functional $d_{\bP}W[\bw]$ coincides with the pairing between the Willmore equation in conservative form and $\bw$. Now by approximation the same is true for variations in $W^{1,\infty}\cap W^{2,2}(D^2,\R^m)$ with compact support in $D^2$. By partition of unity, the same statement holds for $\bP\in {\mathcal F}_{\Sp^2}$ with branched points $\{b^1,\ldots,b^N\}$ and any   variation $\bw \in W^{1,\infty}\cap W^{2,2}(D^2,T_{\bP} M)$  with compact support in $\Sp^2\setminus \{b^1,\ldots,b^N\}$.

The proof regarding the differentiability of $F$ is analogous since ${\mathbb I}_{\bP}$ is a vectorial function of $(\nabla^2 \bP, \nabla \bP, \bP)$ linear in $\nabla^2 \bP$. Moreover for smooth immersions and smooth variations, combining Corollary \ref{co:ConsFGen} and Lemma \ref{Lm:IntPart}, the first variation of $F$ is exactly \eqref{eq:dF}. With the same approximation argument carried for $W$ one checks that the same expression holds for a weak immersion. 

The proof regarding the differentiability and  the expression  of the  differential of the area functional  is easier since $dvol_{g}$ is function just of $(\nabla \bP,\bP)$, and  can be performed along the same lines once recalled that in the smooth case the differential of the area functional is exactly \eqref{eq:dA}. 
\end{proof}

Let us now prove the following approximation Lemma used in the proof of Lemma \ref{lem:dW}.
\begin{Lm}
\label{approx-lemma}
Let $\vec{\Phi}$ be a conformal weak immersion in ${\mathcal F}_{D^2}$ into ${\R}^m$ without branch points. Let $\varphi$ be a non negative compactly supported function 
of $C^\infty_0({\R})$ such that $\varphi$ is identically equal to 1 in a neighborhood of $0$ and
\[
2\pi \int_{\R}\varphi(t)\,t\ dt=1.
\]
Denote $\varphi_\ep(t):=\ep^{-2}\ \varphi(t/\ep)$. Denote for $\ep<1/4$ and for any $x\in D^2_{1/2}$, 
\[
\vec{\Phi}_\ep(x):=\varphi_\ep(|x|)\star\vec{\Phi}:=\int_{D^2}\varphi_{\ep}(|x-y|)\ \vec{\Phi}(y)\ dy\quad.
\]
There exists $0<\ep_{\vec{\Phi}}<1/4$ such that for any $\ep<\ep_{\vec{\Phi}}$ the  map $\vec{\Phi}_\ep$ realizes a smooth
immersion from $D^2_{1/2}$ into ${\R}^m$, moreover we have
\be
\label{A.1}
\lim_{\ep\rightarrow 0}\|g_{\vec{\Phi}_\ep}-g_{\vec{\Phi}}\|_{L^\infty(D^2_{1/2})}=0\quad,
\ee
we have also
\be
\label{A.2}
\lim_{\ep\rightarrow 0}\|\vec{n}_{\vec{\Phi}}-\vec{n}_{\vec{\Phi}_\ep}\|_{W^{1,2}(D^2_{1/2})}=0\quad,
\ee
and
\be
\label{A.3}
\lim_{\ep\rightarrow 0}\|\vec{H}_{\vec{\Phi}}-\vec{H}_{\vec{\Phi}_\ep}\|_{L^2(D^2_{1/2})}=0\quad.
\ee
\hfill $\Box$
\end{Lm}
Before to prove the lemma~\ref{approx-lemma} we establish the following $\varphi-$Poincar\'e inequality.
\begin{Lm}
\label{Poinc-ineg}
Let $u\in W^{1,2}(D^2)$. Let $\varphi$ be a non negative compactly supported function 
of $C^\infty_0({\R})$ such that $\varphi$ is identically equal to 1 in a neighborhood of $0$ and
\[
2\pi\int_{\R}\varphi(t)\,t\ dt=1.
\]
Denote $\varphi_\ep(t):=\ep^{-2}\ \varphi(t/\ep)$. For $\ep<1/4$ and $x\in D^2_{1/2}$ denote
\[
u_\ep(x):=\varphi_\ep(|x|)\star u:=\int_{D^2}\varphi_{\ep}(|x-y|)\ u(y)\ dy\quad.
\]
There exists a constant $C>0$ such that for any $x\in D^2_{1/2}$
\be
\label{A.4}
\frac{1}{|B_\ep(x)|}\ \int_{B_\ep(x)}|u(y)-u_\ep(x)|^2\ dy\le C\ \int_{B_\ep(x)}|\nabla u|^2(y)\ dy\quad.
\ee
\hfill $\Box$
\end {Lm}
{\bf Proof of lemma~\ref{Poinc-ineg}.}
For any $x\in D^2_{1/2}$ and $0<\ep<1/4$ we denote 
\[
\ov{u}^{\ep,x}:=\frac{1}{|B_\ep(x)|}\ \int_{B_\ep(x)}u(y)\ dy=\int_{D^2}\chi_\ep(|x-y|)\ u(y)\ dy\quad,
\]
where $\chi_\ep(t)\equiv (\pi\ep^2)^{-1}$  on $[0,\ep]$ and equals to zero otherwize. The classical Poincar\'e inequality
gives the existence of a universal constant such that
\be
\label{A.5}
\frac{1}{|B_\ep(x)|}\ \int_{B_\ep(x)}|u(y)-\ov{u}^{\ep,x}|^2\ dy\le C\ \int_{B_\ep(x)}|\nabla u|^2(y)\ dy\quad.
\ee
We have
\be
\label{A.6}
\frac{1}{|B_\ep(x)|}\ \int_{B_\ep(x)}|u(y)-u_\ep(x)|^2\ dy\le 2\frac{1}{|B_\ep(x)|}\ \int_{B_\ep(x)}|u(y)-\ov{u}^{\ep,x}|^2\ dy +2 |\ov{u}^{\ep,x}-u_\ep(x)|^2\quad.
\ee
We have
\be
\label{A.7}
 \ov{u}^{\ep,x}-u_\ep(x)=\int_{B_\ep(x)}\lf[\chi_\ep(|x-y|)-\varphi_\ep(|x-y|)\rg]\ u(y)\ dy\quad.
 \ee
Since
\[
\int_{B_\ep(x)}\lf[\chi_\ep(|x-y|)-\varphi_\ep(|x-y|)\rg]\ dy=0
\]
The identity (\ref{A.7}) takes the form
\be
\label{A.8}
\ov{u}^{\ep,x}-u_\ep(x)=\int_{B_\ep(x)}\lf[\chi_\ep(|x-y|)-\varphi_\ep(|x-y|)\rg]\ (u(y)-\ov{u}^{\ep,x})\ dy\quad.
\ee
Thus
\be
\label{A.9}
|\ov{u}^{\ep,x}-u_\ep(x)|^2\le C\ \ep^{-4}\lf|\int_{B_\ep(x)}|u(y)-\ov{u}^{\ep,x}|\ dy\rg|^2\le\, C\,\ep^{-2}\ \int_{B_\ep(x)}|u(y)-\ov{u}^{\ep,x}|^2\ dy\quad.
\ee
Combining (\ref{A.5}), (\ref{A.6}) and (\ref{A.9}) gives (\ref{A.4}) and this proves lemma~\ref{Poinc-ineg}.\hfill $\Box$

\medskip

\noindent{\bf Proof of lemma~\ref{approx-lemma}.}
We first establish (\ref{A.1}). Since $\vec{\Phi}$ is a weak conformal immersion, results from \cite{Hel} implies that there exists $\la\in C^0(D^2)$ such that
\[
g_{\vec{\Phi}}=e^{2\,\la}\ [dx_1^2+dx_2^2]\quad,
\]
and $e^\la=|\p_{x_1}\vec{\Phi}|=|\p_{x_2}\vec{\Phi}|$. Then, for any $\delta>0$ there exists $\ep$ such that
\be
\label{A.10}
\forall\,\delta>0\quad\exists\,\ep>0\quad\quad\forall\, x,y\in D^2_{3/4}\quad\quad |x-y|<\ep\quad\Longrightarrow\quad1-\delta<e^{\la(x)-\la(y)}\le 1+\delta\quad.
\ee
Since $\vec{\Phi}\in W^{2,2}(D_2,{\R}^m)$
\be
\label{A.11}
\forall\,\delta>0\quad\exists\,\ep>0\quad\forall\,\ep<\ep_0\quad\sup_{x\in D^2_{1/2}}\int_{B_\ep(x)}|\nabla^2\vec{\Phi}|^2(y)\ dy\le \delta^2
\ee
Applying lemma~\ref{Poinc-ineg} to $u=\nabla\vec{\Phi}$ we deduce then
\be
\label{A.12}
\forall\,\delta>0\quad\exists\,\ep_0>0\quad\forall\,\ep<\ep_0\quad\sup_{x\in D^2_{1/2}}\frac{1}{|B_\ep(x)|}\int_{B_\ep(x)}|\nabla\vec{\Phi}(y)-\nabla\vec{\Phi}_\ep(x)|^2\ dy\le\delta^2\quad.
\ee
Using the mean-value theorem we then deduce that
\be
\label{A.13}
\begin{array}{l}
\forall\,\delta>0\quad\exists\,\ep_0>0\quad\mbox{ s.t. }\quad\forall\,\ep<\ep_0\quad\forall\, x\in D^2_{1/2}\quad\exists\, y_x\in B_\ep(x)\quad\mbox{ s. t. }\\[5mm]
\ds\quad\quad|\nabla\vec{\Phi}(y_x)-\nabla\vec{\Phi}_\ep(x)|\le\sqrt\delta\quad.
\end{array}
\ee
Since 
\be
\label{A.13b}
0<\inf_{y\in D^2_{1/2}}\ |\nabla\vec{\Phi}(y)|^2=\inf_{y\in D^2_{1/2}}\ 2\ e^{2\,\la(y)}\le \sup_{y\in D^2_{1/2}}|\nabla\vec{\Phi}(y)|^2
\ee
then (\ref{A.13}) implies for $i=1,2$
\be
\label{A.14}
\begin{array}{l}
\forall\,\delta>0\quad\exists\,\ep_0>0\quad\forall\,\ep<\ep_0\quad\mbox{ s.t. }\quad\forall\, x\in D^2_{1/2}\quad\exists\, y_x\in B_\ep(x)\\[5mm]
\ds\quad\quad\mbox{ s. t. }\quad 1-\delta\le\frac{|\p_{x_i}\vec{\Phi}_\ep(x)|}{|\p_{x_i}\vec{\Phi}(y_x)|}\le
1+\delta\quad.
\end{array}
\ee
Combining (\ref{A.10}) and (\ref{A.14}) we obtain for $i=1,2$
\be
\label{A.15}
\forall\,\delta>0\quad\exists\,\ep_0>0\quad\forall\,\ep<\ep_0\quad\forall\, x\in D^2_{1/2}\quad\quad1-\delta\le\frac{|\p_{x_i}\vec{\Phi}_\ep(x)|}{|\p_{x_i}\vec{\Phi}(x)|}\le
1+\delta\quad.
\ee
Similarly, using (\ref{A.13}), (\ref{A.13b}) and the fact that $\p_{x_1}\vec{\Phi}(y)\cdot\p_{x_2}\vec{\Phi}(y)\equiv 0$ we have
\be
\label{A.16}
\forall\,\delta>0\quad\exists\,\ep_0>0\quad\forall\,\ep<\ep_0\quad\forall\, x\in D^2_{1/2}\quad\quad \frac{|\p_{x_1}\vec{\Phi}_\ep(x)\cdot\p_{x_2}\vec{\Phi}_\ep(x)|}{|\nabla\vec{\Phi}_\ep(x)|^2}\le
\delta\quad.
\ee
It is clear that (\ref{A.15}) and (\ref{A.16}) imply (\ref{A.1}). Finally (\ref{A.2}) and (\ref{A.3}) are direct consequences of the fact that (\ref{A.1}) and (\ref{A.13b}) hold together with the fact that
$\vec{\Phi}_\ep\rightarrow\vec{\Phi}$ strongly in $W^{2,2}(D^2_{1/2})$. Lemma~\ref{approx-lemma} is then proved.\hfill $\Box$

\begin{Lm}{\bf [Lower semi continuity under $W^{2,2}$-weak convergence]}\label{lem:LSC}
Let $\{\bP_k\}_{k\in \N}\subset {\mathcal F}_{\Sp^2}$ and $\bP_\infty$ be weak branched conformal immersions and assume that there exist $a^1,\ldots, a^N \in \Sp^2$ such that for every compact subset (with smooth boundary) $K\subset \subset \Sp^2$ we have
\begin{eqnarray}
\bP_k &\rightharpoonup& \bP \quad \text{weakly in } W^{2,2}(K) \label{ass:WC} \\
\sup_k  \sup_{x\in K}|\log |\nabla \bP_k|\,|(x) &\leq& C_K <\infty \quad \text{for some constant $c_K$ depending on $K$}  \label{ass:CF}.
\end{eqnarray} 
Then the Willmore and the Energy functional are  lower semicontinuous:
\be\label{eq:LSC}
\int_{K} |H_{\bP_{\infty}}|^2 dvol_{g_{\bP_\infty}} \leq \liminf_k \int_{K} |H_{\bP_{k}}|^2 dvol_{g_{\bP_k}}, \quad \int_{K} |{\mathbb I}_{\bP_{\infty}}|^2 dvol_{g_{\bP_\infty}} \leq \liminf_k \int_{K} |{\mathbb I}_{\bP_{k}}|^2 dvol_{g_{\bP_k}}.
\ee
\hfill $\Box$
\end{Lm}

\begin{proof}
Since $\bP_k$ are conformal, then $\bH_k=\frac{1}{2}e^{-2\lambda_k} \triangle \bP_k$ where $\lambda_k=\log|\p_{x^1}\bP_k|=\log|\p_{x^2}\bP_k|$ is the conformal factor. Let us first show that
\be\label{eq:D'Conv}
\bH_k \sqrt{vol_{g_k}}=\frac{1}{2|\p_{x^1} \bP_k|} \triangle \bP_k \to \frac{1}{2|\p_{x^1} \bP_\infty|} \triangle \bP_\infty  = \bH_\infty \sqrt{vol_{g_\infty}}  \quad \text{in } {\mathcal D}'(K).
\ee
From \eqref{ass:WC} and Rellich Kondrachov Theorem we have that $|\p_{x^1} \bP_k| \to |\p_{x^1} \bP_\infty|$ strongly in $L^p(K)$ for every $1<p<\infty$; moreover assumption \eqref{ass:CF} guarantees that $|\p_{x^1} \bP_k|\geq \frac{1}{C}>0$ independently of $k$. It follows that 
\be\nonumber
\frac{1}{|\p_{x^1} \bP_k|} \to \frac{1}{|\p_{x^1} \bP_\infty|} \quad \text{strongly in } L^p(K) \text{ for every } 1<p<\infty. 
\ee
Since, by assumption \eqref{ass:WC}, clearly $\triangle \bP_k \to \triangle \bP_\infty$ weakly in $L^2(K)$, then \eqref{eq:D'Conv} follows. 
In order to conclude observe that \eqref{ass:WC} implies that $\bP_k$ are uniformly bounded in $W^{2,2}(K)$, then assumption \eqref{eq:LSC} and the conformality of $\bP_k$ give that $\bH_k \sqrt{vol_{g_k}}$ are uniformly bounded in $L^2(K)$. This last fact together with \eqref{eq:D'Conv} implies that
$$\bH_k \sqrt{vol_{g_k}} \rightharpoonup \bH_\infty \sqrt{vol_{g_\infty}} \quad \text{weakly in } L^2(K).$$
The thesis then follows just by lower semicontinuity of the $L^2$ norm under weak convergence.
The proof of the lower semicontinuity of $\int |\mathbb I|^2$ is analogous once observed that in conformal coordinates $|{\mathbb I}|^2= e^{-4\lambda} |\nabla^2 \bP|^2$.
\end{proof}

\end{document}